\documentclass[bj]{imsart}

\RequirePackage[OT1]{fontenc}
\RequirePackage{amsthm,amsmath,natbib}
\usepackage{amsfonts}
\usepackage{arydshln} 
\usepackage{multirow}
\usepackage{graphicx} 
\usepackage{dsfont} 
\usepackage{natbib}
\usepackage{xcolor}
\usepackage{hyperref}
\usepackage{varioref}
\usepackage{cleveref} 

\crefname{prop}{Proposition}{Propositions}
\Crefname{prop}{Proposition}{Propositions}
\crefname{cor}{Corollary}{Corollaries}
\Crefname{cor}{Corollary}{Corollaries}
\crefname{lem}{Lemma}{Lemmas}
\Crefname{lem}{Lemma}{Lemmas}

\startlocaldefs
\numberwithin{equation}{section}
\theoremstyle{plain}

\newtheorem{prop}{Proposition}[section]
\newtheorem{cor}{Corollary}[section]
\newtheorem{lem}{Lemma}[section]
\theoremstyle{remark}

\endlocaldefs

\DeclareMathOperator{\Var}{\mathbf{Var}}
\DeclareMathOperator{\Tr}{\mathbf{Tr}}
\DeclareMathOperator{\Span}{\mathbf{span}}
\newcommand{\B}[1]{\mathds{#1}}
\newcommand{\C}[1]{\mathcal{#1}}
\newcommand{\ds}{\displaystyle}
\renewcommand{\mathbb}{\mathds}
\newcommand{\myeq}[1]{{\rm (\ref{#1})} on page \pageref{#1}}

\begin{document}

\begin{frontmatter}
\title{Robust dimension-free Gram operator estimates}
\runtitle{Robust dimension-free Gram operator estimates}

\begin{aug}

\author{\fnms{Ilaria} \snm{Giulini}\thanksref{t1}\ead[label=e1]{ilaria.giulini@me.com}}

\thankstext{t1}{The results presented in this paper 
were obtained while the author was preparing her PhD under the 
supervision of Olivier Catoni at the D\'epartement de Math\'ematiques et Applications, \'Ecole Normale Sup\'erieure, Paris, with the financial support of the 
R\'egion \^Ile de France.}

\address{{\sc INRIA} Saclay\\
\printead{e1}}

\runauthor{I. Giulini}

\end{aug}

\begin{abstract}
In this paper we investigate the question of estimating the Gram operator by a robust estimator
from an i.i.d. sample in a separable Hilbert space and we present uniform bounds that hold under weak moment assumptions.
The approach consists in first obtaining non-asymptotic dimension-free bounds in finite-dimensional spaces using some PAC-Bayesian 
inequalities related to Gaussian perturbations of the parameter and then in generalizing the results in a separable Hilbert space. 
We show both from a theoretical point of view and with the help of some simulations that such a robust estimator 
improves the behavior of the classical empirical one in the case of heavy tail data distributions. 
\end{abstract}



\end{frontmatter}

\tableofcontents 

\section{Introduction}

Many algorithms, such as spectral clustering, kernel principal component analysis or more generally kernel-based methods, 
are based on estimating eigenvalues and eigenvectors
of integral operators defined by a kernel function, from a given random sample.
To set the context from a statistical point of view, let $\mu \in \mathcal M_+^1(\mathcal X)$ be an unknown probability distribution on a compact space $\mathcal X$ 
and let $k$ be a kernel on $\mathcal X$. 
The goal is to estimate the integral operator 
\[
L_k f(x) = \int k(x,z) \ f(z) \ \mathrm d \mu(z) 
\]
from an i.i.d. random sample drawn according to $\mu.$ \\[1mm] 
A first study on the relationship between the spectral properties of a kernel matrix and the corresponding integral operator can be found in \cite{KolGin} for the case of a symmetric square integrable kernel $k.$ 
They prove that the ordered spectrum of the kernel matrix $K_{ij}=\frac{1}{n} k(X_i, X_j)$ converges to the ordered spectrum of the kernel integral operator
$L_k$.
Connections between this empirical matrix and its continuous counterpart have been subject of much research, 
for example 
in the framework of kernel-PCA 
\citep[see][]{StWCK02, StWCK05, ZBB} 
and spectral clustering \citep[see][]{vLBB}. 
In  \cite{RBDV}, the authors study the connection between the spectral properties of the empirical kernel matrix $ K_{ij}$ 
and those of the corresponding integral operator $L_k$ 
by introducing two extension operators on the (same) reproducing kernel Hilbert space defined by $k$, 
that have the same 
spectrum (and related eigenfunctions) as $K$ and $L_k$ respectively. 
In such a way they overcome the difficulty of dealing with objects ($ K$ and $L_k$) 
operating in different spaces.  \\[1mm]
The integral operator $L_k$ is related to the Gram operator 
\[
\mathcal G v = \int \langle v , \phi(z) \rangle_{\mathcal K} \ \phi(z) \ \mathrm d \mu(z), \quad v \in \mathcal K,
\]
where $\mathcal K$ denotes the reproducing kernel Hilbert space defined by the kernel $k$ and $\phi$ the corresponding feature map. 
\\[1mm]
The main objective of this paper is to estimate Gram operators on (infinite-dimensional) Hilbert spaces. 
Some bounds on the deviation of the empirical Gram operator from the true Gram operator in separable Hilbert spaces can be found in 
\cite{KoltLou1}
in the case of Gaussian random vectors.\\[1mm]
Let us introduce some notation. 
We denote 
by $\mathcal H$ a separable Hilbert space and 
by $\mathrm P \in  \mathcal M_+^1(\mathcal H)$ a (possibly unknown) probability distribution on $\mathcal H$. 
Remark that with the above notation $\mathrm P = \mu \circ \phi^{-1}.$ 
Our goal is to estimate the Gram operator $\mathcal G: \mathcal H \to \mathcal H$ defined as
\[
\mathcal G v = \int \langle v , z \rangle_{\mathcal H} \ z \ \mathrm d \mathrm P(z)
\]
from an i.i.d. sample drawn according to $\mathrm P.$ 
Our approach consists in first considering the finite-dimensional setting where $X$ is a random vector in $\mathbb R^d$
and then in generalizing the results to the infinite-dimensional case of separable Hilbert spaces. 
To be able to go from finite to infinite dimension we will establish dimension-free inequalities. 
To be more precise, we consider the related problem of estimating the quadratic form 
\[
N(\theta) = \langle \mathcal G \theta, \theta\rangle_{\mathcal H}, \qquad \theta \in \mathcal H
\]
which rewrites explicitly as
\[
N(\theta)= \int  \langle \theta , z \rangle_{\mathcal H}^2 \ \mathrm d \mathrm P(z).
\]
In the finite-dimensional setting we construct an estimator of the quadratic form $N(\theta)$ 
and we provide non-asymptotic dimension-free bounds for the approximation error that hold under weak moment assumptions. 
Note that similar techniques have also been used in \cite{Catoni16} in the finite-dimensional setting. 
However, the results presented here are not comparable to those 
in \cite{Catoni16} as the complexity terms appearing 
in the bounds are not the same. 
As a result, in \cite{Catoni16} the bounds on the approximation error 
depend explicitly on the dimension of the ambient space, and grow to $+ \infty$ 
with it.
Observe that in the finite-dimensional case the quadratic form $N(\theta)$ 
can be seen as the quadratic form associated with the Gram matrix
\[
G=\int xx^{\top} \ \mathrm d\mathrm P(x).
\]
Observe also that the study of the Gram matrix is of interest in the case of a non-centered criterion 
and that it coincides, in the case of centered data (i.e. $\mathbb E[X]=0$), with the study of the covariance matrix 
\[
\Sigma = \mathbb E\left[(X -  \mathbb E[X]) (X -  \mathbb E[X])^{\top}\right].
\]
Many theoretical results have been published on the estimation of covariance matrices, e.g. \cite{Rud}, \cite{RVer}, \cite{JTr}. 
These results follow from the study of random matrix theory and use as an estimator of $G$ 
the matrix obtained by 
replacing the unknown probability distribution $\mathrm P$ with the sample distribution $\frac{1}{n} \sum_{i=1}^n \delta_{X_i}.$
In \cite{Rud} the non-commutative Khintchine inequality is used to obtain bounds on the sample covariance matrix of a bounded random vector. 
Non-asymptotic results are obtained in  \cite{RVer} as a consequence of the analysis of random matrices with independent rows, while
in \cite{JTr} 
the author uses an extension of the Bernstein inequality to matrices.
However, such an empirical estimator becomes less efficient when the data have a long tail distribution. 
In \cite{SMin} the author presents a different estimator based on the geometrical median which is more robust than the classical empirical one. 
Another way to construct a robust estimator is to use rank-based coefficients (as the Kendall's tau) but this requires strong hypotheses on the distribution.
\\[1mm]
We first present a way to construct a robust estimator of the Gram matrix $G$ in finite dimension and then 
we extend the results to the infinite-dimensional case. 
Note that the bounds we propose are not formulated in terms of matrix norm, 
and provide instead bounds for $\bigl\lvert \theta^{\top} (\widehat{G} - G) \theta
\bigr\rvert$
depending on the direction $\theta$, whereas a bound on the operator 
norm, for example, would provide a single bound for $ \sup_{\theta \in \mathds{S}_1} \bigl\lvert \theta^{\top} (\widehat{G} 
- G ) \theta \bigr\rvert$.

\vskip5mm
\noindent
The paper is organized as follows. 
Section \ref{sec1} deals with the finite-dimensional case. 
We provide a new robust estimator of the Gram matrix $G$ 
and we use a PAC-Bayesian approach to obtain non-asymptotic dimension-independent bounds of its approximation error. 
In section \ref{sec2} we extend the results to the infinite-dimensional case, taking advantage of the fact that they are independent of the dimension of the ambient space. 
In section \ref{sec_emp} we propose some empirical results to show the performance of our estimator.
In Appendix \ref{emp_G} we compare from a theoretical point of view the behavior of our robust estimator to the one of the classical empirical estimator.
Finally in Appendix \ref{sec3} we extend the results to estimate the expectation of a symmetric random matrix
and we consider the problem of estimating the covariance matrix in the case 
when the expectation is unknown.

\section{The finite-dimensional setting}\label{sec1}

Let $\mathrm P\in \mathcal M_+^1(\mathbb R^d)$ be an unknown probability distribution on $\mathbb R^d$ 
and let $X\in \mathbb R^d$ be a random vector of law $\mathrm P.$
We denote by $\mathbb E$ the expectation with respect to $\mathrm P$. 
Our goal is to estimate 
the quadratic form 
\[
N(\theta) = \mathbb E [\langle \theta, X \rangle^2], \quad \theta \in \mathbb R^d,
\]
(that computes the energy in the direction $\theta$) 
from an i.i.d. sample $X_1, \dots, X_n \in \mathbb R^d$ drawn according to $\mathrm P.$
Observe that $N(\theta)$ can be seen as the quadratic form associated to the Gram matrix
\[
G=\mathbb E[XX^{\top}].
\]
Indeed to recover the Gram matrix $G$ from the above quadratic form it is sufficient to use the polarization identity
\begin{align*}
G_{ij} = e_i^{\top} G e_j & = \frac{1}{4} \left[ (e_i+e_j)^{\top} G (e_i+e_j) - (e_i-e_j)^{\top} G (e_i-e_j) \right]\\
& =  \frac{1}{4} \left[ N(e_i+e_j) - N(e_i-e_j) \right]
\end{align*}
where $\{e_i\}_{i=1}^d$ is the canonical basis of $\mathbb R^d.$

\vskip2mm
\noindent
A classical empirical estimator of the quadratic form $N(\theta)$ is 
\[
\bar N(\theta) = \frac{1}{n} \sum_{i=1}^n \langle \theta, X_i \rangle^2
\]
obtained by replacing the unknown probability distribution $\mathrm P$ with the sample distribution. 
However, as shown in \cite{Catoni12}, if the distribution of $\langle \theta, X \rangle^2$ has a heavy tail for some values of $\theta$, 
the quality of the approximation provided by the classical empirical estimator can be improved, 
using some $M$-estimator with a suitable smooth influence function and 
a scale parameter depending on the sample size. 
Thus in order to construct a robust estimator for $N$, we consider, for any $\theta \in \mathbb R^d$ and any $\lambda>0$, 
\begin{equation}\label{r_lambda}
r_{\lambda}(\theta)= \frac{1}{n} \sum_{i=1}^n \psi\left( \langle \theta, X_i \rangle^2 - \lambda \right),
\end{equation}
where the function $\psi: \mathbb R \to \mathbb R$, defined as
\begin{equation}\label{defnpsi}\psi(t)= \begin{cases} \log(2) & \text{if } t\geq 1, \\
-\log\left( 1-t+\frac{t^2}{2}\right) & \text{if } 0\leq t\leq 1,\\
-\psi(-t) & \text{if } t\leq 0,
 \end{cases}\end{equation}
 
\noindent 
is symmetric non-decreasing, bounded, and satisfies
\[
-\log\left( 1-t+\frac{t^2}{2}\right) \leq \psi (t) \leq \log\left( 1+t+\frac{t^2}{2}\right), \qquad t\in \mathbb R.
\]

\vskip 1mm
\noindent
Introduce
\begin{equation}\label{hatalpha}
 \widehat \alpha_\lambda(\theta) = \sup\{ \alpha\in\mathbb R_+ \ | \ r_{\lambda}(\alpha\theta)\leq0\}.
\end{equation}
In order to simplify notation, in the following we omit the dependence on $\lambda$ and we write
$\widehat \alpha$ instead of $ \widehat \alpha_\lambda$.\\
Observe that, since
the function $\alpha \mapsto r_{\lambda}(\alpha\theta)$ is continuous, 
 $r_{\lambda}(\widehat \alpha(\theta) \theta)=0$ as soon as $\widehat \alpha(\theta)<+\infty.$
Moreover, since the function $\psi$ is close to the identity in a neighborhood of the origin, 
\[
0 = r_{\lambda}(\widehat{\alpha}(\theta) \theta) \simeq  
\widehat{\alpha}(\theta)^2 \bar{N}(\theta) -\lambda
\]
and therefore it is natural to consider as an estimator of $N(\theta)$ 
a quantity related to $\lambda/ \widehat\alpha(\theta)^2$, for a suitable value of $\lambda.$ 
We consider the family of (robust) estimators
\begin{equation}\label{tildeN}
\widetilde N_{\lambda}(\theta) = \frac{\lambda}{ \widehat\alpha(\theta)^2}
\end{equation}
and we observe that, since $\widehat\alpha(\theta)$ is homogeneous of degree $-1$ in $\theta$, 
\[
\widetilde N_{\lambda}(\theta) = \|\theta\|^2 \widetilde N_{\lambda}\left({\theta}/{ \|\theta\|}\right).
\]

\vskip1mm
\noindent
In the following we will use a PAC-Bayesian approach linked to Gaussian perturbations of the parameter $\theta$ to 
first construct a confidence region for $N(\theta)$
and then define and study a robust estimator by choosing a suitable value 
$\widehat{\lambda}$ for the parameter $\lambda$. \\[1mm]
Given $\theta \in \mathbb R^d$, we consider the family of Gaussian perturbations $\pi_{\theta}\sim \mathcal N(\theta, \beta^{-1} \mathrm I)$ 
of mean $\theta$ and covariance matrix $\beta^{-1} \mathrm I$ where $\beta>0$ is a free parameter.\\
Let
$\Lambda \subset \bigl( \mathbb{R}_+ \setminus \{ 0 \} \bigr)^2$ be a finite set
of possible values of the couple of parameters $(\lambda, \beta)$
and $|\Lambda|$ its cardinality. 
Let us introduce
\begin{align}\label{skappa}
s_4 = \mathbb E \left[ \| X\|^4\right]^{1/4} \quad \text{and} \quad
\kappa  = \sup_{\substack{\theta \in \mathbb{R}^d \\ \mathbb E [ \langle \theta, X \rangle^2 ] > 0}} \frac{\displaystyle\mathbb E \bigl[ \langle \theta , X \rangle^4 \bigr]}{\mathbb E \bigl[ \langle \theta, X \rangle^2 \bigr]^2}
\end{align}
assuming that these two quantities are finite.
{\footnote {As it will be explained later, it is sufficient to know upper bounds for these quantities since the following results still hold true replacing $s_4$ and $\kappa$ by upper bounds.}} We will prove later on that 
\[ 
s_4^2 
\leq \kappa^{1/2} \B{E} \bigl( \lVert X \rVert^2 \bigr) = \kappa^{1/2} 
\Tr(G).
\]  
Note that in the case where the probability distribution $\mathrm P$ is Gaussian,
$\kappa=3$.\\[1mm]
For any $(\lambda, \beta) \in \Lambda$ and $\epsilon>0$, we put
\begin{equation}
\label{c_notat}
\begin{aligned}
\xi & =  \frac{\kappa \lambda}{2}, \\
\mu & =  \lambda(\kappa-1) + 
\frac{(2 + c) \kappa^{1/2}s_4^2}{\beta},\\ 
\gamma & =  \frac{\lambda}{2}(\kappa-1) + \frac{(2+c) \kappa^{1/2}s_4^2}{
\beta} +
\frac{(2 + 3c) s_4^4}{2 \beta^2 \lambda} + 
\frac{\log(\lvert \Lambda \rvert /  \epsilon  )}{n\lambda},\\
\delta & = \frac{\beta }{2n\lambda}, 
\end{aligned}
\end{equation}
where
\begin{equation}\label{defc}
c= \frac{15}{8 \log(2)(\sqrt{2}-1)} \exp \left( \frac{1 + 2 \sqrt{2}}{2} \right)\leq 44.3.
\end{equation}

\begin{prop}\label{prop0} With probability at least $1-2\epsilon$, for any $\theta \in \mathbb{R}^d$, any $(\lambda, \beta) \in \Lambda$, 
\[
\Phi_{\theta,-} \biggl( \frac{\lambda}{\widehat{\alpha}(\theta)^2} \biggr) \leq N(\theta) \leq  \Phi_{\theta,+}^{-1} \biggl( \frac{\lambda}{\widehat{\alpha}(\theta)^2}\biggr)
\]
where $\Phi_{\theta,-}$ and $\Phi_{\theta,+}$ are non-decreasing functions defined as
\begin{align*}
\Phi_{\theta,-}(t) & = t \left( 1 - \frac{\gamma + \delta \lambda \lVert \theta \rVert^2 / t}{1 + \mu - \gamma - \delta \lambda \lVert \theta \rVert^2 / t} \right) 
\mathds{1} \Bigg[ \xi - \mu + 2 \gamma + 2 \delta \lambda \lVert \theta \rVert^2 / t < 1 \Bigg]\\
\Phi_{\theta,+}(t) & =  t \, \left( 1 + \frac{\gamma + \delta \lambda \lVert \theta \rVert^2 / t}{1 - \mu - \gamma - 2 \delta \lambda \lVert\theta \rVert^2 / t }\right)^{-1}  
\mathds{1} \Bigg[ \xi+ \mu+\gamma + 2 \delta \lambda \lVert \theta\rVert^2 / t < 1 \Bigg],
\end{align*}
and where $\Phi_{\theta, +}^{-1} (u) = \sup \bigl\{ t \in \mathds{R}_+ \, : \, 
\Phi_{\theta, +}(t) \leq u \bigr\}$.  
\end{prop}

\vskip 2mm
\noindent
For the proof we refer to section \ref{proof_prop0}.\\[1mm]
Observe that since those functions depend on $\theta$ only through $\|\theta\|$, if $\theta$ is such that $\|\theta\|=1$ it is natural to omit the dependence of $\theta$ 
and write 
$\Phi_-$ and $\Phi_+$. 
In the following we will omit the dependence of $\theta$ of the functions defined in Proposition \ref{prop0}, 
so that we write $\Phi_-$ and $\Phi_+$ instead of $\Phi_{\theta,-}$ and $\Phi_{\theta,+}$.

\begin{prop}\label{prop1} Let $\sigma \in \mathbb{R}_+$ be any energy level. We consider the set 
\[
\Gamma = \left\{ (\lambda, \beta, t)\in\Lambda \times \mathbb{R}_+ \ | \  
\xi + \mu + \gamma + 2 \frac{\delta \lambda }{ \max \{t, \sigma \} }< 1 \right\}
\]
and the bound 
\begin{equation}\label{defBlb}
B_{\lambda, \beta}(t) = \begin{cases} \displaystyle\frac{\gamma + \lambda \delta / \max \{ t , \sigma \} }{ 1 - \mu - \gamma - 2 \lambda \delta / \max \{ t,  \sigma \}} & (\lambda, \beta, t) \in \Gamma \\
+ \infty & \text{ otherwise.}
\end{cases}
\end{equation}
With probability at least $1 - 2 \epsilon$, 
for any $\theta \in \mathbb{R}^d$, any $(\lambda, \beta) \in \Lambda$,
\[ 
\Biggl\lvert \,  \frac{\max \bigl\{ N(\theta), \sigma \lVert \theta \rVert^2 \bigr\}}{\max \bigl\{ \widetilde N_{\lambda}(\theta), \sigma \lVert \theta \rVert^2 \bigr\}} - 1 \Bigg\rvert 
\leq  B_{\lambda, \beta} \left[  \lVert \theta \rVert^{-2} \widetilde N_{\lambda}(\theta)\right].
\] 
\end{prop}

\begin{proof}
We observe that, for any $z,t, \sigma \in \mathbb R_+$, 
if $\Phi_+(z)  \leq t$ then 
\[
\Phi_+ \bigl( \max \{ z, \sigma \} \bigr) \leq \max \{ t, \sigma \}
\]
since it is clear from the definition of $\Phi_+$ 
that $\Phi_+(\sigma) \leq \sigma$. Similarly if  $\Phi_-(z)  \leq t$, then 
\[
\Phi_- \bigl( \max \{ z, \sigma \} \bigr) \leq \max \{ t, \sigma \}.
\] 
Thus, according to the definition of $B_{\lambda, \beta}$ in equation \eqref{defBlb}, we get
\begin{align}
\label{phi+eq}
\Phi_+\bigl( \max \{ z, \sigma \} \bigr) & = \max \{ z, \sigma \} \Big( 1 + B_{\lambda, \beta} (z) \Big)^{-1} \\   
\label{phi-eq}
\Phi_-\bigl( \max \{ z, \sigma \} \bigr) & \geq \max \{ z, \sigma \} \Big( 1 - B_{\lambda, \beta}(z) \Big).
\end{align}
Therefore
\begin{align*}
\Phi_+^{-1}\bigl( \max \{ z, \sigma \} \bigr)  
& = \sup \bigl\{ t \, : \, \Phi_+(t) \leq \max \{ z, \sigma \} \bigr\} \\
& = \sup \bigl\{ t \geq z \, : \, \Phi_{+} (t) \leq \max \{ z, \sigma \} 
\bigr\} \\ 
& = \sup \bigl\{ t \geq z \, : \, \bigl( 1 + B_{\lambda, \beta}(t) \bigr)^{-1} 
\max \{ z, \sigma \} \bigr\} \\ 
& \leq \sup \bigl\{ t \geq z \, : \, \bigl( 1 + B_{\lambda, \beta}(z) 
\bigr)^{-1} t \leq \max \{ z, \sigma \} \bigr\} \\ 
& = \max \{ z, \sigma \} \Bigl( 1 + B_{\lambda, \beta}(z) \Bigr), 
\end{align*}
where we have have used the fact that $B_{\lambda, \beta}$ is 
non-decreasing. In view of these inequalities, the present proposition 
is a consequence of the previous one. 
\end{proof}

\vskip 2mm
\noindent
From now on we fix the finite set $\Lambda$ of all possible values of the couple $(\lambda, \beta)$ 
as
\begin{equation}\label{defLambda}
\Lambda = \bigl\{ (\lambda_j, \beta_j) \ | \ 0 \leq j < K \bigr\}, 
\end{equation}
where $\displaystyle K  = 1 + \left\lceil a^{-1} \log \biggl( \frac{n}{72(2+c) \kappa^{1/2}} \biggr) \right\rceil$, with $a>0$ and $c$ defined in equation \eqref{defc}, and 
\begin{align*}
\lambda_j & = \sqrt{ \frac{2}{n (\kappa - 1)} \Bigg(\frac{(2 + 3c)}{4(2 + c) \kappa^{1/2} \exp (-j a) }  + \log(K / \epsilon) \Bigg)} \\ 
\beta_j & = \sqrt{2 (2 + c) \kappa^{1/2} s_4^4 n \exp \bigl[  - (j-1/2) a \bigr]}.
\end{align*}
We put
\[ 
(\widehat{\lambda}, \widehat{\beta}) = \arg \min_{(\lambda, \beta)  \in \Lambda} B_{\lambda, \beta} \Bigg[ \lVert \theta \rVert^{-2} \widetilde{N}_{\lambda} (\theta) \Bigg]
\] 
and we define our robust estimator as
\begin{equation}\label{defhatn}
\widehat{N} (\theta) = \widetilde{N}_{\widehat{\lambda}}(\theta).
\end{equation}

\begin{prop} \label{prop2}
Let us fix a threshold $\sigma \leq s_4^2$ and set the value of the 
parameter $a$ used to define $\Lambda$ to $1/2$. 
Introduce 
\[
\zeta_* (t) = \sqrt{ 2.032 (\kappa-1) \Bigg( \frac{0.73 \ \mathbf{Tr}(G)}{t} 
+ \log(K) + \log(\epsilon^{-1}) \Bigg)} + \sqrt{ \frac{98.5 \, \kappa 
\mathbf{Tr}(G)}{ t}}, \quad t \in \mathbb R_+,
\]
where $\mathbf{Tr}(G) = \mathbb E\left[\|X\|^2 \right]$ denotes the trace of the Gram matrix, and
\[ 
B_*(t) = \begin{cases} 
\displaystyle\frac{n^{-1/2} \zeta_*(\max \{ t, \sigma \} )}{1 - 4 \, n^{-1/2} \zeta_*( \max \{ t, \sigma \} )} &  \bigl[ 6 + (\kappa-1)^{-1} \bigr] \zeta_*( \max \{ t, \sigma \} ) \leq \sqrt{n} \\ 
+ \infty & \text{ otherwise.}
\end{cases} 
\] 
With probability at least $1 - 2 \epsilon$, 
for any $\theta \in \mathbb{R}^d$, 
\[ 
\left\lvert \, \frac{\max \{ N(\theta), \sigma \lVert \theta \rVert^2 \}}{\max \{ \widehat{N}(\theta), \sigma \lVert \theta \rVert^2 \}} - 1 \, \right\rvert \leq B_* \Bigg[ \lVert \theta \rVert^{-2} N(\theta) \Bigg].
\] 
\end{prop}

\vskip 2mm
\noindent
For the proof we refer to section \ref{proof_prop2}.
Remark that equation \eqref{eq:precise} of the proof provides a bound for any choice of the parameter $a > 0$ 
and that we report here only the numerical value of the bound when 
$a = 1/2$ for the sake of simplicity.
\\[1mm]
Assuming any reasonable bound on the sample size we can bound the logarithmic factor $\log \log(n)$ hidden in $\log(K)$ with a relatively small constant. 
In particular, if we choose $n \leq 10^{20}$, we get $\log(K)\leq 4.35$.\\[1mm]
In order to provide a more friendly version of the above bound, we introduce here the $\mathcal O$ notation, 
where $A = \mathcal O (B)$ means that there exists a numerical constant $\tau$ such that $A\leq \tau B$. 
Thus, the result stated in Proposition \ref{prop2} becomes, with probability at least $1 - 2 \epsilon$, 
for any $\theta \in \mathbb{R}^d$, 
\[
\left\lvert \, \frac{\max \{ N(\theta), \sigma \lVert \theta \rVert^2 \}}{\max \{ \widehat{N}(\theta), \sigma \lVert \theta \rVert^2 \}} - 1 \, \right\rvert 
\leq \mathcal O \left( \sqrt{\frac{\kappa}{n} \left( \frac{\mathbf{Tr}(G)}{ \max\{ \|\theta\|^{-2} N(\theta), \sigma\}} + \log \bigl( 
\log(n) / \epsilon \bigr) \right)} \; \right).
\]

\vskip1mm
\noindent
Observe that the bound $B_*$, 
unlike the bound provided in \cite{Catoni16}, 
does not depend explicitly on the dimension $d$ 
of the ambient space.  
More specifically, the dimension has been replaced by the effective dimension coming from the entropy term of the PAC-Bayesian bound
\[\mathbf{Tr}(G)/ \max\big\{ \|\theta\|^{-2} N(\theta), \sigma \big\}.
\] 
To get an intuition of why this entropy term replaces the dimension, it is sufficient to consider the case where
the energy $N$ is uniformly distributed, so that $N(\theta) = N_*\geq \sigma$ for any $\theta \in \mathds R^d$ with $\|\theta\|=1$. In this case indeed
\[
\frac{\mathbf{Tr}(G)}{ \max\{ \|\theta\|^{-2} N(\theta), \sigma\}}
= \frac{\sum_{i=1}^d N(v_i)}{N_*} = d
\]
where $v_i$ denotes an orthonormal basis of eigenvectors of $G$. In summary, 
when the Gram matrix $G$ to be estimated is proportional to the identity 
matrix, then our bound coincides (up to some moderate increase in the 
constants) with the bound proved in \cite{Catoni16}, but when the eigenvalues 
of $G$ are decreasing, then our bound balances the complexity 
in a different way and is looser when $\theta$ is in the span of 
eigenvectors with low eigenvalues, but tighter when $\theta$ 
is in the span of eigenvectors with high eigenvalues.\\[1mm]
Let us also remark that 
the variance $\Var \bigl( \langle \theta, X \rangle^2 \bigr)$ of $\langle \theta, X \rangle^2$ is related to $\kappa$ by the relation 
\[ 
\kappa = 1 + \sup \biggl\{ \frac{\Var \bigl( \langle \theta, X \rangle^2 
\bigr)}{\mathds{E} \bigl( \langle \theta, X \rangle^2 \bigr)^2} \, : \, 
\theta \in \mathds{R}^d, \mathds{E} \bigl( \langle \theta, X \rangle^2 \bigr) 
> 0 \biggr\}. 
\] 
Moreover, we do not need to know the exact values of $\kappa$ and $\mathbf{Tr}(G) = \mathds{E} \bigl( \lVert X \rVert^2 \bigr)$ 
to compute the estimator and evaluate the bound, 
it is sufficient to know upper bounds instead. 
Indeed, if we use those upper bounds to define our estimator, the above result is still true with $\kappa$ and $\mathbf{Tr}(G)$ replaced by their upper bounds. \\[1mm]
We also observe that in order to have a meaningful (finite) bound
we can choose the threshold $\sigma$ such that
\begin{equation}
\label{eq:13}
8 \zeta_*(\sigma) \leq \sqrt n
\end{equation}
so that $B_*(t) <+\infty$ for any $t \in \mathbb R_+$, assuming that $\kappa \geq 3/2$.  
More precisely, using the inequality $(\sqrt{a} + \sqrt{b} )^2 \leq 
2(a+b)$, we see that equation \eqref{eq:13} holds when  
\[
\sigma = \frac{ 100 \, \kappa \mathbf{Tr}(G)}{n/128 - 4.35 -  \log \bigl( \epsilon^{-1} \bigr)}.
\]
With this choice the threshold $\sigma$ decays to zero at speed $1/n$ 
as the sample size grows to infinity. 

\vskip5mm
\noindent
Remark that the estimator $\hat N$ is not necessarily a quadratic form. 
We conclude this section by introducing and studying a quadratic estimator of $N,$ 
that is an estimator of the form 
$\theta^{\top} Q \theta$, where $Q$ is an estimate 
of the Gram matrix $G$.
\\[1mm]
We observe that Proposition \ref{prop0} provides a confidence region for $N(\theta)$.
Define
\[
B_-(\theta) =  \max_{(\lambda,\beta) \in \Lambda}  \Phi_- \bigl(\widetilde N_{\lambda}(\theta) \bigr) \quad 
\text{and} \quad 
B_+(\theta) = \min_{(\lambda,\beta) \in \Lambda} \Phi_+^{-1} \bigl( \widetilde N_{\lambda}(\theta) \bigr)
\]
where we recall that $\widetilde N_{\lambda}(\theta) = \frac{\displaystyle\lambda}{\displaystyle\widehat{\alpha}(\theta)^2}$ and $\Lambda$ is defined in equation \eqref{defLambda}.
According to Proposition \ref{prop0}, with probability at least $1-2\epsilon$, for any $\theta \in \mathbb R^d$, 
\begin{equation}\label{opt_ci}
B_-(\theta) \leq N(\theta) \leq B_+(\theta) .
\end{equation}
From a theoretical point of view we can consider as an estimator of $N$ any quadratic form belonging to 
the confidence interval $\left[ B_-(\theta), B_+(\theta)\right]$ 
for any $\theta$. Such a quadratic form exists with probability at least 
$1 - 2 \epsilon$ according to equation \eqref{opt_ci}. 
However, from an algorithmic point of view, 
we would like to impose these constraints only 
for a finite number of directions $\theta$. 
In particular, in the following
we are going to study the properties of a symmetric matrix $Q$ that satisfies
$\mathbf{Tr}(Q^2) \leq \mathbf{Tr}(G^2)$ and
\[
B_-(\theta)  \leq \theta^{\top}Q\theta \leq B_+(\theta), \quad \theta \in \Theta_{\delta}, 
\]
where $\Theta_{\delta}$ is any finite $\delta$-net of the unit sphere $\mathbb{S}_d 
= \bigl\{ \theta \in \mathbb{R}^d, \lVert \theta \rVert = 1 \bigr\} $, 
meaning that 
\[
\sup_{\theta \in \mathbb{S}_d} \min_{\xi \in \Theta_{\delta}} 
\lVert \theta - \xi \rVert \leq \delta. 
\]
The matrix $Q$ 
can be computed using a convex optimization algorithm as described in (section 1.2.4 of)
\cite{Giulini2015}.\\[1mm]
Note that a more straightforward choice would have been 
to set 
\begin{equation}
\label{eq:2.15}
Q_{i,j} = \frac{1}{4} \left[ \widehat N( e_i + e_j) - \widehat N(e_i- e_j) \right]
\end{equation}
where $\{e_i\}_{i=1}^d$ denotes the canonical basis of $\mathds{R}^d$. 
Unfortunately this simple choice is not adequate to  
control the approximation error independently of the dimension $d$ in all directions. 
To get a dimension-free bound we need an estimator that behaves well in a far larger set of directions $\Theta_\delta$ than the $d^2$ directions $e_i\pm e_j$.
\\[1mm]
From now on let $\sigma \in ]0,s_4^2]$ be a threshold such that $8\zeta_*(\sigma)\leq \sqrt n$.
The next proposition provides the analogous for the quadratic form 
$ \theta^{\top} Q \, \theta$ of the dimension-free bound 
presented in Proposition \ref{prop2} for $\widehat N(\theta)$. 

\begin{prop}\label{prop3}
With the same notation as in Proposition \ref{prop2}, with probability at least $1-2\epsilon,$ for any $\theta \in \mathbb S_d$,
\begin{align*}
\Bigl\lvert \max \{ \theta^{\top} Q \, \theta, \sigma \}  - 
\max \{ N(\theta), \sigma \}  \Bigr\rvert 
& \leq 2 \max \bigl\{ N(\theta), \sigma \bigr\} B_* \bigl( 
N(\theta)\bigr) + 5 \delta \sqrt{\mathbf{Tr}(G^2)}, \\
\Bigl\lvert \max \{ \theta^{\top} Q \, \theta, \sigma \} - 
\max \{ N(\theta), \sigma \}  \Bigr\rvert 
& \leq 2 \max \bigl\{ \theta^{\top} Q \, \theta, \sigma \bigr\} B_* \bigl( \min \{ 
\theta^{\top} Q \, \theta, s_4^2 \bigr\} \bigr) + 5 \delta \sqrt{\mathbf{Tr}(G^2)}.
\end{align*}
\end{prop}
Remark that the parameter $\delta$ of the net $\Theta_{\delta}$ governs 
the computation cost of $Q$. Thus, we can in theory (that is if we 
have an arbitrarily fast computer at our disposal), take $\delta$ 
as close to zero as we want. 
\begin{proof}
Since for any $\theta \in  \mathbb S_d$ there is $\xi \in \Theta_{\delta}$ such that $\|\theta- \xi \|\leq \delta,$ we have
\begin{multline}\label{bxitheta}
\bigl\lvert \theta^{\top} Q \, \theta - \xi^{\top} Q \, \xi 
\bigr\rvert =  \bigl( \theta + \xi \bigr)^{\top} 
Q \, \bigl( \theta - \xi \bigr) \\ \leq \lVert \theta + \xi \rVert \, 
\lVert Q \rVert_{\infty} \, \lVert \theta - 
\xi \rVert \leq 2 \delta \sqrt{ \mathbf{Tr}(Q^2) } 
\leq 2 \delta \sqrt{ \mathbf{Tr}(G^2)}. 
\end{multline}
Let us put $\eta =  2 \delta \sqrt{ \mathbf{Tr}(G^2)}.$
We observe that, with probability at least $1-2\epsilon,$
\begin{equation}\label{pf1}
\begin{aligned}
\Phi_- \circ \Phi_+ \bigl( \theta^{\top} Q \, \theta - \eta \bigr) 
& \leq N(\theta) + \eta, \\ 
\Phi_- \circ \Phi_+ \bigl( N( \theta) - \eta \bigr) & \leq 
\theta^{\top} Q \, \theta + \eta,
\end{aligned}
\end{equation}
where $\Phi_+$ and $\Phi_-$ are defined in Proposition \ref{prop0} and depend on $\theta$ only through $\|\theta\|$.
Indeed, in the event of probability at least $1-2\epsilon$ described in equation \eqref{opt_ci},
\[
\theta^{\top} Q \, \theta \leq \Phi_+^{-1} \bigl( \widetilde{N}_{\lambda}
(\xi) \bigr) + \eta \leq \Phi_+^{-1} \circ \Phi_-^{-1} \bigl( 
N(\xi) \bigr) + \eta
 \leq \Phi_+^{-1} \circ \Phi_-^{-1} \bigl( N(\theta) + \eta \bigr) + \eta,
\]
since equation \eqref{bxitheta} also holds for $N$, and in the same way we get
\[ 
\theta^{\top} Q \, \theta \geq \Phi_- \bigl( \widetilde{N}_{\lambda} (\xi) \bigr) 
- \eta \geq \Phi_- \circ \Phi_+ \bigl( N(\xi) \bigr) - \eta 
\geq \Phi_- \circ \Phi_+ \bigl( N(\theta) - \eta \bigr) - \eta
\] 
which proves equation \eqref{pf1}. 
We conclude the proof 
using Corollary \ref{lem1A} in section \ref{appx} 
\end{proof}

\noindent
Note that the estimated matrix $Q$ of the previous proposition is 
not necessarily positive semi-definite. We can remedy that 
shortcoming by considering instead its positive part $Q_+$
(obtained by taking the positive part of its eigenvalues in 
the framework of functional calculus on symmetric matrices). 

\vskip2mm
\noindent
Based on the fact that $\theta^{\top} Q \theta \geq B_-(\theta) \geq 0$ 
on $\Theta_{\delta}$ and on equation \eqref{bxitheta}, stating that 
for any $\theta \in  \mathbb S_d$, 
\[
\bigl\lvert \theta^{\top} Q \, \theta - \xi^{\top} Q \, \xi \bigr\rvert 
\leq 2 \delta \sqrt{ \mathbf{Tr}(G^2)}
\]
where $\xi\in \Theta_{\delta}$ is such that $\|\theta- \xi \|\leq \delta$,
we can see that we do not loose much when replacing $Q$ by $Q_+$.  
\begin{prop}\label{prop4}
With probability at least $1 - 2 \epsilon$, 
for any $\theta \in \mathbb{S}_d$, 
\begin{align*}
\Bigl\lvert \max \{ \theta^{\top} Q_+ \theta, \sigma \}  - 
\max \{ N(\theta), \sigma \}  \Bigr\rvert 
& \leq 2 \max \bigl\{ N(\theta), \sigma \bigr\} B_* \bigl( 
N(\theta)\bigr) + 7 \delta \sqrt{\mathbf{Tr}(G^2)}, \\
\Bigl\lvert \max \{ \theta^{\top} Q_+ \theta, \sigma \} - 
\max \{ N(\theta), \sigma \}  \Bigr\rvert 
& \leq 2 \max \bigl\{ \theta^{\top} Q_+ \theta, \sigma \bigr\} B_* \bigl( \min \{ 
\theta^{\top} Q_+ \theta, s_4^2 \bigr\} \bigr) \\
&\hskip45mm + 7 \delta \sqrt{\mathbf{Tr}(G^2)}, 
\end{align*}
where $B_*$ is defined in Proposition \ref{prop2}.
\end{prop}

\begin{proof}
Let us put as before $\eta =  2 \delta \sqrt{ \mathbf{Tr}(G^2)}.$ For any 
$\theta \in \mathbb{S}_d$, there is $\xi \in \Theta_{\delta}$ such that 
$\lVert \theta - \xi \rVert \leq \delta$, so that, according to equation \eqref{bxitheta}, 
\[ 
\theta^{\top} Q \theta \geq \xi^{\top} Q \xi - \eta 
\geq - \eta.
\] 
Then we deduce that 
\[ 
\lVert Q_- \rVert_{\infty} = \sup_{\theta \in \mathbb{S}_d} \theta^{\top} 
Q_- \theta = - \inf_{\theta \in \mathbb{S}_d} \theta^{\top} Q \theta \leq \eta.
\] 
Therefore, for any $\theta \in \mathbb{S}_d$, 
\[ 
\Bigl\lvert \max \{ \theta^{\top} Q \theta, \sigma \}  - 
\max \{ \theta^{\top} Q_+ \theta , \sigma \} \Bigr\rvert 
\leq 
\Bigg\lvert \theta^{\top} Q \theta   - 
\theta^{\top} Q_+ \theta \Bigg\rvert = \theta^{\top} Q_- \theta \leq \eta. 
\] 
Combining the above equation with Proposition \ref{prop3} we conclude the proof.
\end{proof}
\noindent Since for any $a, b \in \mathds{R}_+$, $a - b - \sigma \leq \max \{a, \sigma\} - \max\{b, \sigma\}$, 
we obtain as a consequence
\begin{cor}
With probability at least $1 - 2 \epsilon$, 
for any $\theta \in \mathbb{S}_d$, 
\begin{align*}
\Bigl\lvert \theta^{\top} Q_+ \theta  - 
 N(\theta) \Bigr\rvert 
& \leq 2 \max \bigl\{ N(\theta), \sigma \bigr\} B_* \bigl( 
N(\theta)\bigr) + 7 \delta \sqrt{\mathbf{Tr}(G^2)} +\sigma, \\
\Bigl\lvert \theta^{\top} Q_+ \theta - 
 N(\theta) \Bigr\rvert 
& \leq 2 \max \bigl\{ \theta^{\top} Q_+ \theta, \sigma \bigr\} B_* \bigl( \min \{ 
\theta^{\top} Q_+ \theta, s_4^2 \bigr\} \bigr) \\
&\hskip45mm + 7 \delta \sqrt{\mathbf{Tr}(G^2)}+ \sigma.
\end{align*}
\end{cor}

\vskip2mm
\noindent
To conclude we mention that it is possible to get similar dimension-free bounds under light tail hypotheses for the  
classical empirical estimator 
\[
\bar G = \frac{1}{n} \sum_{i=1}^n X_i X_i^{\top}.
\]  
For more details we refer to section \ref{emp_G}.

\section{The infinite-dimensional setting}\label{sec2}
In this section we extend the results obtained in the previous section to the infinite-dimensional setting. \\[1mm]
Let $\mathcal H$ be a separable Hilbert space and let $\mathrm P\in \mathcal M_+^1(\mathcal H)$ be an unknown probability distribution on $\mathcal H.$
We consider the Gram operator $\mathcal G: \mathcal H \to \mathcal H$ defined by 
\[
\mathcal G \theta = \int \langle \theta, v \rangle_{\mathcal H} v \ \mathrm d \mathrm P(v)
\]
and we assume $\mathbf{Tr}(\mathcal G) = \mathbb{E} \bigl( \lVert X \rVert_{\mathcal H}^2 \bigr) < +\infty,$ where $X\in \mathcal H$ denotes a random vector with law $\mathrm P.$
In analogy to the previous section we denote by $N$ the quadratic form associated with 
the Gram operator 
$\mathcal G$ so that
\[
N(\theta) =\langle \mathcal G \theta, \theta \rangle_{\mathcal H}= \int \langle \theta, v \rangle_{\mathcal H}^2\ \mathrm d \mathrm P(v), \quad \theta \in \mathcal H.
\]
We consider $(\mathcal H_k)_k$ an increasing sequence of subspaces of $\mathcal H$ of finite dimensions such that  $\overline{\cup_k \mathcal H_k }= \mathcal H$
. For instance if $ \{ e_i, i \in \B{N}^* \}$ is an orthonormal 
basis of $\C{H}$, we can take 
$ \C{H}_k = \Span \{ e_1, \dots, e_k \}$. To give a more concrete example, 
in the case when $\C{H} = \B{L}^2([0,1])$, we can take for $e_i$ the 
Fourier basis. 
In this example, the orthogonal projector $\Pi_k$ on $\C{H}_k$ 
is given by 
\[ 
\Pi_k v = \sum_{i=1}^k \langle v, e_i \rangle_{\C{H}} e_i, \qquad v \in \C{H}.
\] 
We denote by $N_k$ the quadratic form in $\mathcal H_k$ associated with 
the probability distribution of $\Pi_k X$, so that 
\[
N_k(\theta) = \B{E} \bigl( \langle \theta, \Pi_k X \rangle^2 \bigr) 
= N(\theta), 
\qquad \theta \in \C{H}_k. 
\]
Remark that for any $\theta \in \C{H}$, 
\[
N_k\left(\Pi_k\theta\right) = N \left(\Pi_k\theta\right) .
\]
In the following we consider an i.i.d. sample of size $n$ in $\mathcal H$ drawn according to $\mathrm P.$
According to Proposition \ref{prop0}, the event 
\[ 
\mathcal{A}_k = \biggl\{ \forall {\theta} \in \mathcal{H}_k, \ \forall (\lambda, 
\beta) \in \Lambda, \; \Phi_{\theta, -} 
\biggl( \frac{\lambda}{\widehat{\alpha}(\theta)^2} \biggr) \leq 
N(\theta) \leq \Phi^{-1}_{\theta, +} \biggl( 
\frac{\lambda}{\widehat{\alpha}(\theta)^2} \biggr)  \biggr\}
\]  
is such that $\mathrm{P}^{\otimes n} \bigl( \mathcal{A}_k \bigr) \geq 1 - 2 \epsilon$. 
Since $\mathcal{A}_{k+1} \subset \mathcal{A}_k$, by the continuity of measure, 
\[ \mathrm{P}^{\otimes n} \biggl( \bigcap_{k \in \mathbb{N}} \mathcal{A}_k \biggr) \geq 
1 - 2 \epsilon.\]
This means that with probability at least $1 - 2 \epsilon$, for any 
$\theta \in \bigcup_{k} \mathcal{H}_k$ and any $(\lambda, \beta) \in \Lambda$, 
\[
\Phi_{\theta, -} 
\biggl( \frac{\lambda}{\widehat{\alpha}(\theta)^2} \biggr) \leq 
N(\theta) \leq \Phi^{-1}_{\theta, +} \biggl( 
\frac{\lambda}{\widehat{\alpha}(\theta)^2} \biggr). 
\]
Consequently, since
$\displaystyle N(\theta) = \lim_{k \rightarrow + \infty} N \bigl( \Pi_k(\theta) \bigr)$,  for any $\theta \in \mathcal{H}$, the following result holds. 
\begin{prop}
\label{prop_Hspace} 
With probability at least $1 - 2 \epsilon$, for any $\theta \in \mathcal{H}$, 
\[B_-(\theta) \leq N(\theta) \leq B_+(\theta)\]
where
\begin{align*}
 B_-(\theta) & = \lim \sup_{k \rightarrow + \infty} \max_{(\lambda, \beta) \in \Lambda} 
 \Phi_{\Pi_k\theta, -} \biggl( \frac{\lambda}{\widehat{\alpha}(\Pi_k\theta)^2} \biggr), \\
 B_+(\theta) & = \lim \inf_{k \rightarrow + \infty}  \min_{(\lambda, \beta) \in \Lambda} 
 \Phi_{\Pi_k\theta, +}^{-1} \biggl( \frac{\lambda}{\widehat{\alpha}(\Pi_k\theta)^2} \biggr).
 \end{align*}
\end{prop}

\vskip2mm
\noindent 
If we do not want to go to the limit, we can use the explicit 
bound
\begin{multline*}
\bigl\lvert N(\theta) - N(\Pi_k(\theta)) \bigr\rvert 
= 
\bigl\lvert \langle  \theta + \Pi_k(\theta), 
\mathcal G \bigl( \theta - \Pi_k(\theta) \bigr) \rangle_{\mathcal H} \rvert 
\\ \leq 2 \lVert \theta \rVert_{\mathcal H} \lVert \mathcal G \rVert_{\infty} 
\lVert \theta - \Pi_k(\theta) \rVert_{\mathcal H} 
\leq 2 \lVert \theta \rVert_{\mathcal H} \mathbf{Tr}(\mathcal G) 
\lVert 
\theta - \Pi_k(\theta) \rVert_{\mathcal H} \\
= 
2 \lVert \theta \rVert_{\mathcal H} \mathbb{E} \bigl( \lVert X \rVert_{\mathcal H}^2 \bigr)
\lVert \theta - \Pi_k(\theta) \rVert_{\mathcal H}.
\end{multline*}
This bound depends on $\lVert \theta - \Pi_k \theta \rVert_{\mathcal H}$.
We will see in the following another bound that goes uniformly 
to zero for any $\theta \in \mathbb{S}_{\mathcal{H}}$ when $k$ tends to infinity.
In the same way, 
proceeding as already done in the previous section we state the analogous 
of Proposition \ref{prop2}. \\[1mm]
Let 
\[ 
\kappa  \geq  \sup_{\substack{\theta \in \mathcal H \\ \mathbb 
E ( \langle \theta, X \rangle_{\mathcal H}^2 ) > 0}} 
\frac{ \mathbb E \bigl( \langle \theta , X \rangle_{\mathcal H}^4 
\bigr)}{\mathbb E \bigl( \langle \theta, X \rangle_{\mathcal H}^2 
\bigr)^2} \qquad \text{ and } \qquad s_4 \geq \mathbb{E} \bigl( \lVert X \rVert_{\mathcal H}^4 
\bigr)^{1/4}
\] 
be known constants and put
\[
 K  = 1 + \left\lceil 2 \log \biggl( \frac{n}{72(2+c) \kappa^{1/2}} \biggr) \right\rceil
\]
where $\displaystyle c= \frac{15}{8 \log(2)(\sqrt{2}-1)} \exp \left( \frac{1 + 2 \sqrt{2}}{2} \right)$
as in equation \eqref{defc}. 
Define
\[
\zeta_* (t) = \sqrt{ 2.032 (\kappa-1) \Bigg( \frac{0.73 \mathbf{Tr}(\mathcal G)}{
t} + \log(K)+ \log(\epsilon^{-1}) \Bigg)}
+ \sqrt{ \frac{98.5\kappa\mathbf{Tr}(\mathcal G)}{ t}} 
\]
and consider, 
according to equation \eqref{defhatn}, the estimators  
\[
\widehat{N}_k(\theta) = \widetilde N_{\widehat \lambda} 
(\theta), \qquad \theta \in 
\mathcal{H}_k.
\]
For any $\theta \in \mathcal{H}$, define $\widehat{N}_{\mathcal H}(\theta)$ by choosing 
any limit point of $\widehat{N}_k \bigl( \Pi_k\theta \bigr)$, 
such as for example $\lim \inf_{k \rightarrow \infty} \widehat{N}_k \bigl( 
\Pi_k \theta \bigr)$. 

\begin{prop}\label{prop1.40fr}
Define the bound
\[ 
B_*(t) = 
\frac{n^{-1/2} \zeta_*(\max \{ t, \sigma \} )}{1 - 4 \, n^{-1/2} 
\zeta_*( \max \{ t, \sigma \} )}, 
\] 
where $\sigma \in ] 0, s_4^2]$ is some energy level 
such that 
\[ 
\bigl[ 6 + (\kappa-1)^{-1} \bigr] \zeta_*( \sigma ) \leq \sqrt{n}. 
\]  
With probability at least $1 - 2 \epsilon$, 
for any $\theta \in \mathcal H$, 
\[ 
\left\lvert \, \frac{\max \{ N(\theta), \sigma \lVert \theta \rVert_{\mathcal H}^2 \}}{\max \{ \widehat{N}_{\mathcal H}(\theta), \sigma \lVert \theta \rVert_{\mathcal H}^2 \}} - 1 \, \right\rvert \leq B_* \Bigg[ \lVert\theta \rVert_{\mathcal H}^{-2} N(\theta) \Bigg].
\] 
\end{prop}
\begin{proof}
This is a consequence of the fact that $\displaystyle\lim_{k\to +\infty} N \bigl( \Pi_k(\theta) 
\bigr) = N(\theta)$ and of the continuity of $B_*$. 
\end{proof}

\vskip2mm
\noindent
As already discussed at the end of Proposition \ref{prop2}, any reasonable bound on the sample size $n$ allows bounding the logarithmic factor $\log(K)$ by a relatively small constant. 
In particular, assuming that $n \leq 10^{20}$, we get $\log(K)\leq 4.35$.

\vskip2mm
\noindent
In the following we construct an estimator of the Gram operator $\mathcal G$.
Let $X_1, \dots, X_n \in \mathcal H$ be an i.i.d. sample drawn according to $\mathrm P.$ 
Define $V = \Span \bigl\{ X_1, \dots, X_n \bigr\}$ and 
\[
V_k = \Span \bigl\{ \Pi_k X_1, \dots, \Pi_k X_n \bigr\} = \Pi_k(V).
\]
Let $\Theta_\delta$ be a $\delta$-net of $\mathbb{S}_{\mathcal{H}} \cap 
V_k$ (where $\mathbb S_{\mathcal H}$ denotes the unit sphere in $\mathcal H$) 
and remark that $\Theta_\delta$ is finite because $\dim(V_k) \leq n  
< + \infty$. 
We compute a linear operator 
$\widehat{\mathcal{G}}_k : V_k \rightarrow V_k$ with minimal Hilbert Schmidt 
norm --- so that 
$\mathbf{Tr}(\widehat{\mathcal{G}}_k^2) \leq \mathbf{Tr}(\mathcal{G}^2 )$ ---
satisfying  
\[
\max_{(\lambda, \beta) \in \Lambda} 
\Phi_- \bigl( \widetilde{N}_{\lambda} (\theta) \bigr) 
\leq \langle \widehat{\mathcal{G}}_k \theta, \theta
\rangle_{\mathcal H} \leq \min_{(\lambda, \beta) \in \Lambda} 
\Phi_+^{-1} \bigl( \widetilde{N}_{\lambda}(\theta) \bigr), 
\qquad \theta \in \Theta_{\delta}.  
\] 
Observe that $\widehat{\mathcal{G}}_k $ plays the same role as the symmetric matrix $Q$ in the finite-dimensional setting.
We consider as an estimator of $\mathcal{G}$ the operator 
\begin{equation}
\label{hilbert_Q}
\mathcal{Q} = \widehat{\mathcal{G}}_k \circ \Pi_{V_k},
\end{equation}
where $\Pi_{V_k}$ is the orthogonal projector on $V_k$. 
Let us decompose 
$\mathcal Q$ in its positive and negative parts and write $\mathcal Q = \mathcal Q_+ - \mathcal Q_-$. 
\begin{prop}\label{propq}
For any threshold $\sigma \in \mathbb{R}_+$ such that $\sigma \leq s_4^2$ 
and $8 \zeta_*(\sigma) \leq \sqrt{n}$,
with probability at least $1 - 2 \epsilon$, 
for any $\theta \in \mathbb{S}_{\mathcal{H}}$ and for any $k$, 
\begin{align*}
\bigl\lvert \max \bigl\{ \langle \theta, \mathcal{Q}_+ \theta \rangle_{\mathcal H}, \sigma \bigr\} & - \max \bigl\{ \langle \Pi_k \theta, \mathcal{G} \Pi_k \theta \rangle_{\mathcal H} , \sigma \bigr\}  \bigr\rvert \\
& \leq 2 \max \bigl\{ \langle \Pi_k \theta, \mathcal{G} \Pi_k \theta \rangle_{\mathcal H}, \sigma \bigr\} B_* \bigl( \langle \Pi_k \theta, \mathcal{G} \Pi_k \theta \rangle_{\mathcal H} \bigr) + 7 \delta \sqrt{ \mathbf{Tr}(\mathcal{G}^2)} \\[1.5mm]  
\bigl\lvert \max \bigl\{ \langle \theta, \mathcal{Q}_+ \theta \rangle_{\mathcal H}, \sigma \bigr\} & - \max \bigl\{ \langle \Pi_k \theta, \mathcal{G} \Pi_k \theta \rangle_{\mathcal H} , \sigma \bigr\} \bigr\rvert \\
& \leq 2 \max \bigl\{ \langle \theta, \mathcal{Q}_+\theta \rangle_{\mathcal H}, \sigma \bigr\} B_* \bigl( 
\min \{ \langle \theta, \mathcal{Q}_+ \theta \rangle_{\mathcal H}, s_4^2 \} \bigr)  + 7 \delta \sqrt{ \mathbf{Tr}(\mathcal{G}^2)}.
\end{align*}
\end{prop}

\vskip2mm
\noindent
For the proof we refer to section \ref{proof_propq}.

\vskip2mm
\noindent
Let us consider $\{ p_i\}_{i = 1}^{+ \infty}$ an orthonormal basis of eigenvectors of $\mathcal{G}$ such that the corresponding 
sequence of eigenvalues $\{ \lambda_i, i = 1, \dots, + \infty \}$ 
is non-increasing.

\begin{prop}
\label{prop1.41fm} 
Consider some threshold $\sigma \in \mathbb{R}_+$ such that 
$\sigma \leq s_4^2$ and $8 \zeta_*(\sigma) \leq \sqrt{n}$. 
With probability at least $1 - 2 \epsilon$, for any $\theta \in \mathbb{S}_{\mathcal{H}}$ and
for any $k$,
\begin{align*}
\bigl\lvert \max \bigl\{ \langle \theta, \mathcal{Q}_+ \theta \rangle_{\mathcal H}, \sigma \bigr\} - \max \bigl\{ \langle \theta, \mathcal{G} \theta \rangle_{\mathcal H} , \sigma \bigr\} \bigr\rvert 
& \leq 2 \max \bigl\{ \langle \theta, \mathcal{G} \theta \rangle_{\mathcal H}, \sigma \bigr\} B_* \bigl( \langle \theta, \mathcal{G} \theta \rangle_{\mathcal H} \bigr) \\ 
& \qquad + 7 \delta \sqrt{ \mathbf{Tr}(\mathcal{G}^2)} + 3 \nu_k \\  
\bigl\lvert \max \bigl\{ \langle \theta, \mathcal{Q}_+ \theta \rangle_{\mathcal H}, \sigma \bigr\} - \max \bigl\{ \langle \theta, \mathcal{G} \theta \rangle_{\mathcal H} , \sigma \bigr\} \bigr\rvert 
& \leq 2 \max \bigl\{ \langle \theta, \mathcal{Q}_+\theta \rangle_{\mathcal H}, \sigma \bigr\} B_* \bigl( \min \{ \langle \theta, \mathcal{Q}_+ \theta \rangle_{\mathcal H}, s_4^2 \} \bigr) \\ 
& \qquad + 7 \delta \sqrt{ \mathbf{Tr}(\mathcal{G}^2)} + 2 \nu_k,  
\end{align*}
where 
\begin{multline*}
\nu_k = \inf_{m =1, \dots, + \infty} \Bigg( \sum_{i=1}^{m-1} 
\lambda_i \lVert p_i - \Pi_k p_i \rVert_{\mathcal H} + \lambda_m / 2 \Bigg)
\\ \leq  
\inf_{m =1, \dots, + \infty} \Bigg( \sum_{i=1}^{m-1} 
\lambda_i \lVert p_i - \Pi_k p_i \rVert_{\mathcal H} + \mathbf{Tr}(\mathcal{G}) / (2m) \Bigg)
\underset{k \rightarrow + \infty}{\longrightarrow} 0. 
\end{multline*}
\end{prop}

\begin{proof}
It is enough to observe that
\begin{multline*}
\bigl\lvert \max \bigl\{ \langle \theta, \mathcal{Q}_+ \theta \rangle_{\mathcal H}, \sigma \bigr\} - \max \bigl\{ \langle \theta, \mathcal{G} \theta \rangle_{\mathcal H} , \sigma \bigr\} \bigr\rvert \\
\leq \bigl\lvert \max \bigl\{ \langle \theta, \mathcal{Q}_+ \theta \rangle_{\mathcal H}, 
\sigma \bigr\} - \max \bigl\{ \langle \Pi_k \theta, \mathcal{G} \Pi_k \theta \rangle_{\mathcal H} 
, \sigma \bigr\}  \bigr\rvert  + \bigl\lvert \langle \theta, \mathcal{G} \theta \rangle_{\mathcal H} - 
\langle \Pi_k \theta, \mathcal{G} \Pi_k \theta \rangle_{\mathcal H} 
\bigr\rvert
\end{multline*}
and, for any $\theta \in \mathbb{S}_{\mathcal{H}}$,  we have
\begin{multline*}
\bigl\lvert \langle \theta, \mathcal{G} \theta \rangle_{\mathcal H} - 
\langle \Pi_k \theta, \mathcal{G} \Pi_k \theta \rangle_{\mathcal H} 
\bigr\rvert = \Bigg\lvert \sum_{i=1}^{+ \infty} 
\bigl( \langle \Pi_k \theta, p_i \rangle_{\mathcal{H}}^2 - \langle 
\theta, p_i \rangle_{\mathcal{H}}^2 \bigr) \lambda_i \Bigg\rvert
\\ =
\Bigg\lvert \sum_{i=1}^{+ \infty} 
\bigl( \langle \theta, \Pi_k p_i \rangle_{\mathcal{H}}^2 - \langle 
\theta, p_i \rangle_{\mathcal{H}}^2 \bigr) \lambda_i \Bigg\rvert
\\ = \Bigg\lvert \sum_{i=1}^{m-1} \langle \theta, \Pi_k p_i + p_i \rangle_{\mathcal{H}} 
\langle \theta, p_i - \Pi_k p_i \rangle_{\mathcal{H}} \lambda_i \Bigg\rvert 
+ \Bigg\lvert \sum_{i=m}^{+ \infty} \bigl( 
\langle \theta, \Pi_k p_i \rangle^2_{\mathcal{H}} - 
\langle \theta, p_i \rangle^2_{\mathcal{H}} \bigr) \lambda_i \Biggr\rvert
\\ 
\leq \sum_{i=1}^{m-1} 2 \lambda_i \lVert p_i - \Pi_k p_i \rVert_{
\mathcal{H}} + \max \Biggl\{ \sum_{i=m}^{+ \infty} \lambda_i \langle 
\Pi_k \theta, p_i \rangle_{\mathcal{H}}^2, \sum_{i=m}^{+\infty} \lambda_i 
\langle \theta, p_i \rangle_{\mathcal{H}}^2 \Biggr\}
\\
\leq \inf_{m =1, \dots, +\infty} \Bigg( 
\sum_{i=1}^{m-1} 2 \lambda_i \lVert p_i - \Pi_k p_i \rVert_{\mathcal{H}} 
+ \lambda_m \Bigg). 
\end{multline*}
Indeed, 
$\displaystyle \sum_{i=m}^{+ \infty} \langle \Pi_k \theta, p_i 
\rangle_{\mathcal{H}}^2 \leq \lVert \Pi_k \theta \rVert_{\mathcal{H}}^2 
\leq \lVert \theta \rVert^2 \leq 1, $ so that 
\[
\sum_{i=m}^{+ \infty} \lambda_i \langle \Pi_k \theta, p_i \rangle_{\mathcal{H}}^2 
\leq \bigl( \sup_{i \geq m} \lambda_i \bigr) \Biggl( 
\sum_{i=m}^{+ \infty} \langle \Pi_k \theta, p_i \rangle^2_{\mathcal{H}} 
\Biggr) \leq \lambda_m, 
\]
and in the same way $ \displaystyle 
\sum_{i=m}^{+ \infty} \lambda_i 
\langle \theta, p_i \rangle_{\mathcal{H}}^2 \leq \lambda_m$.
\end{proof}

\vskip 2mm
\noindent
Remark that we can use this result to bound $\bigl\lvert \langle \theta, 
(\mathcal{G} - \mathcal{Q}_+) 
\theta \rangle_{\mathcal H} \bigr\rvert$, using the inequality
\[ 
\bigl\lvert \langle \theta, (\mathcal{G} - \mathcal{Q}_+) 
\theta \rangle_{\mathcal H} \bigr\rvert \leq 
\bigl\lvert \max \bigl\{ \langle \theta, \mathcal{Q}_+ \theta \rangle_{\mathcal H}, 
\sigma \bigr\} - \max \bigl\{ \langle \theta, \mathcal{G} \theta \rangle_{\mathcal H} 
, \sigma \bigr\} \bigr\rvert + \sigma.  
\] 

\section{Empirical results}\label{sec_emp}
We present some empirical results 
that show the performance of our robust estimator. 
We use here a simplified 
construction that does not follow exactly the 
definition of the 
estimator $Q$, but exhibits the same kind of behaviour. 
We have simplified the construction in two ways. 
First, we do less intensive computations 
by using more directions than in 
equation \myeq{eq:2.15} but less than specified in the $\delta$-net 
$\Theta_{\delta}$ required by the theory.  
More precisely, we estimate repeatedly 
using equation \myeq{eq:2.15} in an eigen-basis of the 
previous iterate of the estimation.  
Second, we do not use the theoretical value of $\lambda$, 
that is necessarily pessimistic 
for the sake of mathematical correctness. We use instead the 
optimal constant for estimating $\mathds{E} \bigl( \langle 
\theta, X \rangle^2 \bigr)$ in a single direction, 
as given in \cite{Catoni12}.\\[2mm]

\vskip 2mm
\noindent
Let $X_1, \dots, X_n \in \mathbb R^d$ be  a sample drawn according to the probability distribution $\mathrm P$ and let $\lambda>0.$  
Let $p\in \mathbb R^n$ and define $S(p, \lambda)$ 
as the solution of 
\[
\sum_{i=1}^n \psi\left[ \lambda\left( S(p, \lambda)^{-1} p_i^2 - 1\right)\right]=0.
\]
In practice we compute $S(p,\lambda)$ using the Newton algorithm. 
We observe that, when $p_i = \langle \theta, X_i\rangle$ and $\lambda$ is suitably chosen, 
$S(p, \lambda)$ is an approximation of the estimator $\widehat N(\theta)$ of the quadratic form $N(\theta)$
and more precisely, in this case, $S(p, \lambda)$ is exactly $\widetilde N_{\lambda}(\theta)$ defined in equation \eqref{tildeN}.\\[1mm]
Define $S(p)$ as the 
solution obtained when the parameter $\lambda$ is set to 
\[
\lambda = m\  \sqrt {\frac{1}{v} \left[ \frac{2}{n} \log(\epsilon^{-1}) \left(1-  \frac{2}{n} \log(\epsilon^{-1})  \right)  \right]}
\] 
where $\displaystyle m = \frac{1}{n}\sum_{i=1}^n p_i^2$, 
$\displaystyle v = \frac{1}{n-1} 
\sum_{i=1}^n \left( p_i^2 - m \right)^2$ and $ \epsilon = 0.1$.\\[1mm]
\vskip 2mm
\noindent
Let $\bar{\lambda}_1\geq \dots \geq \bar{\lambda}_d \geq 0$ be the 
eigenvalues of the empirical Gram matrix $\bar{G}$,
that will be our starting point, and let $u_1, \dots, u_d$ be a 
corresponding orthonormal basis of eigenvectors. 
We decompose the empirical Gram matrix as
\[
\bar{G} = U D U^{\top}
\]
where $U$ is the orthogonal matrix whose columns are the eigenvectors of 
$\bar{G}$ and $D$ is the diagonal matrix 
$D = \mathrm{diag}(\bar{\lambda}_1,\dots,\bar{\lambda}_d )$. 
Based on the polarization formula,
\[
u_i^{\top} G u_j = \frac{1}{4} \left[ N(u_i+u_j) - N(u_i-u_j) \right], \quad i,j =1, \dots, d,
\]
we revise iteratively our estimation of $G$ by estimating 
$N(u_i+\sigma u_j)$, with $\sigma \in \{+1, -1\}$, by  
\[
S \Bigl( \langle u_i+\sigma u_j, X_{\ell} \rangle \ , \ 1 \leq \ell \leq n  
\Bigr).
\]
More precisely, for any $n \times d$ matrix $W$, we define $C(W)$ as the $d \times d$ matrix with entries
\[
C(W)_{i, j} =  \frac{1}{4} \Bigg[S\Bigg(\bigl(  W_{\ell, i}+ 
W_{\ell, j} \bigr) \ | \ 1\leq \ell \leq n \Bigg) - S \Bigg(\bigl(  
W_{\ell, i}- W_{\ell, j} \bigr) \ | \ 1 \leq \ell \leq n \Bigg) \Bigg].
\]
Let $Y$ be the matrix whose $\ell$-th row is the vector $X_\ell$, so that 
\[
(YU)_{\ell, k} = \langle u_k , X_{\ell} \rangle, \quad 1\leq  \ell \leq n, \ 1\leq k \leq d.
\] 
We update the Gram matrix estimate to 
\[
Q_0 = U C(YU) U^{\top}. 
\]
Then we iterate the update scheme, decomposing 
$Q_0$ as 
\[
Q_0 = O_0 D_0 O_0^{\top},
\]
where $O_0 O_0^{\top} =O_0^{\top} O_0= \mathrm I$ and $D_0$ is a diagonal matrix and computing 
\[
Q_1 = O_0 C(YO_0) O_0^{\top}. 
\]
The inductive update step is more generally the following. 
Assuming we have constructed $Q_k$, we decompose it as
\[
Q_k = O_k D_k O_k^{\top}
\]
where $O_k O_k^{\top} =O_k^{\top} O_k= \mathrm I$ and $D_k$ is a diagonal matrix
and we define the new updated estimator of $G$ as 
\[
Q_{k+1} = O_k C(YO_k) O_k^{\top}.
\]
We stop this iterative estimation scheme when $\| Q_k - Q_{k-1} \|_F$ falls under a given threshold. 
In the following numerical experiment we more simply performed four updates.
We take as our robust estimator of $G$ the last update $Q_k$.

\vskip 5mm
\noindent
We now present an example of the performance of this estimator,
for some i.i.d. sample of size $n = 100$ in $\mathbb R^{10}$ drawn according to the Gaussian mixture distribution
\[
\mathrm P = (1 - \alpha) \, \mathcal{N}(0,M_1) + \alpha \, \mathcal{N}(0, 16 \; \mathbf{I}),
\]
where $\alpha= 0.05$ and
\[
M_1 =
\left[ 
\begin{array}{ll:lll}
2 & 1 & \multicolumn{3}{c}{\multirow{2}{*}{
\Large 0}
} \\
1 & 1 &  & & \\ 
\hdashline 
\multicolumn{2}{c:}{ 
\multirow{3}{*}{
\Large 0} }  
& 0.01 & & \raisebox{-10pt}[0pt][0pt]{\large 0} \\ 
 &  &  & \ddots & \\ 
 &  &  \hspace{10pt}\raisebox{8pt}[0pt][0pt]{\large 0} & & 0.01 
\end{array}
\right]. 
\]
The Gram matrix of $P$ is equal to 
\[
G = (1 - \alpha) M_1 + 16 \, \alpha \, \mathbf{I} = 
\left[ 
\begin{array}{ll:lll}
2.7 & 0.95 & \multicolumn{3}{c}{\multirow{2}{*}{\Large 0}} \\
0.95 & 1.75 &  & & \\ 
\hdashline \multicolumn{2}{c:}{\multirow{3}{*}{\Large 0}}  & 0.8095 & & \raisebox{-10pt}[0pt][0pt]{\large 0} \\ 
 &  &  & \ddots & \\ 
 &  &  \hspace{10pt}\raisebox{8pt}[0pt][0pt]{\large 0} & & 0.8095 
\end{array}
\right]. 
\]

\vskip2mm
\noindent
This example illustrates a favorable situation where the 
performance of the robust estimator is particularly striking
when compared to the empirical Gram matrix. 
As it can be seen on the expression of the sample distribution as 
well as on the configuration plots below, this is a situation 
of intermittent high variance : the sample is a mixture of 
a rare high variance signal and a frequent low variance more structured signal. \\[1mm]
We tested the algorithm on 500 different samples, of size $n = 100$ each, drawn according to the Gaussian mixture distribution defined above. Random sample configurations are presented in figure \ref{fig1}. \\[1mm]
Figure \ref{fig2} shows that the robust estimator $Q$ 
significantly improves the error in terms of square of the 
Frobenius norm when compared to the empirical estimator $\bar G$. 
The red solid line represents the empirical quantile function 
of the errors of the robust estimator, 
whereas the blue dotted line represents the quantiles of   
$\| \bar G - G\|^2_F$. \\This quantile function is obtained 
by sorting the 500 empirical errors in increasing order.\\[1mm]
We also represented in \vref{fig2} the boxplots of the 
distributions of $\lVert \bar{G} - G \rVert^2_{F}$ and 
$\lVert Q - G \rVert^2_{F}$ on $500$ statistical experiments.
(The boxplots show the first, second and third quartiles, 
with whiskers extending to the most extreme data point 
within 1.5 of interquartile range. Further away extreme data points 
are printed individually).  

\noindent
The mean of the square distances $\| Q - G\|^2_F$ on 
500 trials is $5.6 \pm 0.4$, 
where the indicated mean estimator and confidence interval is the 
non-asymptotic confidence interval given by Proposition  
2.4 of \cite{Catoni12} at confidence level $0.99$. 
In the case of the empirical estimator, the mean is $15.5 \pm 2$. 
The empirical standard deviations accross 500 trials 
of $\lVert Q - G \rVert^2_{F}$ 
and $\lVert \bar{G} - G \rVert^2_{F}$ respectively 
are close to $2$ and $10$. 
So we see that in this case the robust estimator reliably 
decreases the expected error by a factor larger than $2$ and also produces 
errors with a much smaller standard deviation from sample to sample. \\[1mm]
In Appendix \ref{emp_G} we show  
from a theoretical point of view that
the two estimators $Q$ and $\bar G$ behave in a similar way in light tail situations (meaning that in this case their predictions are the same).

\begin{figure}[htbp]
\begin{center}
\includegraphics[height=50mm]{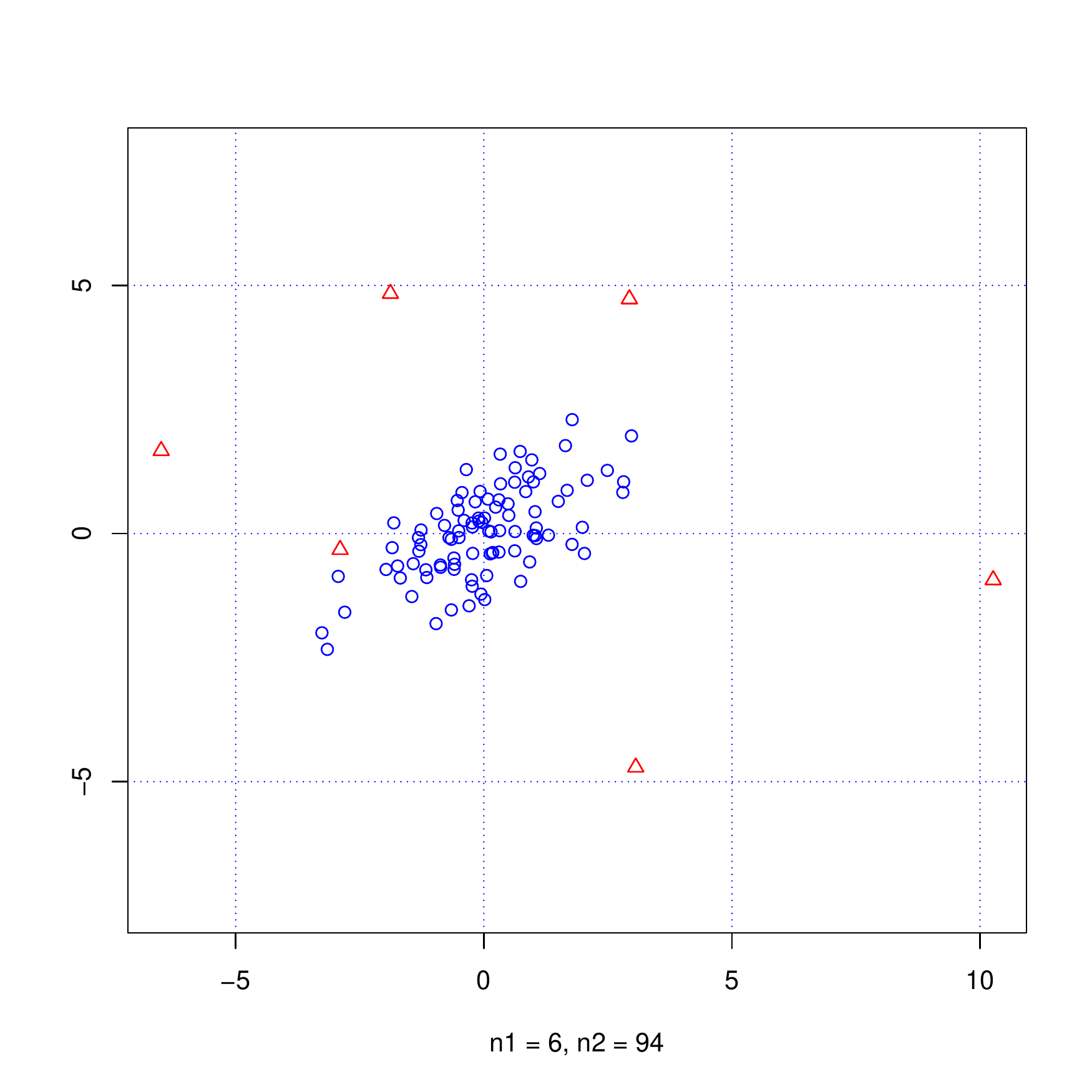}
\includegraphics[height=50mm]{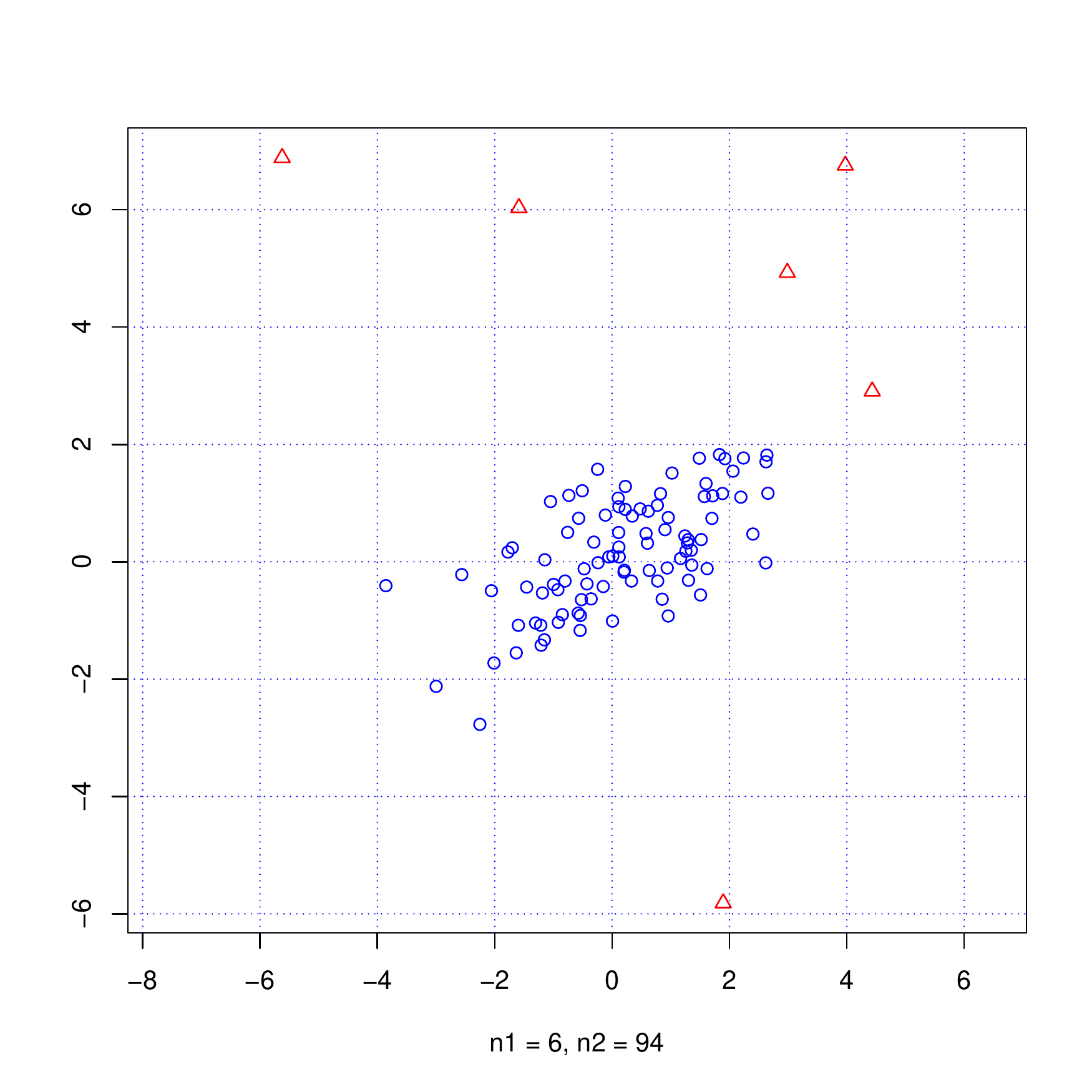}
\includegraphics[height=50mm]{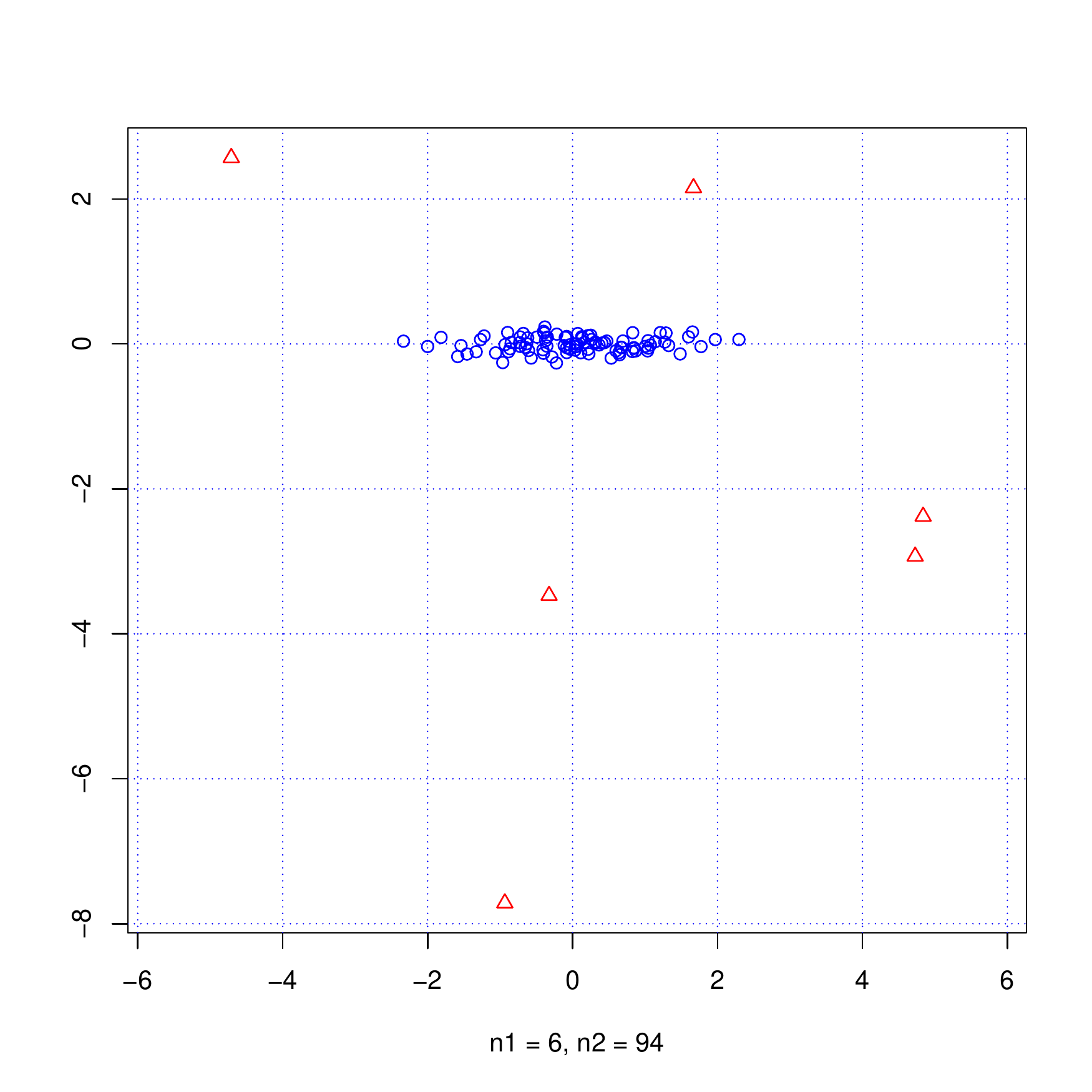}
\includegraphics[height=50mm]{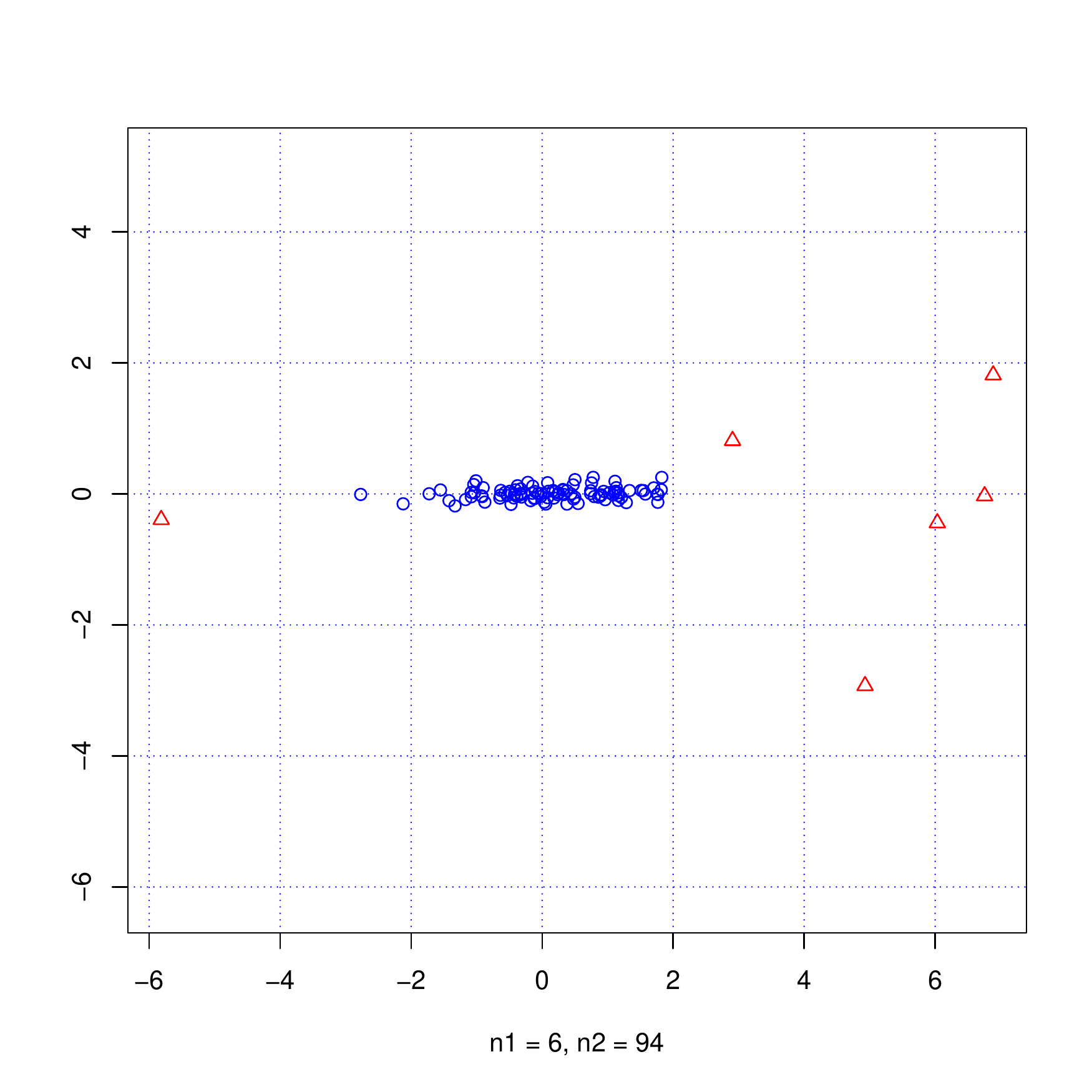}
\caption{ Two data samples projected onto the two first coordinates 
(above) and the second and third coordinates (below).
Blue circles are drawn from the most frequent distribution 
and red triangles from the less frequent one.}
\label{fig1}
\end{center}
\end{figure}

\begin{figure}[htbp]
\begin{center}
\includegraphics[height=50mm]{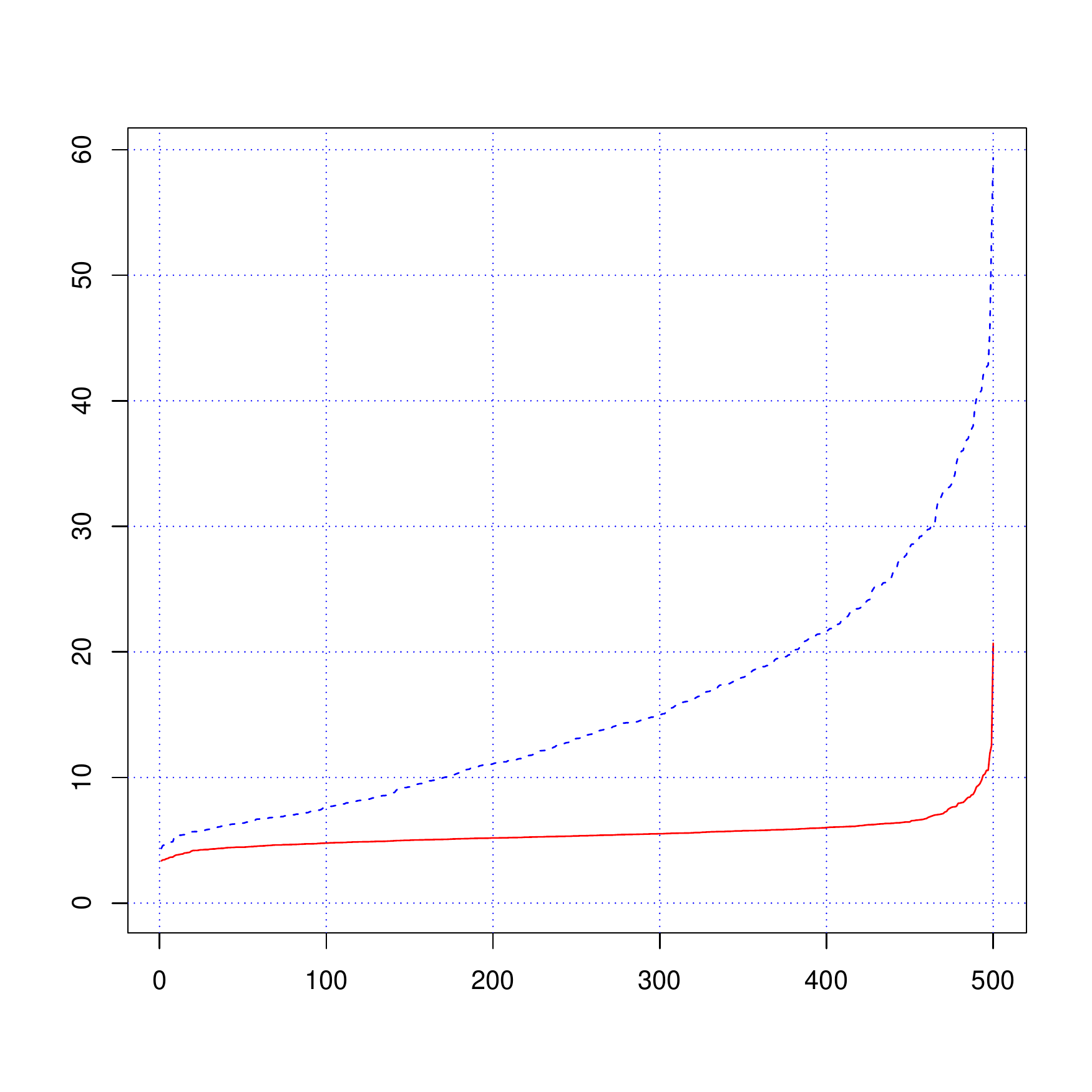} 
\includegraphics[height=50mm]{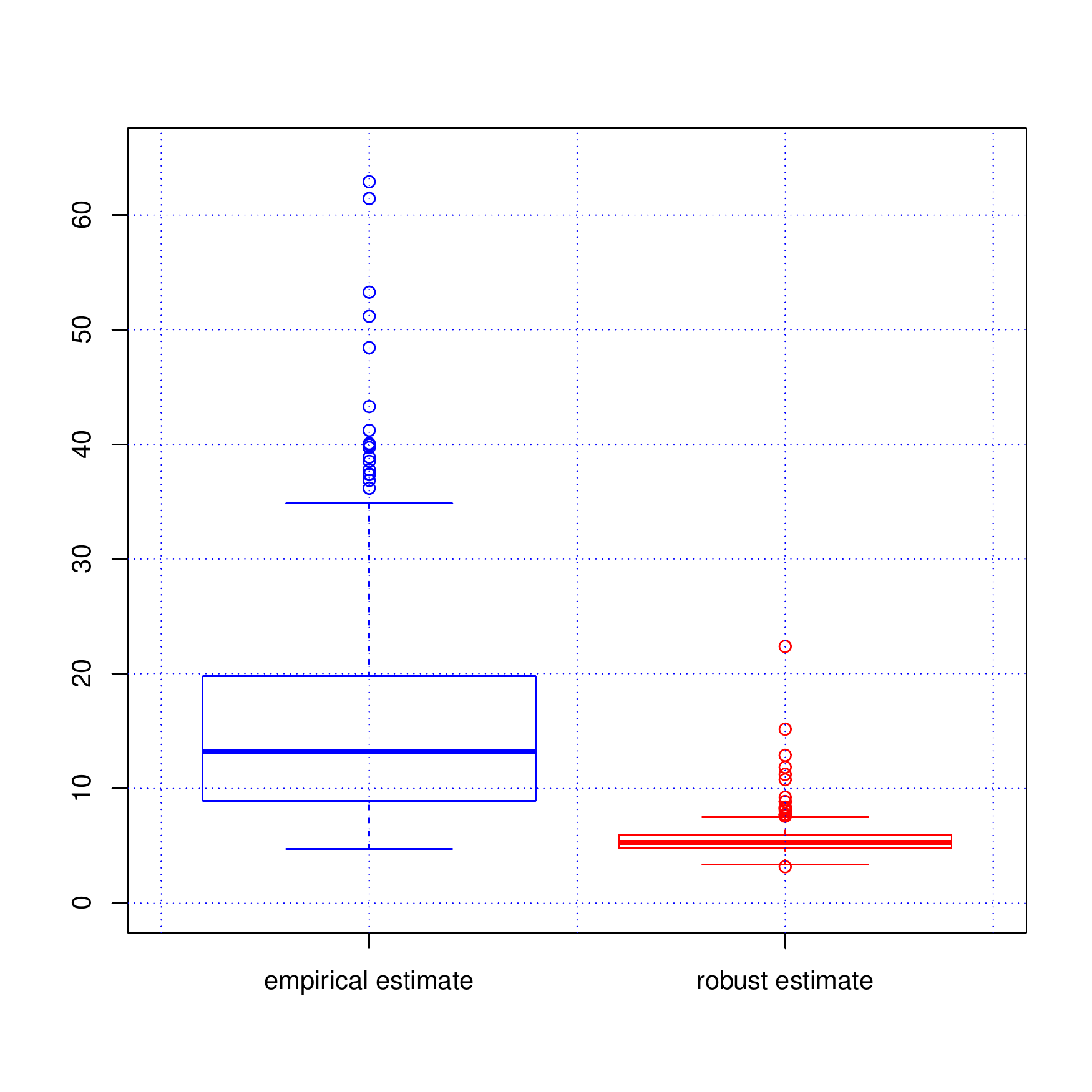} 
\caption{ The solid red line represents the empirical quantile 
function of the square distances $\| Q - G\|_F^2$ in 
500 statistical experiments,
the dotted blue line represents the empirical quantile function 
of the square distances $\| \bar G - G\|_F^2$. The corresponding boxplots
are also displayed on the right figure. 
}
\label{fig2}
\end{center}
\end{figure}

\newpage
\section{The classical empirical estimator}\label{emp_G}
The main goal of section \ref{sec1} is to estimate the Gram matrix $G = \mathbb E \left( XX^{\top} \right)$, where $X\in \mathbb R^d$ is a random vector of unknown law $\mathrm P \in \mathcal M_+^1(\mathbb R^d)$, 
from an i.i.d. sample $X_1, \dots, X_n\in \mathbb R^d$ drawn according to $\mathrm P.$ 
We have constructed a robust estimator of the Gram matrix and in section \ref{sec_emp} we have shown empirically its performance in the case of a Gaussian mixture distribution.
In this section we show from a theoretical point of view that 
the classical empirical estimator 
\[
\bar G = \frac{1}{n} \sum_{i=1}^n X_i X_i^{\top}
\]
behaves similarly to our robust estimator 
in light tail situations, while it may perform worse otherwise.\\[1mm]
As already done in section \ref{sec1}, we consider the quadratic form 
\[
N(\theta) = \theta^{\top} G \theta = \mathbb E \left( \langle \theta, X \rangle^2 \right)
\]
and we denote by 
\[
\bar N(\theta) = \theta^{\top} \bar G \theta = \frac{1}{n} \sum_{i=1}^n \langle \theta, X_i \rangle^2
\]
the quadratic form associated to the empirical Gram matrix $\bar G$. 
According to the notation introduced in section \ref{sec1}, let $a>0$ and let 
\[ 
 K  = 1 + \left\lceil a^{-1} \log \biggl( \frac{n}{72(2+c) \kappa^{1/2}} \biggr) \right\rceil
\]
where $\displaystyle\kappa = \sup_{\substack{
\theta \in \mathbb{R}^d \\ 
\mathbb E ( \langle \theta, X \rangle^2 ) > 0
}} \frac{\mathbb E \bigl( \langle \theta , X \rangle^4 \bigr)}{
\mathbb E \bigl( \langle \theta, X \rangle^2 \bigr)^2}$ 
and $\displaystyle c= \frac{15}{8 \log(2)(\sqrt{2}-1)} \exp \left( \frac{1 + 2 \sqrt{2}}{2} \right)$.
Let us put
\begin{equation}
\label{eq2.5}
R = \max_{i=1,\dots,n} \|X_i\|
\end{equation}
and let us introduce
\[
\tau_* (t) = \frac{\lambda_*(t)^2 \exp(a/2) R^4}{3 \max \{ t, \sigma \}^2}, \quad t \in \mathbb R_+,
\]
where $\lambda_*$ is defined in equation \eqref{eq1.33} as 
\[
\lambda_*(t)= \sqrt{ \frac{2}{n (\kappa - 1)} \Biggl(  
\frac{(2 + 3c) \mathbb{E} \bigl( \lVert X \rVert^4 \bigr)^{1/2}}{4 (2 + c) 
\kappa^{1/2} \max \{ t, \sigma \}} + \log(K / \epsilon)
\Biggr)}. 
\]
At the end of the section we mention some assumptions under which 
it is possible to give a non-random bound for $R$. 

\vskip1mm
\noindent
The following proposition, compared with the result
obtained for the robust estimator $\widehat N( \theta)$, presented in 
Proposition \ref{prop2},
shows that the different behavior of the two estimators $\widehat N$
and $\bar N$ can appear only for heavy tail data distributions. 

\begin{prop}\label{prop2.8eg}
Consider any threshold $\sigma \in \mathbb{R}_+$ such that $\sigma \leq 
\mathbb{E} \bigl( \lVert X \rVert^4 \bigr)^{1/2}$. 
With probability at least $1 - 2 \epsilon$, for any $\theta \in \mathbb{S}_d$, 
\[
\biggl\lvert \, \frac{\max \{ \bar{N}(\theta), \sigma \}}{
\max \{ N(\theta), \sigma  \}} - 1 \biggr\rvert \leq 
B_* \bigl( N(\theta) \bigr) + \frac{ \tau_*\bigl( N(\theta) \bigr)}{ 
\bigl[ 1 - \tau_* \bigl( N(\theta) \bigr) \bigr]_+ 
\bigl[ 1 - B_* \bigl( N(\theta) \bigr) \bigr]_+}.
\]
where $B_*$ is defined in Proposition \ref{prop2}.
\end{prop}

\vskip2mm
\noindent
For the proof we refer to section \ref{sec65pf}.\\[1mm]
Observe that, also in this case, the bound does not depend
explicitly on the dimension $d$ of the ambient space and thus the result can be extended
to any infinite-dimensional Hilbert space.

\vskip 5mm
\noindent
We continue this section by stating assumptions under which it is possible to give a non-random bound for $R,$ defined in equation \eqref{eq2.5}.

\vskip 2mm
\noindent
Assume that, for some exponent $p \geq 1$ and some 
positive constants $\alpha$ and $\eta$,
\[ 
\mathbb E \biggl[ \exp \biggl( \frac{\alpha}{2} \biggl( \frac{\lVert X \rVert^{2/p}}{ \mathbf{Tr}(G)^{1/p}} - 1 - \eta^{2/p} \biggr) \biggr) \biggr] \leq 1.
\] 
In this case, with probability at least $1 - \epsilon$, 
\begin{equation}\label{eqR_notes}
R \leq \mathbf{Tr}(G)^{1/2} \bigg( 1 + \eta^{2/p} + 2 \alpha^{-1} \log \bigl( n / \epsilon \bigr) \biggr)^{p/2},
\end{equation}
where we recall that 
$
\mathbf{Tr}(G) = \mathbb E \bigl[ \lVert X \rVert^2 \bigr].
$\\[1mm] 
To give a point of comparison, in the centered Gaussian case where $X \sim \mathcal{N}(0, G)$ is a Gaussian vector, we have, for any $\alpha \in ]0, \lambda_1^{-1} \mathbf{Tr}(G) [,$
\[
\mathbb E \Biggl[ \exp \Biggl[ \frac{\alpha}{2} \Biggl( \frac{\lVert X \rVert^2}{
\mathbf{Tr}(G)}  + \frac{1}{\alpha} 
\sum_{i=1}^d \log \biggl( 1 - \frac{\alpha \lambda_i}{\mathbf{Tr}(G)} \biggr) 
\Biggr) \Biggr] \Biggr] = 1, 
\]
where $\lambda_1 \geq \cdots \geq \lambda_d$ are the eigenvalues 
of $G.$ 
Therefore, with probability at least $1 - \epsilon$,
\[ 
R \leq \mathbf{Tr}(G)^{1/2} \Biggl( - \frac{1}{\alpha} \sum_{i=1}^d \log 
\biggl( 1 - \frac{\alpha \lambda_i}{\mathbf{Tr}(G)} \biggr) + 
\frac{2 \log (n / \epsilon)}{\alpha} \Biggr)^{1/2} . 
\] 
We can then consider the optimal value of $\alpha$ in the right-hand side 
of the previous equation, to establish that with probability at least $1 - 
\epsilon$, 
\begin{align*}
R & \leq \inf_{\alpha \in ]0, \Tr(G) / \lambda_1 [} 
\Tr(G)^{1/2} \Biggl( - \frac{1}{\alpha} \sum_{i=1}^d \log \biggl( 
1 - \frac{\alpha \lambda_i}{\Tr(G)} \biggr) + \frac{2 \log (n / \epsilon)}{
\alpha} \Biggr)^{1/2}  \\ 
& \leq 
\inf_{\alpha \in ]0, \Tr(G) / \lambda_1 [} 
\Tr(G)^{1/2} \Biggl( \sum_{i=1}^d 
\frac{\lambda_i }{\Tr(G) - \alpha \lambda_i } 
+ \frac{2 \log (n / \epsilon)}{
\alpha} \Biggr)^{1/2}
\\ & \leq 
\inf_{\alpha \in ]0, \Tr(G) / \lambda_1 [} 
\Biggl( 
\frac{\Tr(G)^2 }{\Tr(G) - \alpha \lambda_1} 
+ \frac{2 \Tr(G) \log (n / \epsilon)}{
\alpha} \Biggr)^{1/2}
\\ & \leq 
\inf_{\alpha \in ]0, \Tr(G) / \lambda_1 [} 
\Biggl( 
\Tr(G) + \frac{\alpha \lambda_1 \Tr(G)}{\Tr(G)  - \alpha \lambda_1} + 
\frac{2 \log(n / \epsilon) \lambda_1 ( \Tr(G)  - \alpha 
\lambda_1)}{\alpha \lambda_1 } \\ & \qquad \qquad  
+ 2 \lambda_1 \log(n / \epsilon)    
\Biggr)^{1/2} \\ 
& = \Bigl( \Tr(G) + 2 \sqrt{2 \log(n / \epsilon) \lambda_1 \Tr(G)} + 
2 \lambda_1 \log(n / \epsilon) \Bigr)^{1/2} = \sqrt{\Tr(G)} + \sqrt{
2 \lambda_1 \log(n / \epsilon)}.   
\end{align*}

\vskip 3mm
\noindent
In order to replace hypothesis \eqref{eqR_notes} by some polynomial assumptions it is convenient to replace $R$ with 
\[
\widetilde{R} = \bigg( \frac{1}{n} \sum_{i=1}^n \lVert X_i \rVert^6 \biggr)^{1/6}.
\]
Indeed, by the Bienaym\'e Chebyshev inequality, we get that,
with probability at least $1 - \epsilon$, 
\[ 
\widetilde{R} \leq \Biggl( \mathbb E \bigl[ \lVert X \rVert^6 \bigr]
+ \biggl( \frac{ \mathbb E \bigl[ \lVert X \rVert^{12} \bigr]}{n \epsilon}
\biggr)^{1/2} \Biggr)^{1/6}
\leq \Bigl( 1 + (n \epsilon)^{-1/2} \Bigr)^{1/6} \,  \mathbb E 
\Bigl[ \lVert X \rVert^{12} \Bigr]^{1/12}
\] 
and hence, with probability at least $1 - n^{-1}$, 
\[ 
\widetilde{R} \leq 2^{1/6} \, \mathbb E \Bigl[ \lVert X \rVert^{12} \Bigr]^{1/12}.
\] 

\vskip 2mm
\noindent
We can prove an analogue of Proposition 
\ref{prop2.8eg}, where $\widetilde R$ plays the role 
of $R$.  
\begin{prop}\label{prop52empg}
Let $0 < \sigma \leq \mathbb{E} \bigl( \lVert X \rVert^4 \bigr)^{1/2}$ and let us put 
\[
\zeta_*(t) = \frac{\lambda_*(t)^2 \  \exp(a/2)\ \widetilde{R}^6}{3 \max \{ 
t, \sigma \}^3},  \quad t \in \mathbb R_+. 
\]
With probability at least $1 - 2 \epsilon$, for any 
$\theta \in \mathbb{S}_d$, 
\begin{multline*}
\biggl\lvert \, \frac{\max \{ \bar{N}(\theta), \sigma \}}{\max \{
N(\theta), \sigma \}} - 1 \biggr\rvert \leq 
B_* \bigl( N(\theta) \bigr) + \frac{ \zeta_* \bigl( N(\theta) \bigr)}{
\bigl[ 1 - 
B_* \bigl( N(\theta) \bigr) \bigr]_+} 
\\ \leq \mathcal{O} \Biggl( \sqrt{ \frac{\kappa}{n} \biggl( 
\frac{\Tr(G)}{\max \{ N(\theta), \sigma \}} 
+ \log \bigl( \log(n) / \epsilon \bigr) \biggr)}  \; \Biggr) 
\\ + \mathcal{O} \Biggl( \frac{\mathds{E} \bigl( \lVert X \rVert^{12} 
\bigr)^{1/2}}{n \kappa \bigl( 1 + (n \epsilon)^{-1/2} 
\bigr) \bigl( \max \{ N(\theta), \sigma \} \bigr)^3}
\biggl( \frac{\Tr(G)}{\max \{ N(\theta), \sigma\}} + \log \bigl( 
\log(n) / \epsilon \bigr) \biggr) \Biggr). 
\end{multline*}
\end{prop}
\vskip2mm
\noindent
For the proof we refer to section \ref{pf111}.

\section{Generalization}\label{sec3}
In this section we come back to the finite-dimensional framework and we 
consider the problem of estimating the expectation of a symmetric random matrix. 
We will use these results to estimate the covariance matrix in the case of unknown expectation. 

\subsection{Symmetric random matrix}

Let $A\in M_d(\mathbb R)$ be a symmetric random matrix of size $d.$ As already observed for the Gram matrix, the expectation of $A$ can be recovered via the polarization identity from the quadratic form
\[
N_A(\theta)=\theta^{\top} \B{E} ( A ) \, \theta, \qquad \theta \in \mathbb R^d,
\]
where the expectation is taken with respect to the unknown probability distribution of $A$ on the space of symmetric matrices of size $d.$
Observe that, if we decompose $A$ in its positive and negative parts
\[
A= A_+-A_-,
\]
where $A_+$ and $A_-$ are defined by keeping respectively the 
positive and negative parts of the eigenvalues of $A$, 
in the framework of functional calculus on symmetric matrices, 
the quadratic form $N_A$ rewrites as
\begin{align*}
N_A(\theta) & = \mathbb E \left[ \theta^{\top}A_+\theta\right] - \mathbb E \left[ \theta^{\top}A_-\theta\right] = N_{A_+}(\theta) - N_{A_-}(\theta). 
\end{align*}
Thus in the following we will consider the case of a symmetric positive semi-definite random matrix of size $d.$\\[2mm]
From now on let $A\in M_d(\mathbb R)$ be a symmetric positive semi-definite random matrix of size $d$ 
and let $\mathrm P$ be a probability distribution on the space of symmetric positive semi-definite random matrices of size $d.$
Our goal is to estimate 
\[
N(\theta) = \mathbb E \left[ \theta^{\top}A\theta\right], \quad \theta \in \mathbb R^d,
\]
from an i.i.d. sample $A_1, \dots, A_n \in M_d(\mathbb R)$ of symmetric positive semi-definite matrices drawn according to $\mathrm P.$
We observe that the quadratic form $N(\theta)$ rewrites as
\[
N(\theta)=\mathbb E\left[ \| A^{1/2} \theta\|^2 \right]
\]
where $A^{1/2}$ denotes the square root of $A.$ \\[1mm]
The construction of the (robust) estimator $\widehat N(\theta)$ follows the one already done in the case of the Gram matrix with the necessary adjustments. 
For any $\lambda>0$ and for any $\theta \in \mathbb R^d,$ we consider the empirical criterion 
\[
r_{\lambda} (\theta)  = \frac{1}{n} \sum_{i=1}^n \psi \left( \| A_i^{1/2} \theta\|^2-\lambda\right),
\]
where the influence function $\psi$ is defined as in equation \eqref{defnpsi}, 
and we perturb the parameter $\theta$ with the Gaussian perturbation $\pi_{\theta} \sim \mathcal N(\theta, \beta^{-1}\mathrm I)$ 
of mean $\theta$ and covariance matrix $ \beta^{-1}\mathrm I,$ 
where $\beta>0$ is a free real parameter. 
We consider the family of estimators
\[
\widetilde N(\theta) = \frac{\lambda}{ \widehat\alpha(\theta)^2}
\]
where $\widehat \alpha(\theta) = \sup\{ \alpha\in\mathbb R_+ \ | \ r_{\lambda}(\alpha\theta)\leq0\}$. 
Let us put
\begin{equation}\label{skappa_A}
s_4= \mathbb E[\|A\|_{\infty}^2]^{1/4} \quad \text{and} \quad 
\kappa = \sup_{\substack{\theta\in \mathbb R^d\\ \mathbb E [ \lVert A^{1/2} \theta \rVert^2 ]>0}} \frac{\displaystyle\mathbb E  \bigl[ \lVert A^{1/2} \theta \rVert^4 \bigr]}{ \displaystyle\mathbb E  \bigl[ \lVert A^{1/2} \theta \rVert^2 \bigr]^{2}},
\end{equation}
where $\lVert A \rVert_{\infty}$ is the operator norm, that is in this context 
of symmetric positive semi-definite matrices equal to the largest 
eigenvalue of $A$. 
The finite set $\Lambda \subset \bigl( \mathbb{R}_+ \setminus \{ 0 \} \bigr)^2$ of possible values of the couple of parameters $(\lambda, \beta)$ is defined as
\[
\Lambda = \bigl\{ (\lambda_j, \beta_j) \ | \ 0 \leq j < K \bigr\}, 
\]
where 
\begin{equation}\label{K-cov}
K = 1 + \left\lceil a^{-1} \log \biggl( \frac{n}{72(2+c) \kappa^{1/2}} \biggr) \right\rceil,
\end{equation}
for some real positive parameter $a$ to be chosen later on, and
\begin{align*}
\lambda_j & = \sqrt{ \frac{2}{(\kappa -1) n} \biggl( 
\frac{(2 + 3c) \mathbb{E} \bigl( \mathbf{Tr}(A^2) \bigr)}{4(2+c) \kappa^{1/2} 
\mathbb{E} \bigl( \lVert A \rVert_{\infty}^2 \bigr)} \exp(ja) + 
\log \bigl( K / \epsilon \bigr)}, \\
\beta_j & = \sqrt{ 2(2+c) \kappa^{1/2} \mathbb{E} \bigl( \lVert A \rVert_{\infty}^2
\bigr) \exp \bigl[ - (j-1/2))a \bigr]}.
\end{align*}
Note that in the case of the Gram matrix, the picture is simplified by 
the fact that 
$\lVert X X^{\top} \rVert_{\infty}^2 = \Tr \bigl[(X X^{\top})^2 \bigr]$, 
whereas here we have to take into account the fact that the operator norm 
and the Frobenius norm of $A$ are different when the rank of $A$ 
is larger than one.\\
We recall that $c$ is defined in equation \eqref{defc} as $
\displaystyle c= \frac{15}{8 \log(2)(\sqrt{2}-1)} \exp \left( \frac{1 + 2 \sqrt{2}}{2} \right)$.\\[1mm]
According to equation \eqref{defBlb}, let 
\[
B_{\lambda, \beta}(t) = \begin{cases} \displaystyle\frac{\gamma + \lambda \delta / \max \{ t , \sigma \} }{ 1 - \mu - \gamma - 2 \lambda \delta / \max \{ t,  \sigma \}}, & (\lambda, \beta, t) \in \Gamma,  \\
+ \infty, & \text{ otherwise,} 
\end{cases}
\]
and put $\displaystyle (\widehat{\lambda}, \widehat{\beta}) = \arg \min_{(\lambda, \beta)  \in \Lambda} B_{\lambda, \beta} \Bigl[ \lVert \theta \rVert^{-2} \widetilde{N}_{\lambda} (\theta) \Bigr]$.
Define the estimator $\widehat N$ as
\begin{equation}\label{hatN_A}
\widehat{N} (\theta) = \widetilde{N}_{\widehat{\lambda}}(\theta).
\end{equation}

\begin{prop}\label{props1}
Let $\sigma \in]0,s_4^2]$
be some energy level. With probability at least $1 - 2 \epsilon$, 
for any $\theta \in \mathbb{R}^d$, 
\[ 
\left\lvert \, \frac{\max \{ N(\theta), \sigma \lVert \theta \rVert^2 \}}{\max \{ \widehat{N}(\theta), \sigma \lVert \theta \rVert^2 \}} - 1 \, \right\rvert \leq B_* \Bigl[ \lVert \theta \rVert^{-2} N(\theta) \Bigr],
\] 
where $B_*$ is defined as 
\[ 
B_*(t) = \begin{cases} 
\displaystyle\frac{n^{-1/2} \zeta_*(\max 
\{ t, \sigma \} )}{1 - 4 \, n^{-1/2} \zeta_*( 
\max \{ t, \sigma \} )}, &  
\bigl[ 6 + (\kappa-1)^{-1} \bigr] \zeta_*( 
\max \{ t, \sigma \} ) \leq \sqrt{n}, \\ 
+ \infty, & \text{ otherwise,}
\end{cases} 
\] 
and
\begin{multline*}
\zeta_* (t) = \sqrt{ 2.032 (\kappa-1) \Biggl( \frac{0.73 \mathbb E [\mathbf{Tr}(A^2)]}{ \kappa^{1/2} \mathbb E  \bigl[ \lVert A \rVert_{\infty}^2\bigr]^{1/2} t} 
+ \log(K) + \log(\epsilon^{-1})  \Biggr)} \\ + \sqrt{ \frac{98.5 \kappa^{1/2} \mathbb E \bigl[ \lVert A \rVert_{\infty}^2 \bigr]^{1/2}}{ t}}.
\end{multline*}
\end{prop}

\vskip2mm
\noindent
As already discuss at the end of Proposition \ref{prop2}, if $a=1/2$ and $n\leq 10^{20}$, we can bound the logarithmic factor $\log(K)$ with the (small) constant $4.35$. 

\vskip2mm
\noindent
For the proof we refer to section \ref{pfsym1}.

\vskip2mm
\noindent
Remark that to obtain the above result we have used the fact that
\[
\mathbb E  \bigl[ \lVert A \theta \rVert^2 \bigr] \leq \mathbb E  \bigl[ \lVert A \rVert_{\infty}^2 \bigr]^{1/2} \kappa^{1/2} N(\theta). 
\]
However, if we use any upper bound of the form
\[
\mathbb E  \bigl[ \lVert A \theta \rVert^2 \bigr] \leq f(\mathbb E[A]) N(\theta)
\]
Propositon \ref{props1} holds replacing $\mathbb E  \bigl[ \lVert A \rVert_{\infty}^2 \bigr]^{1/2} \kappa^{1/2}$ with $f(\mathbb E[A])$ in the definition of $\zeta_*$.
Similarly we can replace $\mathbb E [\mathbf{Tr}(A^2)]$ by an upper bound.\\[1mm]
We observe in particular that
\[
\frac{\mathbb E  \bigl[ \lVert A \rVert_{\infty}^2 \bigr]^{1/2}}{\kappa^{1/2}} 
\leq \frac{ \mathbb E  \bigl[ \mathbf{Tr}(A^2) \bigr]}{\kappa^{1/2} \mathbb E  \bigl[
\lVert A \rVert_{\infty}^2 \bigr]^{1/2}} \leq 
\mathbb E  \bigl[ \mathbf{Tr}(A) \bigr] = \mathbf{Tr} \bigl[ \mathbb{E} 
( A ) \bigr].
\]
Indeed, to see the first inequality it is sufficient to observe that $\lVert A \rVert_{\infty}^2 \leq \mathbf{Tr}(A^2)$. 
Moreover we have that
\[ 
\mathbb E  \bigl[ \mathbf{Tr}(A^2) \bigr] \leq \mathbb E  \bigl[ \lVert A \rVert_{\infty} 
\mathbf{Tr}(A) \bigr] \leq \mathbb E  \bigl[ \lVert A \rVert_{\infty}^2 \bigr]^{1/2} 
\mathbb E  \bigl[ \mathbf{Tr}(A)^2 \bigr]^{1/2}, 
\] 
and, denoting by $\{ e_i\}_{i=1}^d$ an orthonormal basis of $\mathbb R^d,$
\begin{multline*}
\mathbb E  \bigl[ \mathbf{Tr}(A)^2 \bigr] = 
\sum_{ \substack{
1 \leq i \leq d, \\ 
1 \leq j \leq d
}} \mathbb E  \Bigl[ \lVert A^{1/2} e_i \rVert^2 \lVert A^{1/2} e_j \rVert^2 \Bigr] 
\\ \leq 
\sum_{ \substack{
1 \leq i \leq d, \\ 
1 \leq j \leq d
}} \mathbb E  \Bigl[ \lVert A^{1/2} e_i  \rVert^4 \Bigr]^{1/2} 
\mathbb E  \Bigl[ \lVert A^{1/2} e_j \rVert^4 \Bigr]^{1/2}  
\\ \leq \kappa \sum_{ \substack{
1 \leq i \leq d, \\ 
1 \leq j \leq d
}} \mathbb E  \Bigl[ \lVert A^{1/2} e_i \rVert^2 \Bigr] 
\mathbb E  \Bigl[ \lVert A^{1/2} e_j \rVert^2 \Bigr]  = \kappa \mathbb E  \bigl[
\mathbf{Tr}(A) \bigr]^2. 
\end{multline*}
This implies that we can bound $\zeta_*$ in Proposition \ref{props1} by
\begin{equation}\label{Azeta*}
\zeta_* (t) = \sqrt{ 2.032 \, (\kappa-1) \Biggl( \frac{0.73\ \mathbb E [\mathbf{Tr}(A)]}{ t} 
+ \log(K) + \log(\epsilon^{-1}) \Biggr)} + \sqrt{ \frac{98.5\ \kappa \ \mathbb E [\mathbf{Tr}(A)]}{ t}}.
\end{equation}

\vskip2mm
\noindent
We conclude this section observing that, since the entropy terms are dominated by $\mathbb E [\mathbf{Tr}(A)]$,
the result can be generalized to the case where $A$ is a random symmetric positive semi-definite operator 
in a infinite-dimensional Hilbert space with the only additional assumption that $\mathbb E [\mathbf{Tr}(A)] <+\infty.$ 

\subsection{Covariance matrix}\label{cov_matrix_sec}
Let $X\in \mathbb R^d$ be a random vector distributed according to the unknown probability measure $\mathrm P\in \mathcal M_+^1(\mathbb R^d).$ 
The covariance matrix of $X$ is defined as 
\[
\Sigma = \mathbb E \left[ (X- \mathbb E[X]) (X- \mathbb E[X])^{\top} \right]
\]
and our goal is to estimate, uniformly in $\theta,$ the  quadratic form 
\[
N(\theta) =\theta^{\top} \Sigma \, \theta = \mathbb E \left[ \langle \theta, X-\mathbb E\left[ X \right]  \rangle^2 \right], \quad \theta \in \mathbb R^d,
\]
from an i.i.d. sample $X_1,\dots, X_n \in \mathbb R^d$ drawn according to $\mathrm P.$ 
We cannot  use the results we have proved for the Gram matrix, 
since  the quadratic form $N$ depends on the unknown quantity $\mathbb E[X]$.
However we can find a workaround, using the results 
of the previous section about symmetric random matrices. 
Indeed, we do not need to estimate $\mathbb E[X]$ in order to estimate $N$
but it is sufficient to observe that the quadratic form $N$ can be written as
\[
N(\theta) = \frac{1}{2}\ \mathbb E \left[\langle \theta , X-X'\rangle^2  \right]
\]
where $X'$ is an independent copy of $X.$ 
More generally, given $q \in \mathbb N,$ we may consider $q$ independent copies $X^{(1)}, \dots, X^{(q)}$ of $X$ and 
the random matrix 
\[ 
A = \frac{1}{q(q-1)} \sum_{1 \leq j < k \leq q} 
\bigl( X^{(j)} - X^{(k)}\bigr) 
\bigl( X^{(j)} - X^{(k)}\bigr)^{\top}
\] 
so that 
\[
N(\theta) = 
\frac{1}{q(q-1)}\ \mathbb E \left[ \sum_{1\leq j < k \leq q} 
\langle \theta , X^{(j)}-X^{(k)} \rangle^2  \right] = 
\mathbb E  \bigl[ \, \theta^{\top} \!\! A \, \theta \, \bigr].
\]
We will discuss later how to choose $q$. 
In the following we use a robust block estimate which consists in 
dividing the sample $X_1, \dots, X_n$ in blocks of size $q$ and then in 
considering the original sample
as a "new" sample of $\lfloor n / q \rfloor$ symmetric matrices
$A_1, \dots, A_{\lfloor n / q \rfloor}$ 
(of independent copies of $A$) 
defined as 
\[
A_i = \frac{1}{q(q-1)} \sum_{(i-1) q < j < k \leq i q} 
(X_i - X_j) (X_i - X_j)^{\top}
\]
that thus correspond to the empirical covariance estimates on each block.
We can use the results of the previous section to define a robust 
estimator of $N(\theta)$. 
\\[1mm]
Let us introduce 
\begin{align*}
\kappa' & = \sup_{\substack{\theta \in \mathbb{R}^d, \\ 
\mathbb E (\lVert A^{1/2} \theta \rVert^2) > 0
}} \frac{\mathbb E  \bigl[ \lVert A^{1/2} \theta \rVert^4 \bigr]}{ 
\mathbb E  \bigl[ \lVert A^{1/2} \theta \rVert^2 \bigr]^2}, \\ 
\text{and } \kappa & = \sup_{\substack {\theta \in \mathbb{R}^d\\\mathbb E  \bigl[
\langle \theta, X - \mathbb E (X) \rangle^2 \bigr]>0}} \frac{\mathbb E  \bigl[ 
\langle \theta, X - \mathbb E (X) \rangle^4 \bigr]}{\mathbb E  \bigl[
\langle \theta, X - \mathbb E (X) \rangle^2 \bigr]^2}. 
\end{align*}
\begin{lem}\label{lem137}
The two kurtosis coefficients introduced above are related by the 
relation 
\[ 
\kappa' \leq 1 + \tau_q(\kappa)/q,
\] 
where $\displaystyle\tau_q(\kappa) = \kappa - 1 + \frac{2}{q-1}$. 
\end{lem}

\vskip2mm
\noindent
For the proof we refer to section \ref{pfkappa}.

\vskip2mm
\noindent
Let $\widehat{N}(\theta)$ be the estimator defined in equation \eqref{hatN_A}
and remark that 
\[ 
\mathbb{E} \bigl( \mathbf{Tr}(A) \bigr) = 
\mathbf{Tr} \bigl( \mathbb{E}(A) \bigr) 
= \mathbf{Tr}(\Sigma) = \mathbb{E} \Bigl( \lVert X - \mathbb{E}(X) \rVert^2 \Bigr). 
\]

\begin{prop}
For any energy level $\sigma \in] 0, \mathbf{Tr}(\Sigma)]$, with probability at least $1 - 
2 \epsilon$, for any $\theta \in \mathbb{R}^d$, 
\[ 
\left\lvert \frac{\max \{ N(\theta), \sigma \lVert \theta \rVert^2 \}}{
\max \{ \widehat{N}(\theta), \sigma \lVert \theta \rVert^2 \}} 
- 1 \right\rvert \leq B_* \Bigl( \lVert \theta \rVert^{-2} N(\theta) 
\Bigr),  
\] 
where 
\[ 
B_*(t) = 
\begin{cases} 
\displaystyle\frac{ \bigl( q \lfloor n / q \rfloor\bigr)^{1/2} \zeta_q \bigl( \max \{ 
t, \sigma \})}{1 - 4 \bigl( q \lfloor n / q \rfloor \bigr)^{1/2} 
\zeta_q( \max \{ t, \sigma \})}, & \text{if } \Bigl( 6 + q / \tau_q(\kappa) \Bigr) 
\zeta_q \bigl( \max \{ t, \sigma \} \bigr) \leq \bigl( q \lfloor 
n / q \rfloor \bigr)^{1/2},\\  
+ \infty, & \text{ otherwise} 
\end{cases}  
\] 
and
\[
\zeta_q(t) = \sqrt{ 2.032 \tau_q(\kappa) \biggl( \frac{0.73 \mathbf{Tr}(\Sigma) }{t } + \log(K) + \log(\epsilon^{-1}) \biggr)} + \sqrt{ \frac{98.5 \bigl( q + \tau_q(\kappa) \bigr) 
\mathbf{Tr}(\Sigma) }{t}} 
\]
with
\[
K = 1 + \left\lceil \frac{1}{2} \log \Biggl( \frac{ \lfloor n/q \rfloor}{72 (2 + c) 
\bigl( 1 + \tau_q(\kappa)/q \bigr)^{1/2}} 
\Biggr) \right\rceil.
\]
\end{prop}

\begin{proof}
The result follows from Proposition \ref{props1}, using the definition of $\zeta_*$, 
where we replace $\kappa$ by $\kappa'$ and $n$ by $\lfloor n / q \rfloor$. 
We conclude the proof according to Lemma \ref{lem137}.
\end{proof}

\vskip2mm
\noindent 
Here we have used the upper bound for the entropy factor defined 
in terms of $\mathbb E  \bigl[ \mathbf{Tr}(A) \bigr]=  \mathbf{Tr} (\Sigma)$,
as mentioned in the remarks 
following Proposition \ref{props1}. 
We can improve somehow the constants 
by evaluating more carefully $\mathbb E  \bigl[ \lVert A \theta \rVert^2 \bigr]$ 
and $\mathbb E  \bigl[ \mathbf{Tr}(A^2) \bigr]$ as shown in the next lemma proved in section \ref{pflem2_cov}.

\begin{lem}\label{lem2}
It holds true that
\begin{align}
\label{sol_eq1}
\mathbb E  \Bigl[ \lVert A \theta \rVert^2 \Bigr] & \leq \biggl( 1 - \frac{q-2}{q(q-1)} \biggr) \lVert \Sigma \rVert_{\infty} N(\theta) +  \frac{1}{q} \biggl( \kappa + \frac{1}{q-1} \biggr) \mathbf{Tr}(\Sigma) N(\theta) \\
\text{and } \quad \mathbb E  \Bigl[ \mathbf{Tr} \bigl( A^2 \bigr) \Bigr] & \leq \biggl( 1 - \frac{q-2}{q(q-1)} \biggr) \mathbf{Tr} \bigl( \Sigma^2 \bigr)  +  \frac{1}{q} \biggl( \kappa + \frac{1}{q-1} \biggr) \mathbf{Tr}(\Sigma)^2.
\end{align}
\end{lem}

\vskip2mm
\noindent
Using these tighter bounds, we can improve $\zeta_q$ to 
\begin{multline*}
\zeta_q(t) = \Biggl[ 2.032 \, \tau_q(\kappa) \Biggl( \frac{0.73 
\bigl[ \bigl( 1 - \frac{q-2}{q(q-1)} \bigr) \mathbf{Tr}(\Sigma^2) + \frac{1}{q} \bigl( 
\kappa + \frac{1}{q-1} \bigr) \mathbf{Tr}(\Sigma)^2 \bigr]}{\bigl[ \bigl( 1 -  \frac{q-2}{q(q-1)} \bigr) \lVert \Sigma \rVert_{\infty} 
+ \frac{1}{q} \bigl( \kappa + \frac{1}{q-1} \bigr) \mathbf{Tr}(\Sigma) \bigr] 
t} 
\\  \shoveright{ + \log(K) + \log(\epsilon^{-1}) \Biggr) \Biggr]^{1/2} }  \\ 
+ \sqrt{ \frac{98.5 \bigl[ q \bigl( 1 -  \frac{q-2}{q(q-1)} \bigr) \lVert \Sigma \rVert_{\infty} 
+ \bigl( \kappa + \frac{1}{q-1} \bigr) \mathbf{Tr}(\Sigma) \bigr] }{
t}}.
\end{multline*}
Therefore, in the case when  
\[ 
q \lVert \Sigma \rVert_{\infty} \leq \mathbf{Tr}(\Sigma),
\]
we have 
\[
\mathbb E  \Bigl[ \lVert A \theta \rVert^2 \Bigr]
\leq \frac{1}{q}\biggl( \kappa+1 + \frac{2}{q(q-1)} \biggr)  \mathbf{Tr}(\Sigma) N(\theta) 
\]
and hence, recalling that $\mathbf{Tr}(\Sigma^2) \leq \mathbf{Tr}(\Sigma)^2$, we can take
\[
\zeta_q(t) = \sqrt{2.032 \tau_q(\kappa) \biggl( \frac{0.73 \mathbf{Tr}(\Sigma)}{t} 
+ \log(K) + \log(\epsilon^{-1}) \biggr)} + 
\sqrt{\frac{98.5 \bigl( \kappa + 1 + \frac{2}{q(q-1)} \bigr) \mathbf{Tr}(\Sigma)}{t}}.
\]
If we compare the above result with the bound obtained in Proposition \ref{prop2} 
for the Gram matrix estimator,
we see that the first appearance of $\kappa$ 
in the definition of $\zeta_q$ has been replaced with  
\[
\tau_{q}(\kappa) + 1 = \kappa + \frac{2}{q-1}, 
\] 
and that the second appearance of $\kappa$ has been replaced with 
\[
\kappa + 1 + \frac{2}{q(q-1)}. 
\] 
Thus, when $\lVert \Sigma \rVert_{\infty} 
\leq \mathbf{Tr}(\Sigma)/ 2$, and this is not a very strong hypothesis, 
we can take at least $q = 2$, 
and obtain an improved bound for the estimation of $\Sigma$. 
This bound is not much larger than 
for the estimation of the centered Gram matrix, 
that we could have performed if we had known $\mathbb{E}(X)$, since the 
difference is just a matter of replacing $\kappa$ with $\kappa + 2$.

\newpage
\section{Proofs}

In this section we give the proofs of the results presented in the previous sections. 
More precisely, section \ref{proof_prop0} deals with Proposition \ref{prop0} (on page \pageref{prop0}), section \ref{proof_prop2} with Proposition \ref{prop2} (on page \pageref{prop2}) and section \ref{proof_propq} with Proposition \ref{propq} (on page \pageref{propq}),
some preliminary lemmas being postponed to section \ref{appx}. 

\subsection{Proof of Proposition \ref{prop0}}\label{proof_prop0}
The proof of Proposition \ref{prop0} requires a sequence of preliminary results. \\[1mm]
Our approach relies on perturbing the parameter $\theta$ with the Gaussian perturbation $\pi_{\theta} \sim \mathcal N(\theta, \beta^{-1} \mathrm I),$ 
where $\beta>0$ is a free parameter.

\begin{lem} 
\label{lem:7.1}
We have
\[
\int  \langle \theta', x \rangle^2 \, \mathrm{d} \pi_{\theta}(\theta') = \langle \theta, x \rangle^2 +\frac{\lVert x \rVert^2}{\beta}.
\]
\end{lem}

\begin{proof} Let $W \in\mathbb R^d$ be a random variable drawn according to $\pi_{\theta} \sim\mathcal N(\theta, \beta^{-1}\mathrm I).$ 
It follows that $\langle W, x\rangle$ is a one-dimensional Gaussian random variable with mean $\langle \theta, x \rangle$ and variance $\displaystyle x^{\top}(\beta^{-1}\mathrm I) x= \frac{\|x\|^2}{\beta}$. Consequently
\[
\int  \langle \theta', x \rangle^2 \, \mathrm{d} \pi_{\theta}(\theta') = \mathbb E[ \langle W, x\rangle]^2 + \mathbf{Var}[\langle W, x\rangle] = \langle \theta, x \rangle^2 +\frac{\lVert x \rVert^2}{\beta}.
\]
\end{proof}

\vskip2mm
\noindent
Accordingly we get
\begin{align*}
r_{\lambda}(\theta)  
=  \frac{1}{n} \sum_{i=1}^n \psi \biggl[ \int \biggl( \langle \theta', x \rangle^2 - \frac{\lVert x \rVert^2}{\beta} - \lambda \biggr) \, \mathrm{d} \pi_{\theta}(\theta') \biggr]. 
\end{align*}
In order to pull the expectation with respect to $\pi_{\theta}$ out of the influence function $\psi$, with a minimal loss of accuracy, we introduce the function 
\begin{equation}\label{defchi}
\chi(z) = \begin{cases} 
\psi(z) & z \leq z_1 \\
 \psi(z_1) + p_1(z-z_1) - (z-z_1)^2/8 & z_1 \leq z \leq z_1 + 4 p_1 \\
 \psi(z_1) + 2 p_1^2 & z \geq z_1 + 4 p_1
\end{cases} 
\end{equation}
where $z_1 \in [0,1]$ is such that $\psi'' (z_1) = -1/4$ and $p_1$ is defined by the condition $p_1 = \psi'(z_1).$
Using the explicit expression of the first and second derivative of $\psi$, we get
\begin{align*}
z_1 & = 1 - \sqrt{ 4 \sqrt{2} - 5}, \\ 
p_1  &= \psi'(z_1) = \frac{\sqrt{ 4 \sqrt{2} - 5}}{2 ( \sqrt{2} - 1)}
\end{align*} 
and $\ds \sup \chi =  \psi(z_1)+ 2 p_1^2 = - \log \bigl[ 2 ( \sqrt{2} - 1 )  \bigr] + \frac{1 + 2 \sqrt{2}}{2} .$

\begin{lem}
For any $z \in \mathbb R$,
\begin{equation}\label{bound_chi}
\psi(z) \leq \chi(z) \leq \log(1+z+z^2/2). 
\end{equation}
\end{lem}

\begin{proof} We first prove that $\psi(z) \leq \chi(z)$. 
The inequality is trivial for $z\leq z_1,$ since  $\chi(z) = \psi(z).$ For $z \in [ z_1, z_1+4p_1],$ performing a Taylor expansion at $z_1,$ we obtain that
\begin{align*} \psi(z) & = \psi(z_1) + p_1 (z-z_1) - \frac{1}{8}(z-z_1)^2 
+ \int_{z_1}^z \frac{\psi'''(u)}{2}(z-u)^2 \, \mathrm{d} u\\
& \leq \psi(z_1) + p_1 (z-z_1) -\frac{1}{8}(z-z_1)^2  = \chi(z),
\end{align*}
since $\psi'''(u) \leq0$ for $u \in [0,1[$. 
Finally we observe that, for any $z \geq z_1+4p_1$,
\[
\chi(z) = \psi(z_1) + 2p_1^2 > \log(2) \geq \psi(z).
\]
Let us now show that $\chi(z) \leq \log \bigl( 1 + z + z^2 / 2 \bigr).$
For $z \leq z_1$, we have already seen that the inequality is satisfied since $\chi(z) = \psi(z).$ 
Moreover we observe that the function
\[
f(z) = \log \bigl( 1 + z + z^2 / 2 \bigr)
\] 
is such that $f(z_1) \geq \chi(z_1)$ and also $f'(z_1) \geq \chi'(z_1).$ Performing a Taylor expansion at $z_1,$ we get
\begin{align*} f(z) & = f(z_1) + f'(z_1) (z-z_1)  + \int_{z_1}^z f''(u) (z-u)^2\mathrm d u\\
& \geq \chi(z_1) + \chi'(z_1) (z-z_1) + \inf f'' \ \frac{(z-z_1)^2}{2}.
\end{align*}
Since for any $t\in [z_1, z_1+4p_1]$, 
\[
\inf f'' = f''(\sqrt 3-1) = -1/4 = \chi''(t),
\]
we deduce that 
\[
f(z) \geq \chi(z_1) +p_1 (z-z_1) -\frac{1}{8}(z-z_1)^2 = \chi(z).
\]
In particular, $f( z_1+4p_1) \geq \chi( z_1+4p_1)$. 
Recalling that $f$ is an increasing function while $\chi$ is  
constant on  the interval $[z_1 + 4 p_1, + \infty[$, we conclude the proof.
\end{proof}

\vskip2mm
\noindent
Next lemma allows us to pull the expectation with respect to $\pi_{\theta}$ out of the function $\chi$.

\begin{lem}\label{l1}
Let $\Theta$ be a measurable space. 
For any $\rho \in \mathcal{M}_+^1(\Theta)$ and any $h \in L^1_{\rho}(\Theta)$, 
\begin{equation}\label{eq1.lem1}
\chi \biggl( \int h \, \mathrm{d} \rho \biggr) \leq \int \chi(h) \, \mathrm{d} \rho 
+ \frac{1}{8} \mathbf{Var} \bigl( h \, \mathrm{d} \rho \bigr), 
\end{equation} 
where by definition 
\[
\mathbf{Var} \bigl( h \, \mathrm{d} \rho \bigr) = \int \biggl( h(\theta)  - \int h \, \mathrm{d} \rho \biggr)^2
 \, \mathrm{d} \rho(\theta) \quad \in \mathbb{R} \cup \{+ \infty\}.
\]
Moreover, 
\[
\psi \biggl( \int h \, \mathrm{d} \rho \biggr) \leq \int \chi(h) \, \mathrm{d} \rho + \min \Bigl\{ \log(4), \frac{1}{8} \mathbf{Var} \bigl( h \, \mathrm{d} \rho \bigr) \Bigr\}. 
\]
\end{lem}

\begin{proof}
To prove equation \eqref{eq1.lem1} we observe that performing a Taylor expansion of the function $\chi$ at $z = \int h \ \mathrm d\rho$
\[
\chi \bigl[ h(\theta) \bigr]  \geq \chi(z) + \bigl( h(\theta) - z \bigr) \chi'(z)  +\inf \chi'' \  \frac{ \bigl( h(\theta) - z \bigr)^2}{2}, 
\]
so that, recalling that $\inf \chi'' = - 1 / 4$, we get
\begin{align*}
\int \chi \bigl[ h(\theta) \bigr] \, \mathrm{d} \rho( \theta) & \geq \chi \Bigl( \int h \, \mathrm{d} \rho\Bigr) - \frac{1}{8} \int \left( h(\theta) - \int h(\theta)  \mathrm{d} \rho \right)^2 \mathrm{d} \rho(\theta)\\
& = \chi \Bigl( \int h \, \mathrm{d} \rho\Bigr) - \frac{1}{8} \mathbf{Var} \bigl( h \, \mathrm{d} \rho \bigr).
\end{align*}
Combining equation \eqref{eq1.lem1} with the fact that $\psi(z) \leq \chi(z)$, for any $z \in \mathbb R$, we obtain that 
\[
\psi \biggl( \int h \, \mathrm{d} \rho \biggr) \leq  \int \chi(h) \, \mathrm{d} \rho+\frac{1}{8} \mathbf{Var} \bigl( h \, \mathrm{d} \rho \bigr).
\]
We conclude the proof by remarking that 
\[
\psi \biggl( \int h \, \mathrm{d} \rho \biggr) - \int \chi(h) \, \mathrm{d} \rho\leq \sup \psi - \inf \chi \leq  \log(4).
\]
\end{proof}

\vskip 2mm
\noindent
Applying this result to our problem we obtain
\begin{multline*}
\psi \bigl( \langle \theta, x \rangle^2 - \lambda \bigr) 
\leq \int \chi \biggl( \langle \theta', x \rangle^2 - \frac{\lVert x \rVert^2}{ \beta} - \lambda \biggr) \, \mathrm{d} \pi_{\theta}(\theta') 
\\ + \min \Bigl\{ \log(4),  \frac{1}{8} \mathbf{Var} \bigl[ \langle \theta', x \rangle^2 \, \mathrm{d} \pi_{\theta}(\theta') \bigr] \Bigr\},
\end{multline*}
where, putting $m = \langle \theta, x \rangle$, $\displaystyle\sigma = \frac{\lVert x \rVert}{\sqrt{\beta}}$ and denoting by $W \sim \mathcal{N}(0, \sigma^2)$ a centered Gaussian random variable, 
\begin{multline}
\label{eq:7.4}
\mathbf{Var} \bigl[ \langle \theta', x \rangle^2 \, \mathrm{d} 
\pi_{\theta}(\theta') \bigr] = \mathbf{Var} \bigl[ (m + W)^2 \bigr] \\
=4 m^2 \sigma^2 + 2 \sigma^4 
= \frac{4 \langle \theta, x \rangle^2 \lVert x \rVert^2}{ \beta} + \frac{2 \lVert x \rVert^4}{\beta^2}.
\end{multline}

\vskip2mm
\noindent
Let us remark that, for any $a,b,c \in \mathbb R_+$ and $W \sim \mathcal N(0,\sigma^2)$,
\begin{equation}\label{min_eq}
\min \bigl\{ a, b m^2 + c \bigr\} \leq \min \bigl\{ a , b ( m + W)^2 + c \bigr\} + \min \bigl\{ a, b ( m  - W)^2 + c \bigr\},
\end{equation}
since $b m^2 + c \leq \max \bigl\{ b ( m + W)^2 + c,  b ( m - W)^2 + c 
\bigr\}$. Therefore, taking the expectation with respect to $W$ of this 
inequality and remarking that $W$ and $-W$ have 
the same probability distribution we get 
\[ 
\min \bigl\{ a, b m^2 + c \bigr\} \leq 2 \mathbb E \Bigl[ \min \bigl\{ a, b (m+W)^2 + c \bigr\} \Bigr] .
\] 
Thus in our context we put  $a=\log(4)$, $b= \lVert x \rVert^2 / (2 \beta)$ and $c=\lVert x \rVert^4 / (4\beta^2)$ and we obtain 
\begin{multline*}
\psi \bigl( \langle \theta, x \rangle^2 - \lambda \bigr) 
\leq \int \chi \biggl( \langle \theta', x \rangle^2 - \frac{\lVert x \rVert^2}{\beta} - \lambda \biggr) \, \mathrm{d} \pi_{\theta}(\theta')
\\ + \int \min \biggl\{ 4 \log(2), \frac{\langle \theta', x \rangle^2 \lVert x \rVert^2 }{\beta} + \frac{\lVert x \rVert^4}{2 \beta^2} \biggr\} \, \mathrm{d} \pi_{\theta}(\theta').
\end{multline*}

\begin{lem}\label{lem1.5}
For any positive constants $b,y$ and any $z \in \mathbb{R}$, 
\[
\chi(z) + \min \{ b, y \} \leq \log \left( 1 + z + \frac{z^2}{ 2} + y  \exp( \sup \chi )\frac{ \bigl( \exp(b) - 1 \bigr) }{ b} \right). 
\]
\end{lem}
\begin{proof}
For any positive real constants $a, b, y$, 
\begin{align*}
\log(a) + \min \{ b, y \} 
&= \log \bigl( a \exp \bigl( \min \{ b, y\} \bigr) \bigr) \\ 
&\leq \log \biggl( a + a \min \{ b, y \} \frac{\bigl( \exp(b) - 1 \bigr)}{b}  \biggr),
\end{align*}
since the function $\displaystyle x \mapsto \frac{\exp(x)-1}{x}$ is non-decreasing for $x\geq0.$ It follows that
$$\log(a) + \min \{ b, y \}  \leq \log \bigl[ a + y a \bigl( \exp(b) - 1 \bigr) / b \bigr].$$

\noindent
Applying this inequality to $a = \exp \bigl[ \chi (z) \bigr]$ and reminding 
that $\chi(z) \leq \log\left(1 + z + z^2 / 2\right)$, we conclude the proof. 
\end{proof}

\vskip 2mm
\noindent
As a consequence, choosing 
$b = 4 \log(2),$ 
$z = \langle \theta', x \rangle^2 - {\lVert x \rVert^2}/{\beta} - \lambda$ and 
$y =  {\langle \theta', x \rangle^2 \lVert x \rVert^2 }/{\beta} + {\lVert x \rVert^4}/{2 \beta^2}$, 
we get
\begin{multline*}
\psi \bigl( \langle \theta, x \rangle^2 - \lambda \bigr) \leq 
\int \log \Biggl[  1 + \langle \theta' x \rangle^2 - \frac{\lVert x \rVert^2}{\beta} - \lambda + 
\frac{1}{2} \biggl( \langle \theta', x \rangle^2 - \frac{\lVert x \rVert^2}{
\beta} - \lambda \biggr)^2
\\ + \frac{c \lVert x \rVert^2}{\beta} \biggl( \langle \theta', x \rangle^2  
+ \frac{\lVert x \rVert^2}{2 \beta} \biggr) \Biggr] \, \mathrm{d} 
\pi_{\theta}(\theta') , 
\end{multline*}
where
$\displaystyle c = \frac{15}{8 \log(2)(\sqrt{2}-1)} \exp \biggl( \frac{1 + 2 \sqrt{2}}{2} \biggr) \leq 44.3.$\\[1mm]
We observe that the above inequality allows to compare 
$\psi \bigl( \langle \theta, x \rangle^2 - \lambda \bigr)$ 
with the expectation with respect to the Gaussian perturbation $\pi_{\theta}$. 
In terms of the empirical criterion $r_{\lambda}$ we have proved that
\begin{multline*}
r_{\lambda}(\theta) \leq \frac{1}{n} \sum_{i=1}^n \int 
\log \Biggl[ 1 + \langle \theta', X_i \rangle^2  - \frac{\lVert X_i \rVert^2}{
\beta} - \lambda + \frac{1}{2} \biggl( \langle \theta', X_i \rangle^2 - 
\frac{\lVert X_i \rVert^2}{\beta} - \lambda \biggr)^2 \\ 
+ \frac{c \lVert X_i \rVert^2}{\beta} \biggl(
\langle \theta', X_i \rangle^2 + \frac{ \lVert X_i \rVert^2}{2 \beta} 
\biggr) \Biggr] \, \mathrm{d} \pi_{\theta}(\theta').  
\end{multline*}
We are now ready to use 
the following general purpose PAC-Bayesian inequality.
\begin{prop}\label{pac}Let $\nu \in \mathcal M_+^1(\mathbb R^d)$ be a 
prior probability distribution on $\mathbb{R}^d$ and let $f: \mathbb{R}^d \times \mathbb{R}^d  
\to [a, +\infty]$ be a measurable function where $a \in \B{R}$. With probability at 
least $1-\epsilon,$ for any posterior distribution 
$\rho \in  \mathcal M_+^1(\mathbb{R}^d)$,
\[
\int \frac{1}{n} \sum_{i=1}^n f(X_i, \theta') \, \mathrm{d} \rho(\theta') 
\leq \int \log \Bigl\{ 
\mathbb E \bigl[ \exp \bigl( f(X , \theta') \bigr) \bigr] \Bigr\} 
\, \mathrm{d} \rho(\theta') + \frac{\mathcal{K}(\rho, \nu) + \log \bigl( 
\epsilon^{-1} \bigr)}{n}, 
\]
where 
\[
\mathcal{K} (\rho, \nu) = 
\begin{cases} 
\displaystyle\int \log\left(\frac{\mathrm{d} \rho}{\mathrm{d} \nu} \right) \, \mathrm{d} \rho, & \text{if } 
\rho \ll \nu, \\ +\infty, & \text{otherwise,} 
\end{cases}
\]
is the Kullback divergence of $\rho$ with respect to $\nu$. By convention, 
a non measurable event is said to happen with probability 
at least $1 - \epsilon$ when it includes a measurable event 
of probability non-smaller than $1 - \epsilon$. 
\end{prop}

\vskip 2mm
\noindent 
For the proof we refer to (page 1164 of) \cite{Catoni12}.

\vskip 2mm
\noindent
In our context, we consider as prior distribution $ \nu = \pi_0$
and  
we restrict the result to posterior distributions $\rho$ 
belonging to the family of Gaussian perturbations 
 \[
 \left\{ \pi_{\theta} \sim \mathcal N(\theta, \beta^{-1} \mathrm I) \ | \  \theta \in \mathbb R^d \right\}
 \]
 so that the Kullback divergence is given by 
\[ 
\mathcal{K}(\pi_{\theta}, \pi_0) = \frac{\beta \|\theta\|^2}{2}.
\]
We observe that since the result holds for any choice of the posterior, 
it allows us to obtain uniform results in $\theta.$
More precisely, we apply the above PAC-Bayesian inequality to 
\[
f(X_i, \theta') = \log \left[ 1 + t(X_i, \theta') + \frac{1}{2} t(X_i, \theta')^2+ \frac{c \lVert X_i \rVert^2}{\beta} \left( \langle \theta', X_i \rangle + \frac{\lVert X_i \rVert^2}{2 \beta} \right) \right],
\]
where  $\displaystyle t(x, \theta')= \langle \theta', x \rangle^2 - \frac{\|x\|^2}{\beta} - \lambda$. 
Using the fact that $\log(1+t) \leq t$, 
we get that, with probability at least $1-\epsilon,$ for any $\theta \in \mathbb R^d,$
\begin{align*}
r_{\lambda}(\theta) 
& \leq \int \mathbb E \left[ t(X, \theta')+\frac{1}{2} t(X, \theta')^2 + \frac{c \lVert X \rVert^2}{\beta} \left( \langle \theta', X \rangle^2 + \frac{\lVert X \rVert^2}{2 \beta} \right) \right] \mathrm d \pi_{\theta} (\theta')
\\ & \hspace{45mm} +  \frac{\beta \|\theta\|^2}{2n} + \frac{ \log(\epsilon^{-1})}{n}\\
& = \mathbb E \Biggl[ \langle  \theta, X \rangle^2 - \lambda+ \frac{1}{2} \biggl( \Bigl( \langle \theta, X \rangle^2 - \lambda \Bigr)^2 + \frac{4 \langle \theta, X \rangle^2 \lVert X \rVert^2}{\beta} + \frac{2 \lVert X \rVert^4}{\beta^2} \biggr) \\ 
& \hskip45mm  + \frac{c \lVert X \rVert^2}{\beta} \biggl( \langle \theta, X \rangle^2 + \frac{3 \lVert X \rVert^2}{2 \beta} \biggr) \, \Biggr] + \frac{\beta \lVert \theta \rVert^2}{2 n} + \frac{\log(\epsilon^{-1})}{n}.
\end{align*}
To obtain the last line we have used Lemma \ref{lem:7.1} and equation \myeq{eq:7.4}. 
\vskip2mm
\noindent
Let us recall the definition of $s_4$ and $\kappa$ introduced in equation \eqref{skappa}. We have defined 
\[
s_4  = \mathbb E \bigl[ \|X\|^4 \bigr]^{1/4}  \quad \text{and} \quad \kappa = \sup_{\substack{\theta \in \mathbb{R}^d \\ \mathbb E [ \langle \theta, X \rangle^2 ] > 0}} \frac{\displaystyle\mathbb E \bigl[ \langle \theta , X \rangle^4 \bigr]}{\mathbb E \bigl[ \langle \theta, X \rangle^2 \bigr]^2}.
\]
Using the Cauchy-Schwarz inequality 
and 
\begin{align*}
\mathbb E[\langle \theta, X\rangle^4] & \leq \kappa N(\theta)^2\\
\text{we deduce that } \mathbb E[\langle \theta, X\rangle^2 \|X\|^2] & \leq \kappa^{1/2} s_4^2 N(\theta),
\end{align*}
and we get that, with probability at least $1-\epsilon,$ for any $\theta \in \mathbb R^d,$
\begin{multline}
\label{eq4} 
r_{\lambda}(\theta) \leq 
\frac{\kappa}{2} \bigl[ N(\theta) - \lambda \bigr]^2 + \Biggl[
1 + (\kappa -1) \lambda + \frac{(2+c) \kappa^{1/2} s_4^2}{\beta} 
\, \Biggr] \bigl[ N(\theta) - \lambda \bigr] \\ 
+ \frac{(\kappa-1) \lambda^2}{2} + 
\frac{(2+c) \kappa^{1/2} s_4^2 \lambda}{\beta} + 
\frac{(2 + 3c) s_4^4}{2 \beta^2} + \frac{\beta \lVert \theta \rVert^2}{2n} 
+ \frac{\log(\epsilon^{-1})}{n}. 
\end{multline}

\noindent
According to the (compact) notation introduced in equation \myeq{c_notat}, the above inequality rewrites as
\begin{equation}\label{up}
\frac{r_{\lambda}(\theta)}{\lambda} \leq 
\xi \Biggl( \frac{N(\theta)}{\lambda} - 1 \biggr)^2 + (1 + \mu) \Biggl( \frac{N(\theta)}{\lambda} - 1 \Biggr) + \gamma + \delta \lVert \theta \rVert^2.
\end{equation}
Similarly, observing that 
\[
- r_{\lambda}(\theta) = \frac{1}{n} \sum_{i=1}^n \psi \bigl( \lambda - \langle \theta, X_i \rangle^2 \bigr),
\]
we obtain a lower bound for the empirical criterion $r_{\lambda}.$ Namely, with probability at least $1 - \epsilon$, for any $\theta \in \mathbb{R}^d$, any $(\lambda, \beta) \in \Lambda$, 
\begin{equation}\label{low}
\frac{r_{\lambda}(\theta)}{\lambda} \geq - \xi \Biggl( \, \frac{N(\theta)}{\lambda} - 1 \Biggr)^2 + (1 - \mu) \Biggl( \, \frac{N(\theta)}{\lambda} - 1 \Biggr) - \gamma - \delta \lVert \theta \rVert^2.
\end{equation}
We now combine the two bounds above to get the confidence region for $N(\theta)$ defined in Proposition \ref{prop0}. 
Assume that both equation \eqref{up} and equation \eqref{low} hold for any 
$\theta \in \mathbb R^d$, an event that happens with probability at least $1-2\epsilon.$\\[1mm]
Let us introduce $\displaystyle\tau(\theta) = \frac{\lambda \delta \lVert \theta \rVert^2}{N(\theta)}$ and 
\[ 
p_{\theta}(z) = - \xi z^2 + \bigl[ 1 - \mu - \tau(\theta) \bigr] \, z 
- \gamma - \tau(\theta), \qquad z \in \mathbb{R}.
\]  
We observe that $\tau( \alpha \theta ) = \tau(\theta)$, 
and consequently  $p_{\alpha \theta}(z) = p_{\theta}(z)$ for any $\alpha \in \mathbb{R}_+$.
We consider the case when 
$p_{\theta}(1) > 0$, meaning that 
\begin{equation}
\label{eq14}
\xi + \mu + \gamma + 2 \tau(\theta) < 1.
\end{equation}
In this case, the second degree polynomial $p_{\theta}$ has 
two distinct real roots, $z_{-1}$ and $z_{+1}$, where 
\[ 
z_{\sigma} = \frac{ 1 - \mu - \tau(\theta) + \sigma \sqrt{  \bigl[ 1 - \mu - \tau(\theta) \bigr]^2 - 4 \xi \bigl[ \gamma + \tau(\theta) \bigr]}}{2 \xi}, \qquad \sigma\in \{1,- 1\}. 
\] 
Since equation \eqref{eq14} can also be written as 
$-\xi > -\bigl[ 1 - \mu - \gamma - 2 \tau(\theta) \bigr] $, 
we get
\begin{multline*}
p_{\theta} \biggl( \frac{\gamma + \tau(\theta)}{1 - \mu - \gamma - 2 \tau(\theta)} \biggr) \\
> \frac{-  \bigl[\gamma + \tau(\theta) \bigr]^2}{ 1 - \mu - \gamma - 2 \tau(\theta)} + \bigl[ 1 - \mu - \tau(\theta) \bigr] \frac{\gamma + \tau(\theta)}{1 - \mu - \gamma - 2 \tau(\theta)} - \gamma - \tau(\theta) = 0, 
\end{multline*}
which implies
\[
z_{-1} < \frac{\gamma + \tau(\theta)}{1 - \mu - \gamma - 2 \tau(\theta)} < z_{+1}.
\]
Therefore, since according to equation \eqref{low}, for any $\alpha \in [0, \widehat{\alpha}(\theta)]$, 
\[
p_{\theta} \left( \frac{\alpha^2 N(\theta)}{ \lambda} - 1\right) \leq \frac{r_{\lambda}(\alpha \theta) }{ \lambda }\leq 0,
\]
it is true that
\[
\left[-1, \frac{\widehat{\alpha}(\theta)^2 N(\theta)}{\lambda} - 1 \right] \cap \bigl] z_{-1}, z_{+1} \bigr[ =  \emptyset. 
\] 
Observing that $z_{-1} \geq 0 > -1$, it follows that $\widehat{\alpha}(\theta)^2 N(\theta) / \lambda - 1 \leq z_{-1}.$
This proves that, for any $\theta \in \mathbb{R}^d$ satisfying equation \eqref{eq14}, 
\[ 
N(\theta) \leq \frac{\lambda}{\widehat{\alpha}(\theta)^2} \bigl( 1 + z_{-1} \bigr) 
\leq \frac{\lambda}{\widehat{\alpha}(\theta)^2} \biggl( 1 + \frac{\gamma + \tau(\theta)}{1 - \mu - \gamma - 2 \tau(\theta)} \biggr),
\] 
which rewrites as
\[
\Phi_{\theta,+} \bigl[ N(\theta) \bigr] \leq \frac{\lambda}{\widehat{\alpha}(\theta)^2}.
\]
Moreover, this inequality is trivially true when condition \eqref{eq14}
is not satisfied, because its left-hand side is equal to zero and 
its right-hand side is non-negative. \\[1mm]
Proving the second part of the proposition requires a new argument
and not a mere update of signs in the proof of the first part. 
Although it may seem at first sight that we are just aiming at 
a reverse inequality, the situation is more subtle than that.

\noindent
Let us first remark that in the case when 
\begin{equation}
\label{eq15}
\xi - \mu + \gamma + \delta \widehat{\alpha}(\theta)^2 \lVert \theta \rVert^2 < 1,
\end{equation}
is not satisfied,
the bound
\begin{equation}
\label{phi-bound}
\Phi_{\theta,-} \left( \frac{\lambda}{\widehat{\alpha}(\theta)^2} \right) \leq N(\theta)
\end{equation} 
is trivially satisfied because the left-hand size is equal to zero.
In the case when equation \eqref{eq15} is true, it is also true that $\widehat{\alpha}(\theta) < + \infty$, 
so that $r_{\lambda} \bigl[ \widehat{\alpha}(\theta) \theta \bigr] = 0$, 
and therefore, according to equation \eqref{up}, 
\[ 
0 \leq q_{\widehat{\alpha}(\theta) \theta} \Biggl( \frac{\widehat{\alpha}(\theta)^2 
N(\theta)}{\lambda} - 1 \Biggr), 
\] 
where 
$q_{\theta}(z) = \xi z^2 + (1 + \mu) z + \gamma + \delta \lVert \theta \rVert^2. $\\[1mm]
Since condition \eqref{eq15} can also be written as $q_{\widehat{\alpha}(\theta)
\theta}(-1) < 0$, it implies that the second order polynomial 
$q_{\widehat{\alpha}(\theta) \theta}$ has two roots and that $\frac{\widehat{\alpha}(\theta)^2 N(\theta)}{\lambda} - 1$ is on the right of its largest root, which is larger than $-1$.
On the other hand, we observe that, under condition \eqref{eq15},  
putting $\widehat{\tau}(\theta) = \delta \widehat{\alpha}(\theta)^2 \lVert \theta \rVert^2$, we get 

\[
q_{\widehat{\alpha}(\theta) \theta} \biggl( - \frac{ \gamma + \widehat{\tau}(\theta)}{1 + \mu - \gamma - \widehat{\tau}(\theta)} \biggr) 
 < \frac{ \bigl( \gamma + \widehat{\tau}(\theta) \bigr)^2}{1 + \mu - \gamma - \widehat{\tau}(\theta)} - \frac{ (1 + \mu) \bigl[ \gamma + \widehat{\tau}(\theta) \bigr]}{1 + \mu - \gamma + \widehat{\tau}(\theta)} + \gamma + \widehat{\tau}(\theta) 
= 0.
\]

\noindent
Therefore, when condition \eqref{eq15} is satisfied, 
\[ 
\frac{\widehat{\alpha}(\theta)^2 N(\theta)}{\lambda} - 1 \geq - \frac{\gamma + \widehat{\tau}(\theta)}{1 + \mu - \gamma - \widehat{\tau}(\theta)},
\] 
which rewrites as equation \eqref{phi-bound}.

\subsection{Proof of Proposition \ref{prop2}}\label{proof_prop2}
We first observe that, according to Proposition \ref{prop1}, with probability at least $1 - 2 \epsilon$, 

\begin{align*}
\left| \frac{\max\{N(\theta),\sigma\|\theta\|^2\}}{\max\{ \widehat N(\theta), \sigma\|\theta\|^2\} } - 1\right| &\leq B_{\widehat \lambda, \widehat \beta}\left[ \|\theta\|^{-2} \widehat{N} (\theta) \right]
 = \inf_{(\lambda, \beta) \in \Lambda} B_{\lambda, \beta} \left[   \|\theta\|^{-2}\widetilde N_{\lambda}(\theta) \right]
\end{align*}
since, by definition, $(\widehat \lambda, \widehat \beta)$ are the values which minimize $B_{\lambda, \beta} \Bigl[\|\theta\|^{-2} \widetilde{N}_{\lambda} (\theta) \Bigr].$
According to equation \myeq{phi+eq}, since $\Phi_+\left( \|\theta\|^{-2} N(\theta) \right) \leq \|\theta\|^{-2} \widetilde N(\theta)$
and $B_{\lambda, \beta}$ is a decreasing function,
we get 
\[ 
\Biggl\lvert \,  \frac{\max \bigl\{ N(\theta), \sigma \lVert \theta \rVert^2 
\bigr\}}{\max \bigl\{ 
\widehat{N}(\theta), \sigma \lVert \theta \rVert^2 \bigr\}} - 1 \Biggr\rvert \leq 
\inf_{(\lambda, \beta) \in \Lambda} B_{\lambda, \beta} \Biggl( \;  \frac{ \lVert \theta \rVert^{-2} 
N(\theta) }{ 1 + B_{\lambda, \beta} \Bigl[ \lVert \theta \rVert^{-2} 
N(\theta) \Bigr]} \Biggr).
\] 
With the same notation as in Proposition \ref{prop1}, we introduce the subset $\Gamma'$ of $\Gamma$ defined as
\[
\Gamma' = \Bigl\{ (\lambda, \beta, t ) \in \Lambda \times \mathbb{R}_+ \ | \  
\begin{aligned}[t]
\xi + \mu + \gamma + 4 \delta \lambda / \max \{ t, \sigma \} & < 1, \\
\mu + \gamma + 2 \delta \lambda / \max \{ t, \sigma \} & \leq 1/2, \\
\text{ \rm and }   2 \gamma + \delta \lambda / \max \{ t,  \sigma \}  & \leq 1/2 ~~ \Bigr\}
\end{aligned}
\]
and the function 
\begin{equation}\label{Btilde}
\widetilde{B}_{\lambda, \beta}(t) = \begin{cases} \displaystyle\frac{\gamma + 
\lambda \delta / \max\{ t, \sigma \}}{ 1 - \mu -2 \gamma 
- 4 \lambda \delta / \max\{ t, \sigma \}}, & (\lambda, \beta, t) \in \Gamma', \\ 
+ \infty, & \text{ otherwise.} 
\end{cases}
\end{equation}

\begin{lem}
\label{lem1.14}
For any $(\lambda, \beta) \in \Lambda$ and any $t \in \mathbb{R}_+$, 
\begin{align*}
B_{\lambda, \beta} \Biggl( \frac{t}{1 + B_{\lambda, \beta}(t)} \Biggr) 
& \leq \widetilde{B}_{\lambda, \beta}(t), \\ 
\frac{B_{\lambda, \beta}(t)}{1 - B_{\lambda, \beta}(t)} & \leq \widetilde{B}_{\lambda, 
\beta}(t).
\end{align*}
\end{lem}

\begin{proof}
We first observe that when $(\lambda, \beta, t) \not \in \Gamma'$ then $\widetilde{B}_{\lambda, \beta}(t) = + \infty$ and hence the two inequalities are trivial. We now assume $(\lambda, \beta, t) \in \Gamma'$ and we put  $\tau = \lambda \delta / \max \{ t, \sigma \}.$ We prove the second inequality first. Since $\Gamma' \subset \Gamma$, we have 
\[ 
\frac{B_{\lambda, \beta}(t)}{1 - B_{\lambda, \beta}(t)} = \frac{\gamma + \tau}{1 - \mu - 2 \gamma - 3 \tau} \leq \widetilde{B}_{\lambda, \beta}(t).
\] 

\noindent
In order to prove the first inequality, we first check that $\left(\lambda, \beta,  \frac{t }{ 1 + B_{\lambda, \beta}(t) } \right)\in \Gamma.$
We start observing that, since 
\[
\max \left\{ t / [ 1 + B_{\lambda, \beta}(t) ], \sigma \right\} \geq \max \{ t, \sigma \} / [1 + B_{\lambda, \beta}(t) ],
\]
then
\begin{align*}
\xi + \mu + \gamma + 2 \delta \lambda / \max \left\{\frac{ t}{ 1 + B_{\lambda, \beta}(t) } , \sigma \right \} 
& \leq \xi + \mu + \gamma + 2 [1 + B_{\lambda, \beta}(t)]  \frac{\delta \lambda  }{ \max \{ t, \sigma \} } \\
& = \xi + \mu + \gamma + 2 \tau + 2 \tau B_{\lambda, \beta}(t).
\end{align*}

\noindent
Moreover, as $\left(\lambda, \beta, t \right) \in \Gamma'$, we get
$$B_{\lambda, \beta}(t) = \frac{\gamma+\tau}{ 1- \mu-\gamma-2\tau} \leq 1,$$
so that
$$\xi + \mu + \gamma + 2 \delta \lambda / \max \{ t / [ 1 + B_{\lambda, \beta}(t) ] , \sigma \}  \leq \xi + \mu + \gamma + 4 \tau < 1, $$
which proves that, indeed, $\left(\lambda, \beta,  \frac{t }{ 1 + B_{\lambda, \beta}(t) } \right) \in \Gamma.$
Therefore

\begin{align*}
B_{\lambda, \beta} \Biggl( \frac{t}{1 + B_{\lambda, \beta}(t)} \Biggr) 
& \leq \frac{\gamma + \tau \left( 1+  B_{\lambda, \beta}(t) \right)}{1 - \mu - \gamma - 2 \tau [1+ \tau B_{\lambda, \beta}(t)]} \\
& = \frac{\left(\gamma + \tau \right) \left( 1+  \tau / (1 - \mu - \gamma - 2 \tau) \right)}{1 - \mu - \gamma - 2 \tau - 2 \tau B_{\lambda, \beta}(t)} ,
\end{align*}

\noindent
where in the last line we have used the definition of $B_{\lambda, \beta}.$ Observing that 
$$ 1 - \mu - \gamma - 2 \tau - 2 \tau B_{\lambda, \beta}(t) =  \frac{(1 - \mu - \gamma - 2 \tau)^2-2 \tau \left(\gamma + \tau \right) }{1 - \mu - \gamma - 2 \tau},$$

\noindent
we obtain 
\begin{align*}
B_{\lambda, \beta} \Biggl( \frac{t}{1 + B_{\lambda, \beta}(t)} \Biggr) 
& \leq \frac{\left(\gamma + \tau \right) \left( 1 - \mu - \gamma -  \tau \right)}{(1 - \mu - \gamma - 2 \tau)^2-2 \tau \left(\gamma + \tau \right) }\\
& =  \frac{\left(\gamma + \tau \right) \left( 1 - \mu - \gamma -  \tau \right)}{(1 - \mu - \gamma - \tau)^2 + \tau^2 - 2 \tau(1 - \mu - \gamma - \tau) - 2 \tau^2 - 2 \gamma \tau}\\
& =  \frac{\gamma + \tau}{1 - \mu - \gamma - \tau  - 2 \tau -\left(  \tau^2 + 2 \gamma \tau\right)/ \left( 1 - \mu - \gamma -  \tau \right)}.
\end{align*}

\noindent
Considering that 
\begin{align*}
\left( \tau^2 + 2 \gamma \tau\right)/ \left( 1 - \mu - \gamma -  \tau \right)  \leq \tau,
\end{align*}

\noindent
since when $(\lambda, \beta, t) \in \Gamma'$, it is true that $1 - \mu - \gamma - \tau \geq 1/2$ and $ 2\gamma + \tau \leq 1/2,$ we conclude that 
\[
B_{\lambda, \beta} \Biggl( \frac{t}{1 + B_{\lambda, \beta}(t)} \Biggr)  \leq \frac{\gamma + \tau}{1 - \mu - \gamma - 4\tau } = \widetilde{B}_{\lambda, \beta} (t).
\]
\end{proof}

\vskip2mm
\noindent
Applying the above lemma to our problem we get that, with probability at least $1 - 2 \epsilon$, 
for any $\theta \in \mathbb{R}^d$, 
\begin{equation}\label{boundwt}
\left\lvert \, \frac{\max \{ N(\theta), \sigma \lVert \theta \rVert^2 \}}{\max \{ \widehat{N}(\theta), \sigma \lVert \theta \rVert^2 \}} - 1 \, \right\rvert \leq \inf_{(\lambda, \beta) \in \Lambda} \widetilde{B}_{\lambda, \beta} \bigl[ \lVert \theta \rVert^{-2} N(\theta) \bigr].
\end{equation}

\vskip 2mm
\noindent
Let us recall the definition of the finite set $\Lambda$ given in \myeq{defLambda}. 
Let $a>0$ and 
\[ 
 K  = 1 + \left\lceil a^{-1} \log \biggl( \frac{n}{72(2+c) \kappa^{1/2}} \biggr) \right\rceil.
\] 
We define
\begin{align*}
\Lambda & = \bigl\{ (\lambda_j, \beta_j) \ | \ 0 \leq j < K \bigr\}, \\
\text{where } \quad \lambda_j & = \sqrt{ \frac{2}{n (\kappa - 1)} \Biggl(\frac{(2 + 3c)}{4(2 + c) \kappa^{1/2} \exp (-j a) }  + \log(K / \epsilon) \Biggr)} \\ 
\text{and } \quad \beta_j & = \sqrt{2 (2 + c) \kappa^{1/2} s_4^4 n \exp \bigl[  - (j-1/2) a \bigr]}. 
\end{align*}

\vskip1mm
\noindent
We introduce the explicit bound 
\[
\zeta (t) = \sqrt{ 2 (\kappa-1) \Biggl( \frac{(2 + 3c) s_4^2}{4(2+c) \kappa^{1/2} t} + \log(K / \epsilon) \Biggr)} \cosh(a/4) + \sqrt{ \frac{2(2+c)\kappa^{1/2} s_4^2}{ t}} \cosh ( a/ 2)
\]
and 
\[ 
B_*(t) = \begin{cases} 
\displaystyle\frac{n^{-1/2} \zeta(\max \{ t, \sigma \} )}{1 - 4 \, n^{-1/2} \zeta( \max \{ t, \sigma \} )} &  \bigl[ 6 + (\kappa-1)^{-1} \bigr] \zeta( \max \{ t, \sigma \} ) \leq \sqrt{n} \\ 
+ \infty & \text{ otherwise.}
\end{cases} 
\]

\begin{lem}
For any $t \in \mathbb{R}_+$, we have
\begin{equation}
\label{eq1.17}
\inf_{(\lambda, \beta) \in \Lambda} \widetilde{B}_{\lambda, \beta}(t) \leq B_*
\bigl( \min \{ t, s_4^2 \} \bigr).
\end{equation}
\end{lem}

\begin{proof} We recall that the function $\widetilde B_{\lambda, \beta}$ is non-increasing so that 
\[
\widetilde B_{\lambda, \beta}(t) \leq \widetilde 
B_{\lambda, \beta} \bigl( \min\{t,s_4^2\} \bigr).
\]
Moreover, since
\[
\max\bigl\{  \min \{ t, s_4^2 \}, \sigma \bigr\} =  \min\bigl\{ \max\{ t, \sigma\},  s_4^2\bigr\},
\]
it is sufficient to prove the result for $ \max\{ t, \sigma\}\in [0,s_4^2]$.\\[1mm]
As equation \eqref{eq1.17} is trivial when $B_*(t) = + \infty,$
we may assume that $B_*(t) < + \infty$,  so that 
$6 \, \zeta(\max\{t, \sigma\}) \leq \sqrt n.$ In particular, by considering only the second term in 
the definition of $\zeta,$ we obtain that 
\[ 
\sqrt{ \frac{2(2+c)\kappa^{1/2} s_4^2}{ \max \{ t, \sigma \}}}  \leq \sqrt{ \frac{2(2+c)\kappa^{1/2} s_4^2}{ \max \{ t, \sigma \} }} \cosh ( a/ 2)\leq \frac{\sqrt n}{6}, 
\]
which implies
\[ 
\frac{\max \{ t, \sigma \} }{s_4^2} \geq \frac{72 (2 + c) \kappa^{1/2}}{n} \geq \exp\left(  -a (K -1) \right). 
\] 
Therefore, since
\[
\log\left(\frac{\max \{ t, \sigma \}}{s_4^2}\right) \in \Bigl[ - a(K-1), 0 \Bigr],
\]
there exists $\widehat{\jmath} \in \{ 0 , \dots,  K-1\}$
for which
\begin{equation}
\label{eqconsut}
\Biggl\lvert \log \biggl( \frac{\max \{ t, \sigma \}}{s_4^2}  \biggr) + \widehat{\jmath} a \Biggr\rvert \leq a / 2.
\end{equation}
We recall that by equation \myeq{c_notat},
\[ 
\gamma + \delta \lambda / \max \{ t , \sigma \}  = \frac{\lambda}{2}(\kappa-1) + \frac{(2+c) \kappa^{1/2} s_4^2}{\beta} + \frac{(2 + 3 c) s_4^4}{2 \beta^2 \lambda} + \frac{\log(K / \epsilon)}{n 
\lambda} + \frac{\beta}{2 n \max \{ t , \sigma \}}
\] 
and we observe that $(\lambda_*, \beta_*)$ defined as
\begin{align}
\label{eq1.33}
\lambda_* & = \sqrt{ \frac{2}{n (\kappa - 1)} 
\Biggl(  
\frac{(2 + 3c) s_4^2}{4 (2 + c) k^{1/2} \max \{ t, \sigma \}} + \log(K / \epsilon)
\Biggr)} \\ 
\beta_* & = \sqrt{ 2(2+c) k^{1/2} s_4^2 \max \{ t, \sigma \} n }  
\end{align}
are the desired values that optimize $\gamma + \delta \lambda / \max \{ t ,\sigma \} .$ We also remark that, by equation \eqref{eqconsut},
\begin{align}
\label{eqbj}
\beta_{\widehat{\jmath}} \, \exp( - a/2) \leq & \beta_* \leq \beta_{\widehat{\jmath}} \\ 
\label{eq33}
\lambda_{\widehat{\jmath}} \, \exp( - a / 4) \leq & \lambda_* \leq 
\lambda_{\widehat{\jmath}} \, \exp ( a / 4). 
\end{align}

\noindent
Thus, evaluating $\gamma + \delta \lambda / \max \{ t , \sigma \} $ in $(\lambda_{\widehat{\jmath}}, \beta_{\widehat{\jmath}}) \in \Lambda,$ we obtain  that

\begin{multline*}
\gamma_{\widehat{\jmath}} + \delta_{\widehat{\jmath}} \lambda_{\widehat{\jmath}} / \max \{ t , \sigma \}   \\
= \frac{\lambda_*(\kappa-1)}{2} \frac{\lambda_{\widehat{\jmath}}}{\lambda_*}+ 
\frac{(2+c) \kappa^{1/2} s_4^2}{\beta_*} \frac{\beta_*}{\beta_{\widehat{\jmath}}} 
+ \frac{(2 + 3 c) s_4^4}{2 \, \beta_{\widehat{\jmath}}^2 \, \lambda_*}  
\frac{\lambda_*}{\lambda_{\widehat{\jmath}}}+ \frac{\log(K / \epsilon)}{n\lambda_*} 
\frac{\lambda_*}{\lambda_{\widehat{\jmath}}}+ \frac{\beta_*}{2 n \max \{ t 
,\sigma \}} \frac{\beta_{\widehat{\jmath}}}{\beta_*}\\
 \leq  \frac{\lambda_*(\kappa-1)}{2} \frac{\lambda_{\widehat{\jmath}}}{\lambda_*}
+ \frac{1}{n \lambda_*} \left[\frac{(2 + 3c) s_4^2}{4 (2 + c) k^{1/2} 
\max \{ t, \sigma \}} + \log(K / \epsilon) \right] \frac{\lambda_*}{
\lambda_{\widehat{\jmath}}} \\
 \shoveright{ + \sqrt{\frac{(2+c) \kappa^{1/2} s_4^2}{2 n \max \{ t, \sigma \}}} \left(  \frac{\beta_*}{\beta_{\widehat{\jmath}}} + \frac{\beta_{\widehat{\jmath}}}{\beta_*}\right)} \\
 \leq \sqrt{\frac{2(\kappa-1)}{n} \left[\frac{(2 + 3c) s_4^2}{4 (2 + c) k^{1/2} \max \{ t, \sigma \}} + \log(K / \epsilon) \right]} \cosh \Biggl[ \log\left( \frac{\lambda_{\widehat{\jmath}}}{\lambda_*} \right) \Biggr]  \\
 + \sqrt{\frac{2 (2+c) \kappa^{1/2} s_4^2}{n \max \{ t, \sigma\}}} \cosh 
\Biggl[ \log\left( \frac{\beta_{\widehat{\jmath}}}{\beta_*}\right) \Biggr].
\end{multline*}

\noindent
By equation \eqref{eqbj} we get
\begin{multline*}
\gamma_{\widehat{\jmath}} + \delta_{\widehat{\jmath}} \lambda_{\widehat{\jmath}} / 
\max \{ t , \sigma \}   \\ 
\leq \sqrt{\frac{2(\kappa-1)}{n} \left[\frac{(2 + 3c) s_4^2}{4 (2 + c) k^{1/2} \max \{ t, \sigma \}} + \log(K / \epsilon) \right]} \cosh \left( \frac{a}{4} \right) \\
 + \sqrt{\frac{2 (2+c) \kappa^{1/2} s_4^2}{n \max \{ t, \sigma \}}} \cosh \left( \frac{a}{2}\right)  .
\end{multline*}
We also observe that
$$\mu_{\widehat{\jmath}} + \gamma_{\widehat{\jmath}} + 4 \delta_{\widehat{\jmath}} \lambda_{\widehat{\jmath}} / \max \{ t, \sigma \} \leq  4 \bigl[ \gamma_{\widehat{\jmath}} + \delta_{\widehat{\jmath}} \lambda_{\widehat{\jmath}} / \max \{ t, \sigma \} \bigr] \leq 4n^{-1/2} \zeta(t),$$  
since by definition  $\mu_{\widehat{\jmath}} \leq 2 \gamma_{\widehat{\jmath}}.$
In the same way, observing that
$$ \xi_ {\widehat{\jmath}} = \frac{\kappa\lambda_{\widehat{\jmath}}}{2} \leq \gamma_{\widehat{\jmath}} \left(1+\frac{1}{\kappa-1}\right)$$
we obtain 
\begin{align*}
\xi_{\widehat{\jmath}} + \mu_{\widehat{\jmath}} + \gamma_{\widehat{\jmath}} + 4 \delta_{\widehat{\jmath}} \lambda_{\widehat{\jmath}} / \max \{ t, \sigma \} & <  \bigl[ 4 + (\kappa -1)^{-1} \bigr] \gamma_{\widehat{\jmath}} + 4 \delta_{\widehat{\jmath}} \lambda_{\widehat{\jmath}} / \max \{ t ,\sigma\} \\
&\leq  \bigl[6 + (\kappa-1)^{-1} \bigr]n^{-1/2}  \zeta( \max \{ t ,\sigma\})
\end{align*}

\noindent
and similarly,
\begin{align*}
2 \bigl[ \mu_{\widehat{\jmath}} + \gamma_{\widehat{\jmath}} + 2 \delta_{\widehat{\jmath}} \lambda_{\widehat{\jmath}} / \max \{ t, \sigma\} \bigr] & \leq 6n^{-1/2}  \zeta( \max \{ t ,\sigma\}), \\ 
2 \bigl[  2\gamma_{\widehat{\jmath}} + \delta_{\widehat{\jmath}} \lambda_{\widehat{\jmath}} / \max \{ t, \sigma \} \bigr] & \leq 4 n^{-1/2} \zeta( \max \{ t ,\sigma\}).
\end{align*}
This implies that, whenever $B_*(t) < + \infty$, then $(\lambda_{\widehat{\jmath}}, \beta_{\widehat{\jmath}}, t) \in \Gamma'.$ We have then proved that
\[ 
\inf_{(\lambda, \beta) \in \Lambda} \widetilde{B}_{\lambda, \beta}(t) 
\leq \widetilde{B}_{\displaystyle\lambda_{\widehat{\jmath}}, \beta_{\widehat{\jmath}}} (t) 
\leq \frac{n^{-1/2}\zeta( \max \{ t ,\sigma\})}{1 - 4 n^{-1/2}\zeta( \max \{ t ,\sigma\})} =  B_*(t). 
\] 
\end{proof}

\vskip 2mm
\noindent
Applying the above lemma to equation \myeq{boundwt} and observing that, for any $\theta \in \mathbb{R}^d$,  
\[ 
\lVert \theta \rVert^{-2} N(\theta) \leq \mathbb E \bigl[ \lVert X \rVert^2 \big] \leq s_4^2,
\] 
we obtain that, with probability at least $1 - 2 \epsilon$, 
for any $\theta \in \mathbb{R}^d$, 
\[ 
\left\lvert \, \frac{\max \{ N(\theta), \sigma \lVert \theta \rVert^2 \}}{\max \{ \widehat{N}(\theta), \sigma \lVert \theta \rVert^2 \}} - 1 \, \right\rvert \leq B_* \Bigl[ \lVert \theta \rVert^{-2} N(\theta) \Bigr].
\] 

\vskip 1mm
\noindent
Since by the Cauchy-Schwarz inequality 
\[
s_4^2 \leq \sqrt \kappa \mathbf{Tr}(G)
\]
we get
\begin{multline}
\label{eq:precise}
\zeta (t) \leq \sqrt{ 2 (\kappa-1) \Biggl( \frac{(2 + 3c) \; \mathbf{Tr}(G)}{4(2+c) t} 
+ \log(K / \epsilon) \Biggr)} \cosh(a/4)  
\\ + \sqrt{ \frac{2(2+c)\kappa \; 
\mathbf{Tr} (G)}{ t}} \cosh ( a/ 2).
\end{multline}
Choosing $a=1/2$ and computing explicitly numerical constants concludes the proof. 

\subsection{Proof of Proposition \ref{propq}}\label{proof_propq}
We observe that it is sufficient to prove that with probability at least $1-2\epsilon$ 
\begin{equation}\label{pf2}
\begin{aligned}
 \Phi_- \circ \Phi_+ \Bigl(  \langle \theta, \mathcal{Q} \theta \rangle_{\mathcal H}- \eta \Bigr) & \leq N \bigl( \Pi_k \theta \bigr) + \eta \\
 \Phi_- \circ \Phi_+ \Bigl( N \bigl( \Pi_k \theta \bigr) - \eta \Bigr) & \leq \langle \theta, \mathcal{Q} \theta \rangle_{\mathcal H}+\eta,
\end{aligned}
\end{equation}
where $\eta =2\delta \sqrt{ \mathbf{Tr}(\mathcal{G}^2)}$ and $N \bigl( \Pi_k \theta \bigr)  = \langle \Pi_k \theta, \mathcal{G} \Pi_k \theta \rangle_{\mathcal H}.$
Indeed, if equation \eqref{pf2} is satisfied, according to the postponed Corollary \ref{lem1A}, 
\begin{align*}
\bigl\lvert \max \bigl\{ \langle \theta, \mathcal{Q} \theta \rangle_{\mathcal H}, \sigma \bigr\} - \max \bigl\{ N \bigl( \Pi_k \theta \bigr)  , \sigma \bigr\}  \bigr\rvert 
& \leq 2 \max \bigl\{N \bigl( \Pi_k \theta \bigr) , \sigma \bigr\} B_* \bigl( N \bigl( \Pi_k \theta \bigr)  \bigr) + 5\eta/2 \\  
\bigl\lvert \max \bigl\{ \langle \theta, \mathcal{Q} \theta \rangle_{\mathcal H}, \sigma \bigr\} - \max \bigl\{ N \bigl( \Pi_k \theta \bigr) , \sigma \bigr\} \bigr\rvert 
& \leq 2 \max \bigl\{ \langle \theta, \mathcal{Q}\theta \rangle_{\mathcal H}, \sigma \bigr\} B_* \bigl( \min \{ \langle \theta, \mathcal{Q} \theta \rangle_{\mathcal H}, s_4^2 \} \bigr) \\ & \hspace{40ex} +5 \eta/2,   
\end{align*}
which is the analogous, in the infinite-dimensional setting, of Proposition \ref{prop3}. Thus, following the proof of Proposition \ref{prop4} we obtain the desired bounds.\\[1mm]
Let us now prove equation \eqref{pf2}. 
Observe that, for any $\theta \in \mathbb S_{\mathcal{H}}$, 
\[
\langle \theta, \mathcal{Q} \theta \rangle_{\mathcal H} = \langle \Pi_{V_k} \theta, 
\mathcal{Q} \Pi_{V_k} \theta \rangle_{\mathcal H} 
\leq \lVert \Pi_{V_k} \theta \rVert_{\mathcal H}^2 \bigl( \langle \xi, \mathcal{Q} \xi 
\rangle_{\mathcal H} + \eta \bigr),  
\]
where $\xi \in \Theta_\delta$ is the closest point in $\Theta_{\delta}$ 
to $\lVert \Pi_{V_k} \theta \rVert_{\mathcal H}^{-1} \Pi_{V_k} \theta$.
Since $\xi \in \mathcal{H}_k$, with probability at least $1 - \epsilon$, 
for any $(\lambda, \beta) \in \Lambda$, 
\[ 
\langle \xi, \mathcal{Q} \xi \rangle_{\mathcal H} \leq \Phi_+^{-1} \bigl( \widetilde{N}_{\lambda} 
(\xi) \bigr) 
= \Phi_+^{-1} \Bigl[ \widetilde{N}_{\lambda} \Bigl( \xi + \lVert \Pi_{V_k} 
\theta \rVert_{\mathcal H}^{-1} \bigl(\Pi_k - \Pi_{V_k} \bigr) \theta \Bigr) \Bigr].
\] 
Let us now remark that for any $a \in [0,1]$, we have
$ \Phi_+(a t) \leq a \Phi_+(t)$,
so that $a \Phi_+^{-1} (t) \leq \Phi_+^{-1}(a t)$. 
Therefore 
\begin{multline*}
\langle \theta, \mathcal{Q} \theta \rangle_{\mathcal H} \leq 
\lVert \Pi_{V_k} \theta \rVert_{\mathcal H}^2 \ \Phi_+^{-1} \Bigl\{ \widetilde{N}_{\lambda} 
\Bigl[ \lVert \Pi_{V_k} \theta \rVert_{\mathcal H}^{-1} 
\Bigl(  \lVert \Pi_{V_k} \theta 
\rVert_{\mathcal H} \xi + \bigl( \Pi_k - \Pi_{V_k} \bigr) \theta \Bigr) \Bigr]  
\Bigr\} + \eta \\
\quad \ \leq 
\lVert \Pi_{V_k} \theta \rVert_{\mathcal H}^2 \ \Phi_+^{-1} \circ \Phi_-^{-1}  
\Bigl\{ N 
\Bigl[ \lVert \Pi_{V_k} \theta \rVert_{\mathcal H}^{-1} 
\Bigl(  \lVert \Pi_{V_k} \theta 
\rVert_{\mathcal H} \xi + \bigl( \Pi_k - \Pi_{V_k} \bigr) \theta \Bigr) \Bigr]  
\Bigr\} + \eta \\ \leq 
\Phi_+^{-1} \circ \Phi_-^{-1} \Bigl\{ 
N \Bigl[ \lVert \Pi_{V_k} \theta \rVert_{\mathcal H} \xi + \bigl( 
\Pi_k - \Pi_{V_k} \bigr) \theta \Bigr) \Bigr] \Bigr\} + \eta
\\ \leq \Phi_+^{-1} \circ \Phi_-^{-1} \Bigl( N \bigl( \Pi_k \theta \bigr) 
+ \eta \Bigr) + \eta. 
\end{multline*}
Indeed, 
\[ 
\Bigl\lVert \Bigl( \lVert \Pi_{V_k} \theta \rVert_{\mathcal{H}} \xi + \bigl( 
\Pi_k - \Pi_{V_k} \bigr) \theta \Bigr) - \Pi_k \theta \Bigr\rVert \leq \delta, 
\] 
and this is a difference of two vectors belonging to the unit ball. 
In the same way 
\begin{multline*} 
\langle \theta, \mathcal{Q} \theta \rangle_{\mathcal H} \geq \lVert \Pi_{V_k} \theta 
\rVert_{\mathcal H}^2 \bigl( \langle \xi, \mathcal{Q} \xi \rangle_{\mathcal H} - \eta \bigr) 
\\ \geq \lVert \Pi_{V_k} \theta \rVert_{\mathcal H}^2 \ \Phi_- \Bigl\{ \widetilde{N}_{\lambda} 
\Bigl[ \lVert \Pi_{V_k} \theta \rVert_{\mathcal H}^{-1} \Bigl( 
\lVert \Pi_{V_k} \theta \rVert_{\mathcal H} \xi + \bigl( \Pi_k - \Pi_{V_k} \bigr) 
\theta \Bigr) \Bigr] \Bigr\} - \eta \\
\quad \ \geq 
\lVert \Pi_{V_k} \theta \rVert_{\mathcal H}^2 \ \Phi_- \circ \Phi_+ \Bigl\{ 
N \Bigl[ \lVert \Pi_k \theta \rVert_{\mathcal H}^{-1} 
\Bigl( \lVert \Pi_{V_k} \theta \rVert_{\mathcal H} \xi 
+ \bigl( \Pi_k - \Pi_{V_k} \bigr) \theta \Bigr) \Bigr] \Bigr\} - \eta
\\ \geq \Phi_- \circ \Phi_+ \Bigl( N \bigl( \Pi_k \theta \bigr) 
- \eta \Bigr) - \eta
\end{multline*}
which proves equation \eqref{pf2}.

\subsection{A technical result}\label{appx}

In all this section we use the same notation as in section \ref{sec1}.
Let $\sigma \in ]0,s_4^2]$ be such that $8\zeta_*(\sigma) \leq \sqrt n$ where $\zeta_*$ is defined in Proposition \ref{prop2}.

\begin{lem}\label{lem0A} The function
\[
t \mapsto F(t) = \max\{ t, \sigma\} B_*(\min\{t,s_4^2\}),
\]
where $B_*$ is defined in Proposition \ref{prop2},
is non-decreasing for any $t\in \mathbb R_+.$
\end{lem}
\begin{proof} 
If $\sigma \geq s_4^2$, then $B_*(\min \{ t, s_4^2 \}) = B_*(\sigma)$, 
so that $F(t) = \max \{ t, \sigma \} B_*(\sigma)$ is obviously 
non-decreasing.
Otherwise, $\sigma \leq s_4^2,$ so that  
\[
\zeta\left( \max\left\{\min\{t,s_4^2\}, \sigma \right\} \right) = 
\zeta\left( \min\left\{\max\{t,\sigma\}, s_4^2 \right\} \right).
\]
Therefore the function $F$ is of the form 
\[
F(t) = c \frac{ u g(u) }{ (1 - g(u))},
\]
where $u = \max \{ t, \sigma \}$, 
\[
g(u) = \sqrt{a_1/u + a_2} + \sqrt{a_3/u},
\]
$g(\sigma) \leq 1/2$, 
and the constants $c, a_1, a_2$, and $a_3$ are positive. 
Let $h(u) = \sqrt{a_1/u} + \sqrt{a_3/u}$ and observe that 
\[
g'(u) = - \frac{1}{2u} \biggl( \frac{a_1/u}{\bigl( a_1/u + a_2 \bigr)^{1/2}} 
+ \sqrt{a_3/u} \biggr) \geq - \frac{1}{2u} \Bigl( \sqrt{a_1/u} + \sqrt{a_3/u}
\Bigr) = h'(u)
\]
and that $g(u) \geq h(u)$.
Therefore $h(u) \leq g(u) \leq 1/2,$ for any $u \geq \sigma$, and  
\[ 
\frac{\partial}{\partial u} \log \Biggl( \frac{u g(u)}{1 - g(u)} \Biggr) 
= \frac{1}{u} + \frac{g'(u)}{g(u) \bigl( 1 - g(u) \bigr)} \geq \frac{1}{u} 
+ \frac{h'(u)}{h(u) \bigl( 1 - h(u) \bigr)} = \frac{1}{u} 
- \frac{1}{2u\bigl(1 - h(u) \bigr)} \geq 0,
\] 
showing that $F$ is non-decreasing.
\end{proof}

\vskip2mm

\begin{lem}\label{lem2A} For any $(a, b) \in \mathbb{R}^2$ such that, 
for any $(\lambda, \beta) \in \Lambda$,  
\[ 
\Phi_- \circ \Phi_+ ( a - \eta ) \leq b + \eta, \quad \text{ and } \quad
\Phi_- \circ \Phi_+ ( b - \eta ) \leq a + \eta, 
\] 
and any threshold $\sigma \in \mathbb{R}_+$ such that 
$8 \zeta(\sigma) \leq \sqrt{n}$ and $\sigma \leq s_4^2$, we have
\begin{align}
\label{eq:1.37}
\bigl\lvert \max \{ a, \sigma \} - \max \{ b, \sigma \} \bigr\rvert
& \leq 2 \max \{a + \eta, \sigma \}   B_* \bigl( \min \{  a  + 
\eta, s^2_4 \} \bigr) + 2 \eta \\  
\label{eq:1.38}
\bigl\lvert \max \{ a, \sigma \} - \max \{ b, \sigma \} \bigr\rvert
& \leq 2 \max \{b + \eta, \sigma \}   B_* \bigl( \min \{  b  + 
\eta, s^2_4 \} \bigr) + 2 \eta.  
\end{align}
\end{lem}

\begin{proof}
By symmetry of $a$ and $b$, equation \eqref{eq:1.38} is a consequence of equation \eqref{eq:1.37}.\\[1mm]
{\bf Step 1.} We will prove that 
\begin{equation}
\label{eq:1.39}
\max \{ b - \eta , \sigma \} \leq \max \{ a + \eta, \sigma \} 
\Bigl( 1 + 2 \widetilde{B}_{\lambda, \beta}  ( a + \eta ) \Bigr),  
\end{equation}
where $\widetilde{B}_{\lambda, \beta}$ is defined in equation \myeq{Btilde}.\\[1mm]
{\bf Case 1.} Assume that 
\[ 
\max \bigl\{ \Phi_+ ( b - \eta ) , \sigma \bigr\} \leq \max \{ 
a + \eta, \sigma \}, 
\] 
and remark that, since $\Phi_+$ is non-decreasing and 
$\Phi_+(\sigma) \leq \sigma$,  
\begin{multline*}
\max \bigl\{ \Phi_+ (b - \eta), \sigma \bigr\} \geq 
\max \bigl\{ \Phi_+ (b - \eta), \Phi_+(\sigma) \bigr\} \\ = 
\Phi_+ \bigl( \max \{ (b - \eta), \sigma \} \bigr) 
 = 
\frac{ \max \{ b - \eta, \sigma \}}{1 + B_{\lambda, \beta}(b - \eta)}, 
\end{multline*}
according to equation \myeq{phi+eq}, where $B_{\lambda, \beta}$ is defined in equation \eqref{defBlb}.
Therefore in this case, 
\begin{equation}
\label{eq:1.40}
\max \{ b - \eta, \sigma \} \leq \max \{ a + \eta, \sigma \} 
\Bigl( 1 + B_{\lambda, \beta} (b - \eta) \Bigr), 
\end{equation}
but when $\max \{ b- \eta, \sigma \} > \max \{ a + \eta, \sigma \}$, 
\[
B_{\lambda, \beta} ( b- \eta) \leq B_{\lambda, \beta} (a + \eta)
\]
because $B_{\lambda, \beta}(t)$ is a non-increasing 
function of $\max \{t, \sigma \}$, thus equation \eqref{eq:1.40} implies that   
\[ 
\max \{ b - \eta, \sigma \} \leq \max \{ a + \eta, \sigma \} 
\Bigl( 1 + B_{\lambda, \beta} (a + \eta) \Bigr).
\] 
Since $B_{\lambda, \beta} \leq 
\widetilde{B}_{\lambda, \beta}$, equation \eqref{eq:1.39} holds true.\\[1mm]
{\bf Case 2.} Assume now that we are not in {\bf Case 1}, 
implying that
\[
\max \{ b - \eta, \sigma \} \geq \max \bigl\{ \Phi_+(b- \eta), \sigma \bigr\} > \max \{ a + \eta, \sigma \}. 
\]
In this case 
\begin{multline*}
\max \{ a + \eta, \sigma \} 
\geq \max \bigl\{ \Phi_- \circ \Phi_+(b - \eta), \sigma \bigr\} 
\geq \max \bigl\{ \Phi_- \circ \Phi_+(b - \eta), \Phi_-(\sigma) \bigr\}
\\ \geq \Phi_- \bigl(  \max \{ \Phi_+(b - \eta), \sigma \} \bigr) 
\geq \max \bigl\{ \Phi_+(b-\eta), \sigma \bigr\} \Bigl[ 1 - B_{\lambda, 
\beta} \Bigl( \max \bigl\{ 
\Phi_+(b-\eta), \sigma \bigr\} \Bigr)   \Bigr]
\end{multline*}
according to equation \myeq{phi-eq}. Moreover, continuing the above chain of inequalities,
\begin{multline*}
\max \{ a + \eta, \sigma \}  \geq \max \bigl\{ \Phi_+(b-\eta), \Phi_+(\sigma) \bigr\}  
\Bigl[ 1 - B_{\lambda, \beta}(\max \{a + \eta, \sigma\}) \Bigr] \\
= \Phi_+ \bigl( \max \{ b- \eta, \sigma \} \bigr) 
\Bigl[ 1 - B_{\lambda, \beta}(a + \eta) \Bigr] 
\\ \quad \geq \max \{ b - \eta, \sigma \}  
\frac{1 - B_{\lambda, \beta}(a + \eta)}{ 1 + 
B_{\lambda, \beta}(\max\{ b- \eta, \sigma \} ) } 
\\ \geq \max \{ b-\eta, \sigma \} \frac{ 1 - B_{\lambda, \beta}(a + \eta)}{
1 + B_{\lambda, \beta}(a+\eta)}. 
\end{multline*}
Therefore 
\begin{multline*}
\max \{ b- \eta, \sigma \} 
\leq \max \{ a + \eta, \sigma \} \frac{1 + 
B_{\lambda, \beta}(a+\eta)}{1 - B_{\lambda, \beta}
(a + \eta)} 
\\ = \max \{ a + \eta, \sigma \} \biggl( 1 + \frac{ 2 B_{\lambda, \beta}(a + \eta)}{
1 - B_{\lambda, \beta}(a + \eta)} \biggr) \leq 
\max \{ a + \eta, \sigma \} \bigl( 1 + 2 \widetilde{B}_{\lambda, \beta}(a + \eta) 
\bigr)
\end{multline*} 
according to Lemma \ref{lem1.14}. This concludes the proof of {\bf Step 1}. \\[1mm]
{\bf Step 2} Taking the infimum in $(\lambda, \beta) \in \Lambda$ 
in equation \eqref{eq:1.39}, 
according to equation \myeq{eq1.17},
we obtain that 
\[ 
\max \{ b - \eta, \sigma \} \leq \max \{ a + \eta, \sigma \} 
\Bigl( 1 + 2 B_* \bigl( \min \{ a + \eta, s_4^2 \} ) \Bigr).
\]  
We can then use the fact that $t \mapsto \max \{t, \sigma \} 
B_*( \min \{ t, s_4^2 \})$ is non-decreasing (proved in Lemma \ref{lem0A})
to deduce that 
\[
\max \{ b - \eta, \sigma \} \leq \max \{ a + \eta, \sigma \} 
+ 2 \max \{ b + \eta, \sigma \} B_*(\min \{ b + \sigma, s_4^2\}),
\] 
since there is nothing to prove when already $\max \{ b + \eta, \sigma \} 
\leq \max \{ a + \eta, \sigma \}$. 
Remark that $\max \{ a + \eta, \sigma \} \leq 
\max \{ a + \eta, \sigma + \eta \} \leq \max \{ a, \sigma \} + \eta$ 
and that in the same way $\max \{ b - \eta, \sigma \} \geq \max \{ b, \sigma \} 
- \eta$. This proves that 
\begin{align*}
\max \{ b, \sigma \} - \max \{ a , \sigma \} & \leq 2 \max \{ a + \eta, \sigma
\} B_* \bigl( \min \{ a + \eta, s_4^2 \} \bigr) + 2 \eta \\
\text{and }  
\max \{ b, \sigma \} - \max \{ a , \sigma \} & \leq 2 \max \{ b + \eta, \sigma
\} B_* \bigl( \min \{ b + \eta, s_4^2 \} \bigr) + 2 \eta. 
\end{align*}
By symmetry, we can then exchange $a$ and $b$ to prove the same bounds 
for $\max \{ a , \sigma \} - \max \{ b, \sigma \}$, and therefore 
also for the absolute value of this quantity, which ends the proof 
of the lemma.
\end{proof}

\vskip2mm
\noindent
As a consequence the following result holds. 

\begin{cor}\label{lem1A}
For any $(a, b) \in \mathbb{R}^2$ such that, 
for any $(\lambda, \beta) \in \Lambda$,  
\[ 
\Phi_- \circ \Phi_+ ( a - \eta ) \leq b + \eta, \quad \text{ and } \quad
\Phi_- \circ \Phi_+ ( b - \eta ) \leq a + \eta, 
\] 
and any threshold $\sigma \in \mathbb{R}_+$ such that 
$8 \zeta(\sigma) \leq \sqrt{n}$ and $\sigma \leq s_4^2$, we have
\begin{align*}
\bigl\lvert \max \{ a, \sigma \} - \max \{ b, \sigma \} \bigr\rvert
& \leq 2 \max \{a, \sigma \}   B_* \bigl( \min \{  a, s^2_4 \} \bigr) 
+ 5  \eta / 2 \\  
\bigl\lvert \max \{ a, \sigma \} - \max \{ b, \sigma \} \bigr\rvert
& \leq 2 \max \{b, \sigma \}   B_* \bigl( \min \{  b, s^2_4 \} \bigr) 
+ 5 \eta / 2.  
\end{align*}
\end{cor}

\begin{proof}
This is a consequence of the previous lemma, of the fact that 
$B_*(\min \{ t, s_4^2 \} ) \leq 1/4$, and of the fact that
$\max \{ a + \eta, \sigma \} \leq \max \{ a, \sigma \} + \eta$. 
\end{proof}

\subsection{Proof of Proposition \ref{prop2.8eg}}\label{sec65pf}

\noindent
Let $\Lambda \subset \bigl( \mathbb{R}_+ \setminus \{ 0 \} \bigr)^2$ be the 
finite set defined in equation \myeq{defLambda}. 
We use as a tool the family of estimators 
\[
\widetilde{N}_{\lambda}(\theta) = \frac{\lambda}{\widehat \alpha(\theta)^2}
\] 
introduced in equation \eqref{tildeN}, where $\widehat{\alpha}(\theta)$ is defined in equation \eqref{hatalpha}.
Let us put
\[ 
\tau_{\lambda}(t)  = \frac{\lambda^2  R^4}{3 \max \{ t, \sigma \}^2}, \quad t \in \mathbb R_+.
\] 
We divide the proof into 4 steps. \\[1mm]
{\bf Step 1.} The first step consists in linking the empirical estimator $\bar N$ with $\widetilde{N}_{\lambda}.$\\[1mm]
We claim that,  
with probability at least $1-\epsilon,$ 
for any $\theta \in \mathbb{S}_d$, any $(\lambda, \beta) \in \Lambda,$ such that $\Phi_+ \bigl( N(\theta) \bigr) > 0$,
\[
\frac{\bar N(\theta)}{\max \{ \widetilde{N}_{\lambda}(\theta), \sigma \} } \leq 
\Bigl[ 
1 - \tau_{\lambda} \Bigl( \widetilde{N}_{\lambda}(\theta)\Bigr) \Bigr]_+^{-1},
\]
with the convention that $\frac{1}{0}= +\infty.$ 
Moreover, with probability at least $1-\epsilon,$ for any $\theta \in \mathbb{S}_d$, 
any $(\lambda, \beta) \in \Lambda$, such that $\Phi_+ \bigl( 
N(\theta) \bigr) > 0$, 
\[
\frac{\bar N(\theta)}{\widetilde N_{\lambda}(\theta) } \geq 1 - \frac{ \lambda^2}{3}.
\]
We first observe that, according to the definition of $\widehat{\alpha}(\theta)$, for any threshold $\sigma \in \mathbb{R}_+$,  
\[ 
\frac{1}{n} \sum_{i=1}^n \psi \Bigl[ \lambda \Bigl( \max \{ \widetilde{N}_{\lambda} (\theta), \sigma \}^{-1} \langle \theta, X_i \rangle^2  - 1 \Bigr) \Bigr] 
\leq  r \Bigl( \lambda^{1/2}  \widetilde{N}_{\lambda}(\theta)^{-1/2} \theta \Bigr) 
= r_{\lambda} \Bigl( \widehat{\alpha}(\theta) \, \theta \Bigr)  \leq 0,    
\] 
where we have used the fact that the function $\psi$, introduce in equation \eqref{defnpsi}, is non-decreasing.
Moreover 
\[
 r_{\lambda} \Bigl( \widehat{\alpha}(\theta) \, \theta \Bigr)  = 0
\]
as soon as $\widehat{\alpha}(\theta)<+\infty$ and this holds true, according to Proposition \ref{prop0}, with probability at least $1 - \epsilon$, for any $\theta \in \mathbb{S}_d$ and any $(\lambda, \beta) \in \Lambda$ such that $\Phi_+ \bigl( N(\theta) \bigr) > 0$.
Indeed, by Proposition \ref{prop0},
with probability at least $1 - \epsilon$, for any $\theta \in \mathbb{S}_d$, any $(\lambda, \beta) \in \Lambda,$
\[
\widetilde{N}_{\lambda} (\theta) \geq \Phi_+ \bigl( N(\theta) \bigr).
\] 
Defining $g(z) = z -\psi(z),$ we get
\begin{multline} \label{eq2.1} 
\frac{\bar N(\theta)}{\max \{ \widetilde{N}_{\lambda}(\theta), \sigma \} }-1 
= \frac{1}{n \lambda}\sum_{i=1}^n \lambda \left( 
\langle \theta,X_i\rangle^2 \max \{ 
\widetilde{N}_{\lambda}(\theta), \sigma \}^{-1} - 1 \right) \\
\leq  \frac{1}{n \lambda}\sum_{i=1}^n g\left[  \lambda \left( 
\langle \theta,X_i\rangle^2 \max \{ 
\widetilde{N}_{\lambda}(\theta), \sigma \}^{-1} - 1 \right)\right].
\end{multline}
In the same way, with probability at least $1 - \epsilon$,
for any $\theta \in \mathbb{S}_d$, any $(\lambda, \beta) \in \Lambda$
such that $\Phi_+ \bigl( N(\theta) \bigr) > 0$, we obtain
\begin{equation}
1 - \frac{\bar{N}(\theta)}{\widetilde{N}_{\lambda}(\theta)} 
\leq \frac{1}{n \lambda} \sum_{i=1}^n g \Bigl[ \lambda \Bigl( 1 - \langle 
\theta, X_i \rangle^2 \widetilde{N}_{\lambda}(\theta)^{-1}\Bigr) 
\Bigr]. 
\end{equation}
We remark that the derivative of $g$ is
\[
g'(z) = 1- \psi'(z) =\begin{cases} 1 & \text{if}\ z \notin [-1,1]\\
\displaystyle \frac{\frac{z^2}{2}}{1+z+\frac{z^2}{2}} & \text{if} \ z \in [-1,0]\\ 
\displaystyle\frac{\frac{z^2}{2}}{1-z+\frac{z^2}{2}}  & \text{if} \ z \in [0,1],
\end{cases}
\]
showing that $0 \leq g'(z) \leq z^2$, and therefore 
that $g$ is a non-decreasing function satisfying 
\begin{equation}\label{eq2.2}
g(z) \leq \frac{1}{3}z_+^3.
\end{equation}
Applying equation \eqref{eq2.2} to equation \eqref{eq2.1} we obtain
\begin{align*}
\frac{\bar N(\theta)}{\max \{ \widetilde{N}_{\lambda}(\theta), \sigma \}} - 1 & 
\leq \frac{\lambda^2}{3n}\sum_{i=1}^n \Bigl(  
\langle \theta,X_i\rangle^2 \max \{ \widetilde{N}_{\lambda}(\theta), \sigma \}^{-1 }-1
\Bigr)^3_+\\
& \leq \frac{ \lambda^2}{3n \max \{ \widetilde{N}_{\lambda}(\theta), \sigma \}^3} \sum_{i=1}^n \langle \theta,X_i\rangle^6,
\end{align*}
where we have used the fact that $(z^2-1)_+\leq z^2.$ 
Since, by the Cauchy-Schwarz inequality, $  \langle \theta,X_i\rangle^2 \leq \|\theta\|^2 R^2 = R^2,$ we get
\begin{align*}
\frac{\bar N(\theta)}{\max \{ \widetilde{N}_{\lambda}(\theta), \sigma \} }-1 &  
\leq \frac{ \lambda^2}{3 n \max \{ \widetilde{N}_{\lambda}(\theta), \sigma \}^3} R^4 \sum_{i=1}^n  \langle \theta,X_i\rangle^2 \\
& = \frac{ \lambda^2}{3} \times \frac{ R^4}{\max \{ 
\widetilde{N}_{\lambda}(\theta), \sigma \}^2} \times \frac{ \bar N(\theta)}{\max \{ 
\widetilde{N}_{\lambda}(\theta), \sigma \}},
\end{align*} 
which proves the first inequality. 
Similarly, since $g$ in non-decreasing, 
we obtain that, with probability at least $1 - \epsilon$, 
for any $\theta \in \mathbb{S}_d$, any $(\lambda, \beta) \in \Lambda$ such that 
$\Phi_+\bigl(N(\theta) \bigr) > 0$, 
\[
1- \frac{\bar N(\theta)}{\widetilde N_{\lambda}(\theta) }\leq \frac{1}{n\lambda}  \sum_{i=1}^n  g(\lambda) \leq \frac{ \lambda^2}{3},
\]
where the last inequality follows from equation \eqref{eq2.2}.
\\[2mm]
{\bf Step 2.} This is an intermediate step. We claim that, with probability at least $1 - 2 \epsilon$, 
for any $\theta \in \mathbb{S}_d$, any $(\lambda, \beta) \in \Lambda$, 
any $\sigma > 0$, 
\begin{align*}
\max \{ \bar{N}(\theta), \sigma \} & \leq \Phi_-^{-1} \Bigl( \max \{ N(\theta), \sigma \} \Bigr) \Bigl[ 1 - \tau_{\lambda} \bigl( N(\theta) \bigr)\Bigr]_+^{-1} \\ 
\max \bigl\{ \bar{N}(\theta), \sigma \bigr\}& \leq \Phi_-^{-1} \Bigl( \max \{ N(\theta), \sigma \} \Bigr)\biggl[ 1 - \tau_{\lambda} \Bigl( 
\bar{N}(\theta) \bigl[ 1 - \tau_{\lambda}(\sigma) \bigr]_+ \Bigr) \biggr]_+^{-1} \\
\bar{N}(\theta) & \geq \biggl( 1 - \frac{\lambda^2}{3} \biggr)_+ \Phi_+\bigl( N(\theta) \bigr) 
\end{align*}
where $\Phi_+$ and $\Phi_-$ are defined in  Proposition \ref{prop0}.\\[1mm]
We consider the threshold 
\[
\sigma' = \Phi_-^{-1} \bigl( \max \{ N(\theta), \sigma \} )  \geq  \max \{ N(\theta), \sigma \},
\]
where we have used the fact that, by definition,  $\Phi_-(t)^{-1}\geq t,$ for any $t \in \mathbb R_+.$
We assume that we are in the intersection of the two events of Proposition \ref{prop0}, which holds true with probability at least $1 - 2 \epsilon,$ so that 
\begin{equation}\label{sigma'}
\sigma' \geq  \max \{ N(\theta), \sigma, \widetilde N_{\lambda}(\theta) \} .
\end{equation}
According to {\bf Step 1}, choosing as a threshold $\max\{\sigma, \sigma'\},$ we get 
\begin{align*}
\frac{\bar{N}(\theta)}{\max\{\widetilde{N}_{\lambda}(\theta), \sigma, \sigma'\}} \leq  \left[ 1 - \tau_{\lambda}\Bigl(\max \bigl\{  \widetilde N_{\lambda}(\theta), \sigma' \bigr\} \Bigr) \right]_+^{-1},
\end{align*}
(where $\tau_{\lambda}$ is still defined with respect to $\sigma$), so that, according to equation \eqref{sigma'},
\begin{equation}\label{2.6prop2.6}
\bar{N}(\theta) \leq \sigma'\left[ 1 - \tau_{\lambda}\bigl( \sigma'  \bigr) \right]_+^{-1}.
\end{equation}
As a consequence, recalling the definition of $\sigma',$ we have 
\[
\bar{N}(\theta) \leq \Phi_-^{-1} \Bigl( \max \{ N(\theta), \sigma \} \Bigr) \Bigl[ 1 - \tau_{\lambda} \bigl( N(\theta)\bigr) \Bigr]_+^{-1}.
\]
Thus, observing that
\[
\sigma \leq \Phi_-^{-1} (\sigma) \leq \Phi_-^{-1} \Bigl( \max \{ N(\theta), \sigma \} \Bigr) \Bigl[ 1 - \tau_{\lambda} \bigl( N(\theta)\bigr) \Bigr]_+^{-1},
\]
we obtain the first inequality. 
To prove the second inequality, 
we use equation \eqref{2.6prop2.6} once to see that 
\[ 
\sigma' \geq \bar{N}(\theta) \bigl[ 1 - \tau_{\lambda} ( \sigma' ) \bigr]_+\geq \bar{N}(\theta) \bigl[ 1 - \tau_{\lambda} ( \sigma ) \bigr]_+,
\] 
and we use it again to get 
\begin{align*}
\bar{N}(\theta) & \leq \Phi_-^{-1} \bigl( \max \{ N(\theta), \sigma \} \bigr) \Bigl[ 	1 - \tau_{\lambda} ( \sigma') \Bigr]_+^{-1}\\
& \leq \Phi_-^{-1} \bigl( \max \{ N(\theta), \sigma \} \bigr) \Bigl[ 	1 - \tau_{\lambda} \Bigl( \bar{N}(\theta) [ 1 - \tau_{\lambda}(\sigma) ]_+\Bigr) \Bigr]_+^{-1}.
\end{align*}
To complete the proof of the second inequality, it is enough to remark 
that  
\[
\sigma \leq \Phi_-^{-1} (\sigma) \leq \Phi_-^{-1} \Bigl( \max \{ N(\theta), \sigma \} \Bigr) \Bigl[ 1 - \tau_{\lambda} \Bigl( \bar{N}(\theta) \bigl[ 1 - \tau_{\lambda}(\sigma) \bigr]_+ \Bigr) \Bigr]_+^{-1}.
\]
To prove the last inequality, it is sufficient to remark that $ \widetilde{N}_{\lambda}(\theta) \geq \Phi_+ \bigl( N(\theta) \bigr)$ by Proposition \ref{prop0} and hence, when $\Phi_+ \bigl( N(\theta) \bigr) > 0,$ 
\[ 
\bar{N}(\theta) \geq \biggl( 1 - \frac{\lambda^2}{3} \bigg)_+\Phi_+ \bigl( N(\theta) \bigr).
\] 
On the other hand, when $\Phi_+ \bigl( N(\theta) \bigr) = 0,$ this inequality is also obviously satisfied.
\\[2mm]
{\bf Step 3.} We now prove that, with probability at least $1 - 2 \epsilon$, for any $\theta \in \mathbb{S}_d$, 
any $(\lambda, \beta) \in \Lambda$, any $\sigma > 0$,  
\begin{align*}
\frac{\max \{ \bar{N}(\theta),  \sigma \} }{\max \{ N(\theta), \sigma \} } - 1 & \leq \widetilde{B}_{\lambda, \beta} \bigl( N(\theta) \bigr) + \frac{\tau_{\lambda}\bigl( N(\theta) \bigr) }{\bigl[ 1 - \tau_{\lambda} \bigl( N(\theta) \bigr) \bigr]_+ \bigl[ 1 - B_{\lambda, \beta} \bigl( N(\theta) \bigr) \bigr]_+ }, \\ 
1 - \frac{\max \{ \bar{N}(\theta),  \sigma \} }{\max \{ N(\theta), \sigma \} } & \leq B_{\lambda, \beta} \bigl(N(\theta) \bigr) +\frac{\lambda^2}{3},
\end{align*}
where $B_{\lambda, \beta}$ is defined in equation \myeq{defBlb} 
and $\widetilde{B}_{\lambda, \beta}$ in equation \myeq{Btilde}.\\[1mm]
We observe that, according to {\bf Step 2}, 
\begin{align*}
\max \{\bar N(\theta), \sigma  \} &\leq \Phi_-^{-1}\bigl( \max \{ N(\theta), \sigma  \} \bigr) \bigl[ 1 - \tau_{\lambda} \bigl( N(\theta) \bigr) \bigr]_+ ^{-1} \\
&\leq \frac{\max \{ N(\theta), \sigma  \} }{\bigl[ 1 - \tau_{\lambda} \bigl( N(\theta) \bigr) \bigr]_+ \bigl[ 1 -B_{\lambda, \beta} \bigl( N(\theta) \bigr) \bigr]_+ },
\end{align*}
which implies
\[
\frac{\max \{ \bar N(\theta), \sigma  \} }{\max \{ N(\theta), \sigma  \} }-1 \leq \frac{B_{\lambda, \beta} \bigl( N(\theta) \bigr)}{\bigl[ 1 -B_{\lambda, \beta} \bigl( N(\theta) \bigr) \bigr]_+ } + \frac{ \tau_{\lambda} \bigl( N(\theta) \bigr)}{\bigl[ 1 - \tau_{\lambda} \bigl( N(\theta) \bigr) \bigr]_+ \bigl[ 1 -B_{\lambda, \beta} \bigl( N(\theta) \bigr) \bigr]_+ }.
\]
Applying Lemma \ref{lem1.14} we obtain the first inequality. \\[1mm]
To prove the second inequality we observe that, using again {\bf Step 2},
\begin{align*}
\max \{\bar N(\theta), \sigma  \} & \geq \biggl( 1 - \frac{\lambda^2}{3} \bigg)_+\Phi_+ \bigl( \max \{ N(\theta), \sigma  \} \bigr)\\
& = \biggl( 1 - \frac{\lambda^2}{3} \bigg)_+ \max \{ N(\theta), \sigma  \} \bigl[ 1 + B_{\lambda, \beta} \bigl( N(\theta) \bigr) \bigr]^{-1},
\end{align*}
where we have used the fact that 
$
\Phi_+\bigl( \max \{ z, \sigma \} \bigr) = \max \{ z, \sigma \} \Bigl( 1 + B_{\lambda, \beta} (z) \Bigr)^{-1}
$
as shown in equation \myeq{phi+eq}. Thus we conclude that 
\begin{align*}
1- \frac{\max \{ \bar N(\theta), \sigma  \} }{\max \{ N(\theta), \sigma  \} } \leq \frac{B_{\lambda, \beta} \bigl( N(\theta) \bigr) +\lambda^2/3 }{ 1 + B_{\lambda, \beta} \bigl( N(\theta) \bigr) }  \leq B_{\lambda, \beta} \bigl(N(\theta) \bigr) +\frac{\lambda^2}{3}.
\end{align*}
\\[2mm]
{\bf Step 4.} From {\bf Step 3} we deduce that 
\[
\biggl\lvert \, \frac{\max \{ \bar{N}(\theta),  \sigma \} }{\max \{ N(\theta), \sigma \} } - 1 \biggr\rvert\leq \widetilde{B}_{\lambda, \beta} \bigl( 
N(\theta) \bigr) + \frac{\tau_{\lambda}\bigl( N(\theta) \bigr) }{\bigl[ 1 - \tau_{\lambda} \bigl( N(\theta) \bigr) \bigr]_+ \bigl[ 1 -B_{\lambda, \beta} \bigl( N(\theta) \bigr) \bigr]_+ }. 
\]
To conclude the proof it is sufficient to apply {\bf Step 3} to 
$(\lambda_{\widehat{\jmath}}, \beta_{\widehat{\jmath}}) \in \Lambda$ 
defined in equation \myeq{defLambda}. Indeed, by equation \myeq{eq1.17}, for any $t \in \mathbb R_+,$ 
\[
B_{\lambda_{\widehat{\jmath}}, \beta_{\widehat{\jmath}}}(t)\leq B_*(t)
\]
and, by equation \myeq{eq33}, we have $\lambda_{\widehat{\jmath}}  \leq  \lambda_*(\theta) \exp ( a/ 4).$

\subsection{Proof of Proposition \ref{prop52empg}}\label{pf111}

\noindent
We observe that another way to take advantage of equation \myeq{eq2.1} is to write 

\[ 
\frac{\bar{N}(\theta)}{\max \{ \widetilde{N}_{\lambda}(\theta), \sigma \}} - 1 \leq \frac{\lambda^2 \lVert \theta \rVert^6}{3 \max \{ \widetilde{N}_{\lambda} (\theta), \sigma \}^3}\frac{1}{ n} \sum_{i=1}^n \lVert X_i \rVert^6.
\] 
Thus, putting 
\[
\zeta_{\lambda} (t) = \frac{\lambda^2 \widetilde{R}^6}{3 \max \{ t, \sigma \}^3}, \quad t \in \mathbb R_+,
\] 
we get that, for any $\theta \in \mathbb{S}_d$,  
\[
\frac{\bar{N}(\theta)}{\max \{ \widetilde{N}_{\lambda}(\theta), 
\sigma \}} \leq 1 + \zeta_{\lambda} \bigl( \widetilde{N}_{\lambda}(\theta) \bigr).
\]
The same reasoning used to prove {\bf Step 2} of Proposition \ref{prop2.8eg} shows that, 
with probability at least $1 - \epsilon$, for any $\theta \in \mathbb{S}_d$, 
any $(\lambda, \beta) \in \Lambda$, any $\sigma > 0$, 
\[ 
\max \{ \bar{N}(\theta), \sigma \} \leq \Phi_-^{-1} 
\bigl( \max \{ N(\theta), \sigma \} \bigr) \bigl[ 1 + \zeta_{\lambda} 
\bigl( N(\theta) \bigr) \bigr]  .
\] 
As a consequence, with probability at least $1 - 2 \epsilon$, for any 
$\theta \in \mathbb{S}_d$, any $(\lambda, \beta) \in \Lambda$,
\[ 
\biggl\lvert \frac{\max \{ \bar{N}(\theta), \sigma \}}{\max \{ 
N(\theta), \sigma \}} - 1 \biggr\rvert \leq \widetilde{B}_{\lambda, \beta}
\bigl( N(\theta) \bigr) + \frac{ \zeta_{\lambda} \bigl( N(\theta) \bigr)}{ 
\bigl[ 1 - B_{\lambda, \beta} \bigl( N(\theta) \bigr) \bigr]_+}. 
\]

\subsection{Proof of Proposition \ref{props1}}\label{pfsym1}

\noindent
To prove Proposition \ref{props1}, we use many results already proved in the case of the Gram matrix (with the necessary adjustments).\\[1mm]
We first observe that, denoting by $W \in \mathbb R^d$ a gaussian random vector with mean $A^{1/2}\theta$ and covariance matrix $\beta^{-1}A,$ we have
\begin{align*}
\mathbb E \Bigl[ \| A^{1/2} \theta'\|^2 \mathrm d \pi_{\theta}(\theta')  
\Bigr] 
& = \sum_{i=1}^d \mathbb E \bigl( \langle W, e_i\rangle^2 \bigr)\\
& =  \| A^{1/2} \theta\|^2+\frac{\mathbf{Tr}(A)}{\beta}
\end{align*}
where $\{e_i\}_{i=1}^d$ is the canonical basis of $\mathbb R^d$ (and $ \langle W, e_i\rangle$ is a one-dimensional Gaussian random  variable with mean $ \langle A^{1/2}\theta, e_i\rangle$ and variance $\beta^{-1} e_i^{\top} A e_i$). 
Therefore the empirical criterion rewrites as
\[
r_{\lambda} (\theta) 
= \frac{1}{n} \sum_{i=1}^n \psi \left[ \int \left(\| A_i^{1/2} \theta'\|^2-\frac{\mathbf{Tr}(A_i)}{\beta}-\lambda\right) d \pi_{\theta}(\theta') \right].
\]
We now use Lemma \ref{l1} to pull the expectation outside the influence function $\psi$. 
We decompose $A$ into $A = U D U^{\top}$, where $U U^{\top} = \mathrm{I}$ 
and $D = \mathrm{diag}(\lambda_1, \dots, \lambda_d)$
and we observe that $U^{\top} \theta'$ has the same distribution as $U^{\top} \theta 
+ W$, where $W \sim \mathcal{N}(0,\beta^{-1} \mathrm{I})$ so that 
\begin{multline*}
\mathbf{Var} \bigl[ \lVert A^{1/2} \theta' \rVert^2 \mathrm{d} \pi_{\theta}(\theta') \bigr] 
= \mathbf{Var} \biggl( \sum_{i=1}^d \bigl( (U^{\top} \theta)_i 
+ W_i \bigr)^2 \lambda_i \biggr) \\ = \sum_{i=1}^d \lambda_i^2 
\mathbf{Var} \bigl[ \bigl( (U^{\top} \theta)_i + W_i \bigr)^2 \bigr] 
= \sum_{i=1}^d \biggl( 
\frac{2}{\beta^2} + \frac{4}{\beta} (U^{\top} \theta))_i^2 \biggr) 
\lambda_i^2 \\ = \frac{2}{\beta^2} \mathbf{Tr}(A^2) + \frac{4}{\beta} \lVert A \theta 
\rVert^2.
\end{multline*}

\noindent
As a consequence we get
\begin{multline*}
r_{\lambda} (\theta)   \leq  \frac{1}{n} \sum_{i=1}^n  \Bigg[ \int \chi \left(\| A_i^{1/2} \theta'\|^2-\frac{\mathbf{Tr}(A_i)}{\beta}-\lambda\right) \, \mathrm{d} \pi_{\theta}(\theta')\\ 
+ \min \biggl\{ \log(4), \frac{1}{2\beta} \| A_i \theta\|^2 + \frac{\mathbf{Tr}(A_i^2)}{4\beta^2}\biggr\} \Bigg]
\end{multline*}
where the function $\chi$ is defined in equation \myeq{defchi}. 
We then apply equation \myeq{min_eq} with 
$m=  \| A \theta\|,$ $a=\log(4),$ 
$b= 1/(2\beta)$ and $c= \mathbf{Tr}(A^2)/(4\beta^2)$ to obtain 
\begin{multline*}
r_{\lambda} (\theta)   \leq  \frac{1}{n} \sum_{i=1}^n  \Bigg[ \int \chi \left(\| A^{1/2} \theta'\|^2-\frac{\mathbf{Tr}(A)}{\beta}-\lambda\right) \, \mathrm{d} \pi_{\theta}(\theta')\\ 
+ \int \min \biggl\{4 \log(2), \frac{1}{\beta} \| A \theta'\|^2 + \frac{\mathbf{Tr}(A^2)}{2\beta^2}\biggr\} \, \mathrm{d} \pi_{\theta}(\theta')\Bigg]
\end{multline*}
and we conclude, by Lemma \ref{lem1.5}, that
\begin{multline*}
r_{\lambda}(\theta) \leq \frac{1}{n} \sum_{i=1}^n \int 
\log \Biggl[ 1 + \| A_i^{1/2} \theta'\|^2  - \frac{\mathbf{Tr}(A_i)}{
\beta} - \lambda + \frac{1}{2} \biggl( \| A_i^{1/2} \theta'\|^2 - 
\frac{\mathbf{Tr}(A_i)}{\beta} - \lambda \biggr)^2 \\ 
+ \frac{c}{\beta} \biggl(\| A_i \theta'\|^2 + \frac{ \mathbf{Tr}(A_i^2)}{2 \beta} \biggr) \Biggr] \, \mathrm{d} \pi_{\theta}(\theta'),
\end{multline*}
where $\ds c = \frac{15}{8 \log(2)(\sqrt{2}-1)} \exp \biggl( \frac{1 + 2 \sqrt{2}}{2} \biggr).$
We then apply the PAC-Bayesian inequality (Proposition \ref{pac}) to 
\[
f(A_i, \theta') = \log \left[ 1 + t(A_i, \theta') + \frac{1}{2} t(A_i, \theta')^2+ \frac{c}{\beta} \biggl(\| A_i \theta'\|^2 + \frac{ \mathbf{Tr}(A_i^2)}{2 \beta} \biggr) \right]
\]
where  $\displaystyle t(A, \theta')=  \| A^{1/2} \theta'\|^2  - \frac{\mathbf{Tr}(A)}{\beta} - \lambda $ 
and we choose as posterior distributions the family of Gaussian perturbations $\pi_{\theta}.$
We obtain that, with probability at least $1-\epsilon$, for any $\theta \in \mathbb R^d$,
\begin{multline*}
r_{\lambda}(\theta) \leq \int \mathbb E 
\left[ t(A, \theta')+\frac{1}{2} t(A, \theta')^2 + \frac{c}{\beta} \left(\| 
A \theta'\|^2 + \frac{\mathbf{Tr}(A^2)}{2 \beta} \right) \right] \mathrm d \pi_{\theta} (\theta')
\\ \shoveright{+  \frac{\beta \|\theta\|^2}{2n} 
+ \frac{ \log(\epsilon^{-1})}{n} } 
\\ = 
\mathbb{E} \Biggl[ \lVert A^{1/2} \theta \rVert^2 - \lambda + 
\frac{1}{2} \Bigl( \lVert A^{1/2} \theta \rVert^2 - \lambda \Bigr)^2 
+ \frac{(c+2)\lVert A \theta \rVert^2}{\beta} + 
\frac{(2 + 3c) \mathbf{Tr}(A^2)}{2 \beta^2} \Biggr] \\ + 
\frac{\beta \lVert \theta \rVert^2}{2 n } + \frac{\log(\epsilon^{-1})}{n}.  
\end{multline*}
Using the Cauchy-Schwarz inequality, we remark that
\[
\mathbb E  \bigl[ \lVert A \theta \rVert^2 \bigr] \leq \mathbb E  \bigl[
\lVert A \rVert_{\infty} \lVert A^{1/2} \theta \rVert^2 \big] 
\leq \mathbb E  \bigl[ \lVert A \rVert_{\infty}^2 \bigr]^{1/2} 
\mathbb E  \bigl[ \lVert A^{1/2} \theta \rVert^4 \bigr]^{1/2}
\leq \mathbb E  \bigl[ \lVert A \rVert_{\infty}^2 \bigr]^{1/2} 
\kappa^{1/2} N(\theta),
\]
where $\kappa$ is defined in equation \myeq{skappa_A}.
Thus
\begin{multline*}
\label{eq4} 
r_{\lambda}(\theta) \leq 
\frac{\kappa}{2} \bigl[ N(\theta) - \lambda \bigr]^2 + \Biggl[
1 + (\kappa -1) \lambda + \frac{(2+c) \kappa^{1/2} 
\mathbb E  \bigl[ \lVert A \rVert_{\infty}^2 \bigr]^{1/2}
 }{\beta} 
\, \Biggr] \bigl[ N(\theta) - \lambda \bigr] \\ 
+ \frac{(\kappa-1) \lambda^2}{2} + 
\frac{(2+c) \kappa^{1/2} 
\mathbb E  \bigl[ \lVert A \rVert_{\infty}^2 \bigr]^{1/2}
\lambda}{\beta} + 
\frac{(2 + 3c) \mathbb E  \bigl[ \mathbf{Tr}(A^2) \bigr]}{2 \beta^2} + \frac{\beta \lVert \theta \rVert^2}{2n} 
+ \frac{\log(\epsilon^{-1})}{n}. 
\end{multline*}
This is the analogous of equation \myeq{eq4}. The end of the proof 
is similar to the case of the Gram matrix. 

\subsection{Proof of Lemma \ref{lem137}}\label{pfkappa}

\noindent
Replacing $X$ with $X - \mathbb E [X]$ we may assume during the proof that 
$\mathbb E [X] = 0$. It is true that
\begin{multline*}
\mathbb E  \bigl[ \lVert A^{1/2} \theta \rVert^4 \bigr] =
\mathbb E  \Biggl[ \biggl( \frac{1}{q(q-1)} \sum_{1 \leq j < k \leq q} \langle \theta, 
X^{(j)} - X^{(k)} \rangle^2 \biggr)^2 \Biggr]  
\\ = \frac{1}{q^2(q-1)^2} 
\sum_{\substack{
1 \leq j < k \leq q\\
1 \leq s < t \leq q
}}  
\mathbb E  \Bigl[ 
\langle \theta, X^{(j)} - X^{(k)} \rangle^2 
\langle \theta, X^{(s)} - X^{(t)} \rangle^2 
\Bigr].
\end{multline*}
Recalling the definition of the covariance, we have
\begin{multline*}
\mathbb E  \bigl[ \lVert A^{1/2} \theta \rVert^4 \bigr] =
\frac{1}{q^2(q-1)^2}  
\Biggl\{ \sum_{\substack{
1 \leq j < k \leq q\\
1 \leq s < t \leq q
}}  
  \mathbb E  \Bigl[ 
\langle \theta, X^{(j)} - X^{(k)} \rangle^2 \Bigr]
\mathbb E  \Bigl[ \langle \theta, X^{(s)} - X^{(t)} \rangle^2 
\Bigr] \\ 
+ 
\sum_{1 \leq j < k \leq q} 
\mathbb E  \Bigl[ \langle \theta, X^{(j)} - X^{(k)} 
\rangle^4 \Bigr] - \mathbb E  \Bigl[ \langle \theta, X^{(j)} - X^{(k)} \rangle^2 
\Bigr]^2 
\\ + 
\sum_{\substack{
1 \leq j < k \leq q\\
1 \leq s < t \leq q\\
\lvert \{ j, k \} \cap \{s, t \} \rvert = 1
}} 
\biggl( \mathbb E  \Bigl[ 
\langle \theta, X^{(j)} - X^{(k)} \rangle^2 
\langle \theta, X^{(s)} - X^{(t)} \rangle^2 
\Bigr] 
\\ -  \mathbb E  \Bigl[ 
\langle \theta, X^{(j)} - X^{(k)} \rangle^2 \Bigr] 
\mathbb E  \Bigl[ \langle \theta, X^{(s)} - X^{(t)} \rangle^2 
\Bigr] \biggr) \Biggr\} 
\\ = 
\frac{1}{4} \mathbb E  \Bigl[ \langle \theta, X^{(2)} - X^{(1)} \rangle^2 
\Bigr]^2 + \frac{1}{2q(q-1)} 
\mathbb E  \Bigl[ \langle \theta, X^{(2)} - X^{(1)} 
\rangle^4 \Bigr] - \mathbb E  \Bigl[ \langle \theta, X^{(2)} - X^{(1)} \rangle^2 
\Bigr]^2 
\\ + \frac{q-2}{q(q-1)}
\biggl( \mathbb E  \Bigl[ 
\langle \theta, X^{(1)} - X^{(2)} \rangle^2 
\langle \theta, X^{(1)} - X^{(3)} \rangle^2 
\Bigr]
-  \mathbb E  \Bigl[
\langle \theta, X^{(1)} - X^{(2)} \rangle^2 \Bigr]^2 \biggr).
\end{multline*}
Define $W_j = \langle \theta, X^{(j)} \rangle$ and observe that
\begin{align*}
\mathbb E  \bigl[ (W_1 - W_2)^2 \bigr]^2 & = 4 N(\theta)^2, \\ 
\mathbb E  \bigl[ (W_1 - W_2)^4 \bigr] &  
= \mathbb E  \bigl[ W_1^4 \bigr] 
+ 6 \mathbb E  \bigl[ W_1^2 \bigr] \mathbb E  \bigl[  W_2^2 \bigr]
+ \mathbb E  \bigl[ W_2^4 \bigr] \\ 
& = 2 \mathbb E  \bigl[ W_1^4 \bigr] 
+ 6 \mathbb E  \bigl[ W_1^2 \bigr]^2 
\leq (2 \kappa + 6) N(\theta)^2, \\
\mathbb E  \Bigl[(W_1 - W_2)^2 (W_1 - W_3)^2 \Bigr] & = 
\mathbb E  \bigl[ W_1^4 \bigr] + 3 \mathbb E  \bigl[ W_2^2]^2 
\leq (\kappa + 3) N(\theta)^2.
\end{align*}
Therefore
\[ 
\mathbb E  \Bigl[ \lVert A^{1/2} \theta \rVert^4 \Bigr] \leq 
\biggl( 1 + \frac{(q-2)(\kappa-1)}{q(q-1)} + \frac{(\kappa +1)}{q(q-1)}  
\biggr) N(\theta)^2 = \biggl( 1 + \frac{\tau_q(\kappa)}{q} \biggr) 
N(\theta)^2,
\] 
and hence $\kappa' \leq 1 + \tau_q(\kappa)/q$, since $
\mathbb E  \bigl[ \lVert A^{1/2} \theta \rVert^2 \bigr] = N(\theta)$.

\subsection{Proof of Lemma \ref{lem2}}\label{pflem2_cov}

\noindent
Replacing $X$ with $X - \mathbb E [X]$ we may assume that $\mathbb E [X] = 0$. 
Recall that 
\[
\mathbb E[ \| X \|^4 ] \leq \kappa \mathbb E[ \| X \|^2 ]^2 = \kappa \mathbf{Tr}(\Sigma)^2
\]
and $\mathbb E [  \langle X^{(1)}, X^{(2)} \rangle^2] = \mathbf{Tr}(\Sigma^2).$
We observe that
\[ 
\mathbb E  \bigl[ \lVert A \theta \rVert^2 \bigr] = 
\mathbb E  \Biggl[  \frac{1}{q^2(q-1)^2} \sum_{\substack{1 \leq j < k \leq q
\\ 1 \leq s < t \leq q
}} 
\langle \theta, X^{(j)} - X^{(k)} \rangle \langle X^{(j)} - X^{(k)}, 
X^{(s)} - X^{(t)} \rangle \langle X^{(s)} - X^{(t)}, \theta \rangle  
 \Biggr]
\] 
and 
\begin{align*}
\mathbb E  \Bigl[ 
\langle \theta, X^{(1)} - X^{(2)} \rangle & \langle X^{(1)} - X^{(2)}, 
 X^{(3)} - X^{(4)} \rangle \langle X^{(3)} - X^{(4)}, \theta \rangle  \Bigr] 
\\ & = 4 \mathbb E  \Bigl[ \langle \theta, X^{(1)} \rangle \langle X^{(1)}, X^{(2)} 
\rangle \langle X^{(2)}, \theta \rangle \Bigr] = 4 \lVert \Sigma \theta \rVert^2
\leq 4 \lVert \Sigma \rVert_{\infty} N(\theta),
\\ 
\mathbb E  \Bigl[ \langle \theta, X^{(1)} - X^{(2)} \rangle & \langle X^{(1)} - X^{(2)}, 
X^{(1)} - X^{(3)} \rangle \langle X^{(1)} - X^{(3)}, \theta \rangle  \Bigr] 
\\ & = \mathbb E  \Bigl[ \langle \theta, X^{(1)} \rangle^2 \lVert X^{(1)} \rVert^2 
\Bigr] + 3 
\mathbb E  \Bigl[ \langle \theta, X^{(1)} \rangle \langle X^{(1)}, X^{(2)} 
\rangle \langle X^{(2)}, \theta \rangle \Bigr] \\ 
& \leq \kappa \mathbf{Tr}(\Sigma) N(\theta) + 3 \lVert \Sigma \rVert_{\infty} N(\theta), \\ 
\mathbb E  \Bigl[ \langle \theta, X^{(1)} - X^{(2)} \rangle & \langle X^{(1)} - X^{(2)}, 
 X^{(1)} - X^{(2)} \rangle \langle X^{(1)} - X^{(2)}, \theta \rangle  \Bigr] 
\\ & = 2 \mathbb E  \Bigl[ \langle \theta, X^{(1)} \rangle^2 \lVert X^{(1)} \rVert^2 
\Bigr] + 2 \mathbb E  \Bigl[ \langle \theta, X^{(1)} \rangle^2 \Bigr]
\mathbb E  \Bigl[ \lVert X^{(1)} \rVert^2 \Bigr] \\ & \quad + 4 
\mathbb E  \Bigl[ \langle \theta, X^{(1)} \rangle \langle X^{(1)}, X^{(2)} 
\rangle \langle X^{(2)}, \theta \rangle \Bigr] 
\\ & \leq 2(\kappa+1) \mathbf{Tr}(\Sigma) N(\theta) + 4 \lVert \Sigma \rVert_{\infty} N(\theta),
\end{align*}
which proves the first inequality.
In the same way, 
\[ 
\mathbb E  \Bigl[ \mathbf{Tr}(A^2) \Bigr] = \mathbb E  \Biggl[ 
\frac{1}{q^2(q-1)^2} \sum_{\substack{1 \leq j < k \leq q
\\ 1 \leq s < t \leq q}} 
\langle X^{(j)} - X^{(k)}, X^{(s)} - X^{(t)} \rangle^2 
\Biggr], 
\]
and 
\begin{align*}
\mathbb E  \Bigl[ \langle X^{(1)} - X^{(2)}, X^{(3)} - X^{(4)} \rangle^2 \Big] 
& = 4 \mathbb E  \Bigl[ \langle X^{(1)}, X^{(2)} \rangle^2 \Bigr] 
\\ & = 4 \mathbf{Tr} \bigl( \Sigma^2 \bigr), \\ 
\mathbb E  \Bigl[ \langle X^{(1)} - X^{(2)}, X^{(1)} - X^{(3)} \rangle^2 \Bigr] 
& = \mathbb E  \Bigl[ \lVert X^{(1)} \rVert^4 \Bigr] + 3 
\mathbb E  \Bigl[ \langle X^{(1)}, X^{(2)} \rangle^2 \Bigr] 
\\ & \leq 
\kappa \mathbf{Tr}(\Sigma)^2 + 3 \mathbf{Tr} \bigl( \Sigma^2 \bigr), \\
\mathbb E  \Bigl[ \langle X^{(1)} - X^{(2)}, X^{(1)} - X^{(2)} \rangle^2 \Bigr] 
& = 2 \mathbb E  \Bigl[ \lVert X^{(1)} \rVert^4 \Bigr] 
+ 2 \mathbb E  \Bigl[ \lVert X^{(1)} \rVert^2 
\Bigr]^2 + 4 \mathbb E  \Bigl[ \langle X^{(1)}, X^{(2)} \rangle^2 \Bigr] \\ 
& \leq 2 (\kappa + 1) \mathbf{Tr}(\Sigma)^2 + 4 \mathbf{Tr} \bigl( \Sigma^2 \bigr),
\end{align*}
which concludes the proof.

\end{document}